\documentclass[11pt,reqno]{article}
\usepackage[margin=1in]{geometry}
\usepackage{dsfont, amssymb,amsmath,amscd,latexsym, amsthm, amsxtra,amsfonts}
\usepackage{lineno}
\usepackage[all]{xy}
\usepackage[active]{srcltx}
\usepackage{tikz}
\usepackage{bbm}
\usepackage{enumerate}
\usepackage{mathrsfs}
\usepackage{graphicx}
\usepackage{subcaption}
\usepackage{comment}
\usepackage{mathtools}
\usepackage{cases}
\usepackage{tcolorbox}
\tcbuselibrary{most}

\usetikzlibrary{calc,arrows}
\usepackage{verbatim}
\usepackage{color}
\usepackage{epstopdf}
\usepackage[affil-it]{authblk}
\usepackage{bm}
\usepackage[title]{appendix}

\newtheorem{theorem}{Theorem}[section]

\newtheorem{condition}[theorem]{Condition}

\newtheorem{lemma}[theorem]{Lemma}

\newtheorem{proposition}[theorem]{Proposition}
\newtheorem{remark}[theorem]{Remark}
\numberwithin{equation}{section}

\newtheorem{definition}{Definition}
\newtheorem{assumption}{Assumption}
\newtheorem{hypothesis}{Hypothesis}
\newtheorem{problem}{Problem}

\newcommand{\md}{\mathrm{d}}

\newcommand{\mR}{\mathbb{R}}

\newcommand{\mE}{\mathbb{E}}

\newcommand{\mP}{\mathbb{P}}

\newcommand{\ind}{\mathbbm{1}}

\renewcommand{\epsilon}{\varepsilon}

\newcommand{\B}{\mathcal{B}}

 \usepackage[pdfstartview=FitH, bookmarksnumbered=true,bookmarksopen=true, colorlinks=true, pdfborder={0 0 1}, citecolor=blue, linkcolor=blue,urlcolor=blue]{hyperref}
\usepackage{graphics}
\graphicspath{{figures/}}

\title{On Convergence Rates of General $N$-Player Stackelberg Games to their Mean Field Limits}

\usepackage{authblk}

\author[a]{Alain Bensoussan\footnote{E-mail:axb046100@utdallas.edu}}

\author[b]{Ziyu Huang\footnote{E-mail: zyhuang19@fudan.edu.cn}}
\author[c]{Sheng Wang\footnote{E-mail: sheng-wa15@tsinghua.org.cn}}
\author[b]{Sheung Chi Phillip Yam\footnote{E-mail: scpyam@sta.cuhk.edu.hk}}

\affil[a]{\small \it International Center for Decision and Risk Analysis, Naveen Jindal School of Management, University of Texas at Dallas, Dallas, Texas, USA.}
\affil[b]{\small \it Department of Statistics and Data Science, The Chinese University of Hong Kong, Shatin, N.T., Hong Kong SAR}
\affil[c]{\small \it Department of Statistics and Actuarial Science, School of Computing and Data Science, The University of Hong Kong, Pokfulam Road, Hong Kong}

\newcommand{\F}{\mathcal{F}}

\renewcommand{\P}{\mathbb{P}}

\newcommand{\RR}{\mathbb{R}}

\newcommand{\f}{\mathscr{F}}
\newcommand{\lr}{\mathcal{L}}
\newcommand{\sr}{\mathcal{S}}
\newcommand{\e}{\mathbb{E}}
\newcommand{\pr}{\mathcal{P}}

\newcommand{\brn}{{\mathbb{R}^n}}

\newcommand{\brd}{{\mathbb{R}^d}}
\newcommand{\de}{\Delta}

\allowdisplaybreaks[2]

\begin{document}
	\maketitle

\begin{abstract}

In this article, we establish precise convergence rates of a general class of $N$-Player Stackelberg games to their mean field limits, which allows the response time delay of information, empirical distribution based interactions, and the control-dependent diffusion coefficients.  All these features makes our problem nonstandard, barely been touched in the literature, and they complicate the analysis and therefore reduce the convergence rate. We first justify the same convergence rate for both the followers and the leader.
Specifically, for the most general case, the convergence rate is shown to be $\mathcal{O}\left(N^{-\frac{2(q-2)}{n_1(3q-4)}}\right)$ when $n_1>4$ where $n_1$ is the dimension of the follower's state, and $q$ is the order of the integration of the initial;  and this rate has yet been shown in the literature, to the best of our knowledge. Moreover, by classifying cases according to the state dimension $n_1$, the nature of the delay, and the assumptions of the coefficients, we provide several subcases where faster convergence rates can be obtained; for instance the $\mathcal{O}\left(N^{-\frac{2}{3n_1}}\right)$-convergence when the diffusion coefficients are independent of control variable. Our result extends the standard $o(1)$-convergence result for the mean field Stackelberg games in the literature, together with the $\mathcal{O}(N^{-\frac{1}{n_1+4}})$-convergence for the mean field games with major and minor players. We also discuss the special case where our coefficients are linear in distribution argument while nonlinear in state and control arguments, and we establish an $\mathcal{O}(1/\sqrt{N})$ convergence rate, which extends the linear quadratic cases in the literature. \\

\noindent{\textbf{Keywords:}} $N$-player Stackelberg games; Mean field limit; 
Response time delay; 
Control-dependent diffusion; Mixture-convexity of Wasserstein metric; {Immersion of filtration}

\noindent {\bf Mathematics Subject Classification (2020):} 91A65; 91A06; 91A15

\end{abstract}

\section{Introduction}
Game problems involving a large number of players have been extensively studied in recent decades, see \cite{MR1783877,MR2449100,MR2629529} for instance. In symmetric, non-cooperative stochastic differential games with $N$ interacting players, each player solves a control problem in which both the cost functional and the dynamics depend not only on their own state but also on the states of the other players.
Solving for an exact Nash equilibrium of an $N$-player game is impractical when $N$ is large due to the curse of dimensionality. Instead, one can take the limit as $N\to\infty$ and consider the limiting problem to give an approximate Nash equilibrium.
This limiting problem is the well-known as mean field game (MFG), which were first introduced by Lasry and Lions in a series of articles \cite{JM1,JM2} and also independently by Huang, Caines and Malham\'e \cite{huang2003individual,HM1}.
For comprehensive studies on MFGs, 
we refer to \cite{AB_book,GDA,GM} for the HJB-FP approach, and to \cite{CP1,GW} for master equation approach, and to \cite{SA1,AB11,carmona2018probabilistic,Moon-Basar} for the probabilistic approach. Convergence results from the $N$-player game to the MFG limit can be found, for example, in \cite{CDLL,MR4522347,Carmona-Zhu,guo2022optimization,han2024gradient,Huang-Tang,Nourian-Caines}. 

Heinrich von Stackelberg \cite{stackelberg1934marktform} introduced a hierarchical game equilibrium  notion in 1934 for markets with a leader and a follower, where in a two-person nonzero-sum game the follower chooses an optimal strategy in response to the leader’s policy, and the leader, anticipating this reaction, announces policies that optimize his own targeted planning. This Stackelberg equilibrium notion  were then extended to more general settings, see \cite{bacsar2010differential,bensoussan2013linear,MR965048,simaan1973stackelberg} for instance. One of an important kind of Stackelberg games is to consider a system consisting of one leader and $N$ followers, where the individuals can also gather information through the interactions with the community. Given any action of the leader and the information of the community, each follower picks up his own optimal strategy. A Stackelberg Nash equilibrium is a set of strategies, with the strategies for the $N$ followers constituting a Nash equilibrium (which can be viewed as a function of the leader's strategy), and the strategy for the leader is the optimal. When $N$ goes to the infinity, the limiting problem is called the mean field Stackelberg game. The mean field Stackelberg game can be viewed as an optimal control problem nested with fixed point problems, which makes it more complicated than traditional Stackelberg game or MFG.  Solving it typically consists of two steps: (1) given the control $v_0$ of the leader, solve an MFG parameterized by $v_0$, whose fixed point is denoted by $z[v_0]$; (2) search for an optimal $v_0$. For the study on the mean field Stackelberg game, we refer to \cite{bensoussan2016mean,Bensoussan-Chau-Yam,Lin-Jiang-Zhang} for the linear quadratic setting for instance. 
The mean field Stackelberg game is different from the MFG with one major and many minor players studied in \cite{Carmona-Zhu,MR2599921,Huang-Tang,Nourian-Caines}. In a MFG with major and minor players, although the major player strongly influences the minors, all players (including the major player) determine their optimal strategies simultaneously, and this limitation narrows its potential applicability in economics and finance, since it is evident that most governors, while not all-powerful, possess some authority to override and steer the future course of the entire community.
Motivated by the latter consideration, \cite{bensoussan2016mean} proposed a substantially different general framework, the MFGs in the presence of a ``dominating player" (also called the ``leader"). Compared with the community of minor players in the MFG, the nature of the dominating is clear in the sense that changes in the behavior of this dominating player would immediately and directly affect both the perception of all followers and the aggregated public information through the evolution of the mean field term, being summarized community data analytics of the whole population. Mathematically, the optimal controls of the followers and the mean field term are both functionals of the dynamics of the leader. That is, given the mean field term and the policies set by the dominating player, we first solve the optimal control for the representative player, and then, by regarding the optimized mean field term as a functional of the dominating player, we next proceed to solve for the dominating player’s optimal control. In summary, the objective is to  approximate the hierarchical equilibrium notion originally introduced in \cite{stackelberg1934marktform} when the number of the followers is large.

In this article, we aim to establish the convergence rate of $N$-Player Stackelberg games to their mean field limits. We consider a general class of mean field Stackelberg games, allowing for the following features: (i)  the information received by the followers have various magnitudes of {\it response time delay} (delay for short); (ii)  the followers interact with one another through the empirical distribution (rather than merely the mean of the states); (iii) the drift and diffusion coefficients of the state processes for both the leader and the followers may depend on the state, the control, and the distribution. 
For (i), in practice, due to heterogeneous technological advancements among agents, it is natural to assume that individuals respond to policy changes with varying magnitudes of delay. Each follower is fully aware of their own delay time but has an incentive to conceal this information from others. Hence, we model the exact delay time of any individual as a hidden random variable $\Delta$, unknown to all other followers including the leader. These hidden random variables $\{\Delta_j,\ 1\le j\le N\}$ ($j$ stands for the $j$-th follower) share a common distribution $\pi_\Delta$, which is known to both the leader and the followers. The introduction of delay complicates the analysis of the convergence rate, since the empirical distribution involves followers with different delays. To address this, we utilize certain estimates of the Wasserstein metric between probability measures obtained via convex  combinations (mixture distribution); see Lemmas~\ref{lma:convex:wa} and \ref{lma:pq:mu}. On deriving the convergence rate, we first establish results for discrete $\Delta$, and by then extend them to the general case via a discretization method. We refer  to \cite{MR3691722,Bensoussan-Chau-Yam} for the earliest studies on stochastic Stackelberg differential games with delay, but with constant diffusion or linear-quadratic settings. 
For (ii), we generalize the model of \cite{Bensoussan-Chau-Yam} by replacing the variable of the mean of the followers’ states by their whole empirical distribution functional. This transforms the problem from estimating the metric between variables in $\mathcal{L}^2$ spaces to estimating the metric between probability measures in Wasserstein spaces. Consequently, our analysis relies heavily on convergence results for the Wasserstein metric of empirical distributions to the corresponding conditional distribution laws; see Section~\ref{sec:convergence} for details. 
For (iii), allowing general diffusion coefficients introduces further challenges. In particular, dependence of the diffusion coefficient on the control variable reduces the convergence rate; see Remark~\ref{prop:holdercontinuity}. Specifically, the rate in the control-dependent case 
is the root of that under the control-independent case.

Our main result shows that the solution of the limiting mean field Stackelberg game provides an approximate Stackelberg Nash equilibrium.  The convergence rate is  $\mathcal{O}\left(N^{\frac{-(q-2)}{2(3q-4)}}\right)$ if $n_1<4$ ($n_1$ is the dimension of the follower's state, and $q>4$ is a constant for the assumptions on the integrability of the initial condition in Assumption~\ref{assumption:initial} (i)), or it is $\mathcal{O}\left(\left(\frac{\log(N)}{\sqrt{N}}\right)^{\frac{q-2}{3q-4}}\right)$ if $n_1=4$, or it is $\mathcal{O}\left(N^{\frac{-2(q-2)}{n_1(3q-4)}}\right)$ if $n_1>4$. This result has two components: (i) an approximate Nash equilibrium for the followers given any arbitrary strategy of the leader $v_0$, and (ii) the approximate optimality for the leader. Together, these two components constitute the approximate Stackelberg Nash equilibrium. 
As a particular case, when the diffusion coefficient of the leader is independent of the control variable, then the convergence rate increase to be $\mathcal{O}\left(N^{-1/6}\right)$ if $n_1<4$, or $\mathcal{O}\left(\left(\frac{\log(N)}{\sqrt{N}}\right)^{\frac{1}{3}}\right)$ if $n_1=4$, or $\mathcal{O}\left(N^{-2/3n_1}\right)$ if $n_1>4$. As a another case, when the delay $\Delta$ is discretely distributed, the convergence rate increases to   $\mathcal{O}\left(N^{-1/4}\right)$ if $n_1<4$, or to $\mathcal{O}\left(N^{-1/4}\log(N)\right)$ if $n_1=4$, or to $\mathcal{O}\left(N^{-1/n_1}\right)$ if $n_1>4$.  
We emphasize that the condition $q>4$ is required only for obtaining explicit convergence rates, but not for establishing convergence itself. Indeed, convergence holds already under the weaker assumption $q=2$; see \cite{Bensoussan-Chau-Yam} for a proof in the case of constant diffusion. The requirement $q>4$ stems from the use of Lemma~\ref{lemma:wasserstien} (also see \cite[Theorem 5.8]{carmona2018probabilistic}), which provides explicit bounds on the convergence rate of empirical measures for independent and identically distributed random variables. 
 Regarding the comparison with existing results, in \cite[Theorem 7.1]{Carmona-Zhu} the authors study conditional propagation of chaos and obtain an $ \mathcal{O}(N^{-1/(n_1+4)})$ convergence rate for a system of $(N+1)$ interacting particles and the associated conditional McKean–Vlasov stochastic differential equations (SDEs). In \cite{huang2025nonlinear}, the authors provide an $o(1)$ convergence for the nonlinear mean field Stackelberg game with constant diffusion. In \cite[Theorem 3]{dayanikli2024machine}, the authors also provided an $ \mathcal{O}(N^{-1/n_1})$ convergence rate (when $n_1>4$) for $N$-player Stackelberg mean field game via a penalization approach without time delay while diffusion for the leader is a constant. In our setting, the presence of time delay together with control dependence significantly complicates the analysis and reduces the rate of convergence.  
 In \cite[Theorem 7.2]{Nourian-Caines}, an $\mathcal{O}(1/\sqrt{N})$ convergence rate is obtained for MFGs with one major and many minor players, under the assumptions that the minors’ state processes are independent of the major’s state and that the drift, diffusion, and cost coefficients depend linearly on the distribution; that is, for a function $\phi:\mathbb{R}^n\times \mathcal{P}_2(\mathbb{R}^n)\times \mathbb{R}^d\to\mathbb{R}^n$,
\begin{align}\label{intr_1}
    \phi(x,z,v)=\int_\brn \overline{\phi}(x,y,v)z(dy).
\end{align} 
The similar convergence rate is also obtained in \cite{MR3857463} for the linear quadratic mean field Stackelberg games. In Section~\ref{sec:example}, we also discuss the special case where our coefficients take a form similar to \eqref{intr_1}, and we establish an $\mathcal{O}(1/\sqrt{N})$ convergence rate (independent of the dimension $n_1$); to this end, we require the coefficients to be separable in the distribution variable $z$ and the other arguments (see Assumption~\ref{assumption:exp}). This separability is necessary because our framework involves time delays, and the auxiliary processes in the proof may be adapted to different filtrations.  In contract to \cite[Theorem 7.2]{Nourian-Caines}, we additionally allow the followers’ state processes to depend on the leader’s state, which is natural in a Stackelberg game since the leader and the followers act sequentially; and our setting can also include the linear cases.

 For studies on the convergence rate in mean field games and mean field type control problems, the proofs are mainly based on results of the rate of the convergence in Wasserstein distance of the empirical measure (see Lemma~\ref{lemma:wasserstien} for instance) and regularity results on SDEs (see Lemma~\ref{prop:moment} for instance). Compared with existing results in mean field theory, our method is based on not only the abovementioned approaches, but also the usage of the mixture-convexity of Wasserstein metric (see Lemma~\ref{lma:convex:wa} and the newly proposed Lemma~\ref{lma:pq:mu}), and the extension of immersion of filtration (see Lemma~\ref{lma:property:probability}). These techniques are required here for the reason that we include the time delay $\de$, which is a random variable and complicates the filtrations, therefore, the problem is totally unconventional; and these techniques are important in proving the crucial estimate, i.e., the Wasserstein distance of the empirical measure involving time delay and the conditional distribution integrated with respect to $\de$ (see the proofs of Lemmas~\ref{prop:convergence:discrete} and \ref{prop:general} for instance). Moreover, we also need to give the regularity of SDEs with respect to the time delay parameter (see Lemma~\ref{prop:holdercontinuity}).

The remainder of the article is organized as follows. Section~\ref{sec:formulation} introduces the formulation of the $N$-player Stackelberg game and its limiting counterpart. Section~\ref{sec:assumption&SDE} presents the standing assumptions and establishes preliminary estimates for the controlled SDEs. In Section~\ref{sec:convergence}, we derive estimates for the Wasserstein metric between the empirical distribution and the conditional distribution. Section~\ref{sec:main} contains the main results, showing that the solution of the limiting problem yields an approximate Stackelberg Nash equilibrium for the $N$-player game, together with the corresponding convergence rate.  Section \ref{sec:conclution} concludes the article with  a discussion of future research directions.
Finally, some technical proofs for Sections~\ref{sec:formulation}, \ref{sec:assumption&SDE} and \ref{sec:convergence} are deferred to Appendices~\ref{app:sec:2}, \ref{app:sec:3} and \ref{pf:prop:example}, respectively.

\section{Preliminaries and Problem Formulation}\label{sec:formulation}

\subsection{Wasserstein space and some properties }

For $q\geq 1$, let \( \mathcal{P}_q(\mathbb{R}^{n_1}) \) be the space of probability measures equipped with the $q$-Wasserstein metric, \( W_q(\cdot, \cdot) \) such that for any \( \mu \) and \( \nu \) in \( \mathcal{P}_q(\mathbb{R}^{n_1}) \),		
		\begin{align}\label{eq:def:w}
		W_q(\nu_1, \nu_2) := \inf_{\gamma \in \Gamma(\nu_1, \nu_2)} \left( \int_{\mathbb{R}^{n_1} \times \mathbb{R}^{n_1}} |x - y|^q \, d\gamma(x, y) \right)^{\!\!1/q},
		\end{align}
		where the infimum is taken over the family \( \Gamma(\nu_1, \nu_2) \), the collection of all joint measures with respective marginals \( \nu_1 \) and \( \nu_2 \).  
		For any probability measure \(\mu \in \mathcal{P}_q(\mathbb{R}^{n_1})\), we write 
		\[
		\mathcal{M}_q(\mu) = \left( \int_{\mathbb{R}^{n_1}} |x|^q  d\mu(x) \right)^{\!\!1/q}.
		\]
Let $(\Omega,\f,\mathbb{P})$ be a complete filtered probability space, $\mE\left[\cdot\right]$ denotes the expectation, and $\mE[\cdot|\mathscr{C}]$  denotes the conditional expectation given the $\sigma$-algebra $\mathscr{C}\subset\mathscr{F}$.

The following lemma establishes explicit bounds on the convergence rate of the empirical measure for independent and identically distributed (iid. for short) random variables, as given in \cite[Theorem 5.8]{carmona2018probabilistic}.

\begin{lemma}\label{lemma:wasserstien}
	Let $\left(X_k\right)_{k \geqslant 1}$ be a sequence of i.i.d. random variables in $\mathbb{R}^{n_1}$ with a common distribution $\mu \in \mathcal{P}_q\left(\mathbb{R}^{n_1}\right)$ for some $q>4$
    , then, for each dimension $n_1 \geqslant 1$, there exists an universal constant $C=C\left(n_1, q\right)$ such that, for all $N \geqslant 2$ :
	\begin{align*}
		\mathbb{E}\left[W_2\left(\frac1N\sum_{k=1}^N\delta_{X_k}, \mu\right)^2\right] 
		\leqslant C(n_1, q) \mathcal{M}_q^2(\mu) f(N),
	\end{align*}
    where $\delta_x$ is the Dirac measure with a unit mass at $x$, and
\begin{align}\label{def:function:f}
	f(N)=\begin{cases}N^{-1 / 2}, & \text { if } n_1<4; \\
		N^{-1 / 2} \log (N), & \text { if } n_1=4; \\
		N^{-2 / n_1}, & \text { if } n_1>4.\end{cases}
\end{align}
\end{lemma}

We give the following lemma concerning estimates of the Wasserstein metric between two probability measures obtained via convex combinations, with the proofs provided in Appendix~\ref{pf:lma:convex:wa}. 

\begin{lemma}\label{lma:convex:wa}
Let $A\subset \mR$ be a Borel set and $\pi$ a probability measure on $A$. For each \( s \in A \), let \( \mu_s, \nu_s \in \mathcal{P}_2(\mathbb{R}^{n_1}) \). Assume that \( \mu_{\cdot} \) and \( \nu_{\cdot} \) are measurable with respect to \( s \in A \), and that $\sup\limits_{s\in A}\mathcal{M}_2(\mu_s)+\sup\limits_{s\in A}\mathcal{M}_2(\nu_s)<\infty$,  then, the following inequality holds:
	\begin{align}\label{lma:convex:wa:1}
		W_2^2\left(\int_A\mu_s\md \pi(s), \int_A\nu_s\md \pi(s) \right)\leq \int_AW_2^2(\mu_s,\nu_s)\md \pi(s).
	\end{align}  
In particular, for any 	$\lambda_k\geq0$ with $\sum_{k=1}^n\lambda_k=1$ and $\nu_k,\mu_k\in\mathcal{P}_2(\mR^{n_1})$, we have
	\begin{align}
		\label{lma:convex:wa:2}W_2^2\left(\sum_{k=1}^{n}\lambda_k\mu_k,\sum_{k=1}^{n}\lambda_k\nu_k\right)\leq \sum_{k=1}^{n}\lambda_k W_2^2(\mu_k,\nu_k).
	\end{align}
	Furthermore, if $\nu_k=\nu $ for any $k$, then 
	\begin{align}
		\label{lma:convex:wa:3}W_2^2\left(\sum_{k=1}^{n}\lambda_k\mu_k,\nu\right)\leq \sum_{k=1}^{n}\lambda_k W_2^2(\mu_k,\nu).
	\end{align}
\end{lemma}

Lemma \ref{lma:convex:wa} addresses the case where the coefficients are identical but the probability measures differ, and we also give the following result, which considers the case where the coefficients differ but the corresponding probability measures are identical. The proof of Lemma~\ref{lma:pq:mu} is given in Appendix~\ref{pf:lma:pq:mu}. 

\begin{lemma}\label{lma:pq:mu}
	Suppose that $\mu_k\in \mathcal{P}_2(\mR^{n_1})$ for $k=1,2,\cdots,n$, and suppose that $p_k\geq0,q_k\geq0$ and $\sum_{k=1}^np_k=\sum_{k=1}^nq_k=1$. Then, the following inequality holds:
	\begin{align*}
		W^2_2\left(\sum_{k=1}^np_k\mu_k,\sum_{k=1}^nq_k\mu_k\right)\leq\sum_{h=1}^n\sum_{l=1}^n\hat{\pi}_{hl}W^2_2(\mu_{h},\mu_{l}),
        \end{align*}
	where $\hat{\pi}_{hh}=\min\{p_h,q_h\}$ for $h=1,\cdots,n$; and for $h\neq l$, $\hat{\pi}_{hl}$ is given by
	\begin{equation}\label{def:pi_ij}
		\hat{\pi}_{hl}=\begin{cases}
			\dfrac{(p_h-q_h)(q_h-p_l)}{\sum_{k\in A^{\complement}}(p_k-q_k)}, \quad &h\in A^{\complement}\,\, \text{and}\,\,l\in A,\\
            0,\quad& h\in A\,\,\text{or}\,\,l\in A^{\complement},\\
		\end{cases}
	\end{equation}
	with $A:=\{h:p_h\leq q_h\}$.
    Moreover, we have
    \begin{align}\label{eq:p_ii}
		\sum_{h=1}^n\hat{\pi}_{hh}=1-\frac{1}{2}\sum_{h=1}^n|p_h-q_h|.
	\end{align}
\end{lemma}

Next, we recall some theory on regular conditional distributions from \cite[Chapter 10]{dudley2018real}. Let $X: (\Omega, \f, \P) \to (E, \B)$ be measurable and $\mathscr{C}$ a sub-$\sigma$-algebra of $\f$, and let $\P|_{\mathscr{C}}$ be the restriction of $\mathbb{P}$ to $\mathscr{C}$. A function $\mathbb{P}_X^{\mathscr{C}}(\omega, B): \Omega \times \B \to [0,1]$ is a {regular conditional distribution} of $X$ given $\mathscr{C}$ if: (i) $\forall \omega \in \Omega$, $\mathbb{P}_X^{\mathscr{C}}(\omega, \cdot)$ is a probability measure on $\B$; (ii) $\forall B \in \B$, $\mathbb{P}_X^{\mathscr{C}}(\cdot, B)$ is $\mathscr{C}$-measurable and $\mathbb{P}_X^{\mathscr{C}}(\cdot, B) = \P(X^{-1}(B) \mid \mathscr{C})(\cdot)$, $\P|_{\mathscr{C}}$-a.s..  If a regular conditional distribution exists, then for any measurable  $f: E \to \RR$ with $\mE\left[|f(X)|\right] < \infty$, we have
\[
	\mE\left[f(X) \mid \mathscr{C}\right](\omega) = \int_E f(x) \mathbb{P}_X^{\mathscr{C}}(\omega, dx) \quad \text{$\P|_{\mathscr{C}}$-a.s.}.
\]
Particularly, if $(E, \B)$ is Polish, then for any sub-$\sigma$-algebra $\mathscr{C} \subset \f$, there exists a regular conditional distribution of $X$ given $\mathscr{C}$.

Next, we introduce a key property of regular conditional distributions, which plays a crucial role in deriving the convergence rate (see \eqref{eq:pro:conditional}). It is worth noting that this property—especially \eqref{eq:property:X}—is closely related to the  ``($\mathscr{H}$)-hypothesis" in \cite{bremaud1978changes} or the concept ``immersion of filtration" in \cite{aksamit2017enlargement,carmona2018probabilistic}. However, our setting requires a  stronger result formulated in terms of regular conditional distributions rather than merely conditional expectations, namely \eqref{eq:property:regular}, whose proof is provided in Appendix~\ref{pf:lma:property:probability}.
\begin{lemma}\label{lma:property:probability}
    Let $X: (\Omega, \f, \P) \to (\mR^{n_1}, \B(\mR^{n_1}))$ be measurable and integrable. Assume that $\mathscr{C}_1\subset \mathscr{C}_2$ and $\mathscr{G}$ are three sub-$\sigma$-algebras of $\f$ such that:(i) $\mathscr{C}_2$ and $\mathscr{G}$ are independent;(ii) $X$ is $\mathscr{C}_1\vee \mathscr{G}$-measurable.  Then we have
    \begin{align}\label{eq:property:regular}
        \mP_X^{\mathscr{C}_1}=\mP_X^{\mathscr{C}_2}, \quad \mP-a.s..
    \end{align}
\end{lemma}

Under the same assumptions as in Lemma \ref{lma:property:probability}, we have, in particular,
\begin{align}\label{eq:property:X}
        \mE\left[X|\mathscr{C}_1\right]= \mE\left[X|\mathscr{C}_2\right], \quad \mP-a.s..
    \end{align}

We also introduce the following notations. For $q\geq 1$, for any $\xi\in \mathcal{L}^q(\Omega,\mathscr{F},\mathbb{P})$, we denote by $\|\xi\|_{q}:= \left(\mE\left[|\xi|^q\right]\right)^{\frac1q}$ its $\mathcal{L}^q$-norm. Suppose that $\mathcal{F}=\{\mathcal{F}_t,0\le t\le T\}$ is a completed filtration on $(\Omega,\f,\mathbb{P})$. We denote by $\lr^q_{\mathcal{F}}([t_1,t_2];\brn)$ the set of all $\mathcal{F}$-progressively-measurable $\brn$-valued processes $\alpha(\cdot)=\{\alpha(t),\ t_1\le t\le t_2\}$ such that 
\begin{align*}
    \|\alpha\|_{\mathcal{L}^q(t_1,t_2)}:=\left(\e\left[\int_{t_1}^{t_2} |\alpha(t)|^q dt\right]\right)^{\frac 1q} <+\infty;
\end{align*}
and denote by $\mathcal{S}^q_{\mathcal{F}}([t_1,t_2];\brn)$ the family of all $\mathcal{F}_t$-adapted and continuous $\brn$-valued processes $\alpha(\cdot)$ such that 
\begin{align*}
    \|\alpha\|_{\mathcal{S}^q(t_1,t_2)}:=\left(\e\left[\sup_{t_1\le t\le t_2} |\alpha(t)|^q\right]\right)^{\frac 1q} <+\infty.
\end{align*}
When $t_1=0$ and $t_2=T$, we simply  write $\|\cdot\|_{\mathcal{L}^q}$ and $\|\cdot\|_{\mathcal{S}^q}$ for $\|\cdot\|_{\mathcal{L}^q(0,T)}$ and $\|\cdot\|_{\mathcal{S}^q(0,T)}$, respectively.
Finally, we denote by $\lr^q_{\mathcal{F}}([t_1,t_2];\mathcal{P}_q(\mR^{n}))$ the set of all $\mathcal{F}$-progressively-measurable $\mathcal{P}_q(\mR^{n})$-valued processes $m(\cdot)=\{m(t),\ t_1\le t\le t_2\}$ such that $
\e\left[\int_{t_1}^{t_2}\mathcal{M}_q^q (m(t))\md t\right]<+\infty$.		

\subsection{Problem formulation}

\subsubsection{$N$-player Stackelberg game}

We consider the game involving one dominating player (also called Player $0$) and $N$ minor players (we call the $i$-th player as Player $i$), where the state process for the followers have time delays in the information of the leader. The $N$-player Stackelberg game is different from the MFG with one major and many minor players studied in \cite{Carmona-Zhu,Huang-Tang,Nourian-Caines}. In the latter, although the major player strongly influences the minors, all players (including the major player) determine their optimal strategies simultaneously, while in the $N$-player Stackelberg game, the followers move after the leader. 

Our problem includes the following randomness:
\begin{enumerate}[(i)]
    \item (Wiener processes) The Wiener process for Player $0$ is denoted by $W_0$, which is a $\mathbb{R}^{d_0}$-dimensional Wiener process on $(\Omega,\mathscr{F},\mathbb{P})$; for $1\le i\le N$, the Wiener process for Player $i$ is denoted by $W_1^i$, which is a $\mathbb{R}^{d_1}$ -dimensional Wiener process on $(\Omega,\f,\mathbb{P})$. These $N+1$ Wiener processes are independent. 
    \item (Initial condition) The initial condition for Player $0$ is a path $\left\{\xi_0(t): t \in[-b, 0]\right\}$; for $1\le i\le N$, the initial condition for Player $i$ is a random variable $\xi_1^i$, which are iid. These $N+1$ conditions are independent, and they are all independent of all  Wiener processes $\left\{W_0, W_1^i,\ 1\le i\le N\right\}$. 
    \item (Uncertain delay parameter) For Player $i$, the delay parameter is a random variable $\de_i:\Omega\to[a,b]$, with some $0\leq a\le b$. $\{\de_i,\ 1\le i\le N\}$ are iid  with the same distribution $\pi_{\Delta}$,  and they are all independent of all Wiener processes $\left\{W_0, W_1^j,\ 1\le i\le N\right\}$ and initial conditions $\left\{\xi_0,\xi_1^i,\ 1\le i\le N \right\}$.  
\end{enumerate}
We can then define the following filtrations:
\begin{align*}
    \mathcal{F}_t^0  :=\ & \left\{
    \begin{aligned}
        &\sigma\left(\{\xi_0(s):  s \in[-b,t] \}\right),\qquad  t \in[-b, 0],\\
        &\sigma\big(\{\xi_0(s):s\in[-b,0]\},\  \{W_0(s): s \in[0,t]\}\big), \qquad t\in(0,T]; 
    \end{aligned}
    \right.\\
    \mathcal{F}_t^{1, i} :=\ & \sigma\left(\xi_1^i,\  \{W_1^i(s): s \in[ 0,t]\}\right), \qquad t\in[0,T],\qquad 1\le i\le N;\\
    \mathcal{F}_t :=\ & \sigma\big(\mathcal{F}_t^0,\left\{\mathcal{F}_t^{1, i},\ 1\le i\le N\right\},\ \left\{\Delta_j,\ 1\le i\le N\right\} \big),\qquad t\in[0,T];
\end{align*}
here, the filtration $\mathcal{F}^0$ is for the leader, while $\mathcal{F}^{1, i}$ denotes the filtration for Player $i$, and $\mathcal{F}$ is the filtration generated by leader and all followers. Consider the following drift and diffusion coefficients:
\begin{align*}
    &g_0:\RR^{n_0}\times\pr_2(\RR^{n_1})\times \RR^{p_0}\to\RR^{n_0},\qquad \sigma_0:\RR^{n_0}\times\pr_2(\RR^{n_1})\times \RR^{p_0}\to\RR^{n_0\times d_0},\\
    &g_1:\RR^{n_1}\times\pr_2(\RR^{n_1})\times \RR^{p_1}\times\RR^{n_0}\to\RR^{n_1},\qquad \sigma_1:\RR^{n_1}\times\pr_2(\RR^{n_1})\times \RR^{p_1}\times\RR^{n_0}\to\RR^{n_1\times d_1},
\end{align*}
whose regularity assumptions will be imposed in Assumption \ref{assump:lip} (i)(ii). For Player $0$, its control is denoted by $v_0(\cdot) \in \mathcal{L}_{\mathcal{F}^0}^2\left([0, T] ; \mathbb{R}^{p_0}\right)$; for Player $i$, its control (with delay $\delta_i$) is denoted by $ v_1^{i,\delta_i}(\cdot)\in \mathcal{L}_{\mathcal{G}^{i,\delta_i}}^2\left([0, T] ; \mathbb{R}^{p_1}\right)$, where the filtration $\mathcal{G}^{i,\delta_i}$ is defined in Subsection~\ref{subsec:Limit}.
Then, the state process $y_0^{\mathbf{v}}$ for Players $0$, and the state process $y_1^{i, \delta_i, \mathbf{v}}$ for Player $i$ (with a delay $\delta_i$), corresponding to the controls $\mathbf{v}(\cdot):=\left(v_0(\cdot),v_1^{1,\delta_1}(\cdot),\dots v_1^{N,\delta_N}(\cdot)\right)$ satisfy the following stochastic differential equations (SDEs):
\begin{align}
	& \left\{ \begin{aligned}\label{eq:y_0}
		&\md y_0^{\mathbf{v}} (t) =g_0\bigg(y_0^{\mathbf{v}}(t),  \text{\footnotesize $\frac 1N \sum_{j=1}^N \delta_{y_1^{j, \Delta_j,\mathbf{v}}(t)}$ }, v_0(t)\bigg) \md t\\
        &\qquad\qquad +\sigma_0\bigg(y_0^{\mathbf{v}}(t), \text{\footnotesize $\frac 1N {\sum_{j=1}^N \delta_{y_1^{j, \Delta_j,\mathbf{v}}(t)}}$}, v_0(t)\bigg) \md W_0(t), \quad t\in(0,T],\\
		&y_0^{\mathbf{v}}(t)  =\xi_0(t), \quad t \in[-b, 0] ;
	\end{aligned}\right.\\
	&\left\{ \begin{aligned}\label{eq:y_1}
		&\md y_1^{i, \delta_i,\mathbf{v}} (t) =g_1\bigg(y_1^{i, \delta_i,\mathbf{v}}(t), \text{\footnotesize$\frac{1}{N-1}{\sum\limits_{j=1, j \neq i}^N \delta_{y_1^{j, \Delta_j,\mathbf{v}}(t)}}$}, v_1^{i, \delta_i}(t), y_0^{\mathbf{v}}\left(t-\delta_i\right)\bigg) \md t\\
        &\qquad\qquad+\sigma_1\bigg(y_1^{i, \delta_i,\mathbf{v}}(t), \text{\footnotesize$\frac{1}{N-1}{\sum\limits_{j=1, j \neq i}^N \delta_{y_1^{j, \Delta_j,\mathbf{v}}(t)}}$}, v_1^{i, \delta_i}(t), y_0^{\mathbf{v}}\left(t-\delta_i\right)\bigg) \md W_1^i(t), \quad t\in(0,T],\\
		&y_1^{i, \delta_i,\mathbf{v}}(0) =\xi_1^i .
	\end{aligned} \right.
\end{align}
Each Player $i$ together with the dominating player $0$ has the knowledge of the prior probability measure $\pi_{\Delta}$; and each Player $i$ only knows the magnitude of his own delay, all others' delay times are hidden random variables to himself. Equivalently, each player's delay is private information (hidden variable) to others, which resembles an adverse selection market. Consider the following cost coefficients
\begin{align*}
    &f_0:\RR^{n_0}\times\pr_2(\RR^{n_1})\times \RR^{p_0}\to\RR,\qquad h_0:\RR^{n_0}\times\pr_2(\RR^{n_1})\to\RR,\\
    &f_1:\RR^{n_1}\times\pr_2(\RR^{n_1})\times \RR^{p_1}\times\RR^{n_0}\to\RR,\qquad h_1:\RR^{n_1}\times\pr_2(\RR^{n_1})\times\RR^{n_0}\to\RR,
\end{align*}
whose regularity assumptions will be imposed in Assumption \ref{assump:lip}(iii). 
The cost objective functional for Player $0$ and Player $i$ are respectively given by 
\begin{align*}
    \mathcal{J}^{0, N}(\mathbf{v}):=\ & \mathbb{E}\bigg[ \int_0^T f_0\bigg(y^{\mathbf{v}}_0(t), \text{\footnotesize$ \frac{1}{N}\sum_{j=1}^N \delta_{y_1^{j, \Delta_j,\mathbf{v}}(t)}$}, v_0(t)\bigg)\md t+h_0\bigg(y_0(T), \text{\footnotesize$ \frac{1}{N}\sum_{j=1}^N \delta_{y_1^{j, \Delta_j,\mathbf{v}}(T)}$}\bigg)\bigg],\\
    \mathcal{J}^{i, \delta_i, N}(\mathbf{v}):=\ & \mathbb{E} \bigg[\int_0^T f_1\bigg(y_1^{i, \delta_i,\mathbf{v}}(t), \text{\footnotesize$\frac{1}{N-1}\sum_{j=1, j \neq i}^N \delta_{y_1^{j, \Delta_j,\mathbf{v}}(t)}$}, v_1^{i, \delta_i}(t), y_0^{\mathbf{v}}\left(t-\delta_i\right)\bigg) \md t \\
    &\quad + h_1\bigg(y_1^{i, \delta_i,\mathbf{v}}(T), \text{\footnotesize$ \frac{1}{N-1}\sum_{j=1, j \neq i}^N \delta_{y_1^{j, \Delta_j,\mathbf{v}}(T)}$}, y_0^{\mathbf{v}}\left(T-\delta_i\right)\bigg)\bigg].
\end{align*}
The dominating player and the $N$ followers minimize their cost functionals respectively. One notable distinction is that Player $0$ and Players $i$ ($1\le i\le N$) move sequentially, where Player $0$ moves first and Player $i$ ($1\le i\le N$) plays second. Then, the $N$-Player Stackelberg game can be formulated as the follows.

\begin{problem}[$N$-Player Stackelberg game] 
The problem consists the following two sequential optimization problems:
\begin{enumerate}[Step 1] \label{problem:N}
    \item (Nash games for the $N$ followers): for any given $v_0(\cdot)\in \mathcal{L}_{\mathcal{F}^0}^2\left([0, T] ; \mathbb{R}^{p_0}\right)$, search for a set of admissible strategies $\left\{\hat{v}_1^{i,\delta_i}[v_0](\cdot)\in \mathcal{L}_{\mathcal{G}^{i,\delta_i} }^2\left([0, T] ; \mathbb{R}^{p_1}\right),\ 1\le i\le N\right\}$, such that for any $1\le i\le N$, $\hat{v}_1^{i,\delta_i}[v_0]$ is optimal for Player $i$, given the other players’s strategies $\left\{v_0, \hat{v}_1^{i,\delta_i}[v_0],\ j\neq i \right\}$. In other words, 
    \begin{align}\label{def:N}
        &J^{i,\delta_i,N}\left(\mathbf{\hat{v}}[v_0]\right)=\min_{v_1^{i,\delta_1}} J^{i,\delta_i,N}\left(\mathbf{v}^i[v_0]\right),\quad 1\le i\le N,
    \end{align}
    where $\mathbf{\hat{v}}[v_0]:= \left\{v_0, {\hat{v}}^{i,\delta_i}_1[v_0],\ 1\le i\le N \right\}$, and $\mathbf{v}^i[v_0]:= \left\{v_0, v_1^{i,\delta_i}, {\hat{v}}^{j,\delta_j}_1[v_0],\ j\neq i \right\}$.
    \item (Stochastic control problem for the Leader): search for a control $\hat{v}_0(\cdot) \in \mathcal{L}_{\mathcal{F}^0}^2\left([0, T] ; \mathbb{R}^{p_0}\right)$, such that 
    \begin{align}\label{def:0}
        J^{0,N}\left(\mathbf{\hat{v}}\left[\hat{v}_0\right]\right)=\min_{v_0} J^{0,N} \left(\mathbf{\hat{v}}[v_0]\right).
    \end{align}
\end{enumerate}
If a set of controls $\mathbf{\hat{v}}\left[\hat{v}_0\right]:= \left\{\hat{v}_0, \hat{v}^{i,\delta_i}_1\left[\hat{v}_0\right],\ 1\le i\le N \right\}$ satisfies \eqref{def:N} and \eqref{def:0}, we call $\mathbf{\hat{v}}\left[\hat{v}_0\right]$ a solution of our $N$-player Stackelberg game. 
\end{problem}

In Step 1 of Problem~\ref{problem:N}, since the Player $i$ interacts with the population through the term $\frac{1}{N-1}\sum_{j=1, j \neq i}^N \delta_{y_1^{j, \Delta_j,\mathbf{v}}(t)}$, establishing an exact Nash equilibrium becomes challenging as $N$ grows large. Instead, we aim to search for an approximate solution of our $N$-player Stackelberg game, which is defined as follows; we also refer to \cite{MR3857463} for the same definition. 

\begin{definition}\label{def:equili}
\textbf{(1) $\epsilon$-Nash equilibrium.}  Given $v_0$, the collection of strategies \[\left\{u_1^{i,\delta_i}[v_0]\in \mathcal{U}_i(v_0),\ 1\le j\le N \right\}\] is said to constitute an $\epsilon$-Nash equilibrium for the $N$ followers  (with the given $v_0$) if
 \begin{align*}
        J^{i,\delta_i,N}\left(\mathbf{u}[v_0]\right)\leq \ & \min_{v_1^{i,\delta_i}\in \mathcal{U}_i(v_0)} J^{i,\delta_i,N}\left(\mathbf{v}^i[v_0]\right)+\epsilon,\quad 1\le i\le N,
    \end{align*}
     where ${\mathbf{v}^i}[v_0]:= \left\{u_0, v^{i,\delta_i}_1[v_0], u^{j,\delta_j}_1[v_0],\ j\neq i \right\}$ and $\mathbf{u}[v_0]:= \left\{v_0, u^{i,\delta_i}_1[v_0],\ 1\le i\le N \right\}$.

\textbf{(2) $(\epsilon_1,\epsilon_2)$-Stackelberg Nash equilibrium.}  
 The collection of admissible strategies \[\left\{u_0\in \mathcal{U}_0,u_1^{i,\delta_i}[u_0]\in \mathcal{U}_i(u_0),\ 1\le j\le N \right\}\] is called an $(\epsilon_1,\epsilon_2)$-Stackelberg Nash equilibrium, if $\left\{u_1^{i,\delta_i}[u_0]\in \mathcal{U}_i(u_0),\ 1\le j\le N \right\}$ constitutes an \(\epsilon_1\)-Nash equilibrium for the $N$ players and
    \begin{align*}
        J^{0,N}\left(\mathbf{u}[u_0]\right)\leq\ &\min_{v_0\in \mathcal{U}_0} J^{0,N} \left(\mathbf{u}[v_0]\right)+\epsilon_2.
    \end{align*}   
\end{definition}

Here, the admissible control set for the $i$-th Player $\mathcal{U}_i(v_0)\subset \mathcal{L}_{\mathcal{G}^{i,\delta_i} }^q\left([0, T] ; \mathbb{R}^{p_1}\right)$ with bounded $\mathcal{L}^q$ norm, see Definition~\ref{def:U_i} for details; and the the admissible control set for the leader  $\mathcal{U}_0\subset \mathcal{L}_{\mathcal{F}^0}^q\left([0, T] ; \mathbb{R}^{p_0}\right)$ with bounded $\mathcal{L}^q$ norm, see Definition~\ref{def:U_0} for details. In this article, we require the condition $q>4$ (stem from the usage of Lemma~\ref{lemma:wasserstien}) to obtain a precise convergence rate faster than $o(1)$. As long as the $o(1)$-convergence rate, $q=2$ is sufficient; see \cite{Bensoussan-Chau-Yam} for  proof of the constant diffusion case. 

Our method is to approximate the equilibrium by considering the $N\to \infty$ limit. As in \cite{Bensoussan-Chau-Yam}, we solve the optimality problem in this limit and use the solution to construct an $(\epsilon_1,\epsilon_2)$-Stackelberg Nash equilibrium, providing an efficient approximation for the large-scale system.

\subsubsection{Limiting mean field Stackelberg game problem}\label{subsec:Limit}

We now give the formulation of the limiting mean field Stackelberg game. When the term \text{\footnotesize$\frac{1}{N}\sum_{j=1}^N \delta_{y_1^{j, \Delta_j,\mathbf{v}}(t)}$} converge to some conditional distribution $z(t)$ (see Condition~\ref{def:z} below). The conditional distribution flow $z(\cdot)\in\mathcal{P}_2(\mR^{n_1})$ is assumed to be adapted to  $\mathcal{F}^0_{\cdot-a}$, and we denote by $\mathcal{F}_t^z$ the filtration generated by $z$. We set $\mathcal{G}_t^{i,\delta}:=\mathcal{F}_t^{1,i} \vee \mathcal{F}_{t-\delta}^0 \vee \mathcal{F}_t^z$. For a control $v_0(\cdot) \in \mathcal{L}_{\mathcal{F}^0}^2\left([0, T] ; \mathbb{R}^{p_0}\right)$ for Player $0$ and a control $v_1^{i,\delta_i}(\cdot)\in \mathcal{L}_{\mathcal{G}^{i,\delta_i} }^2\left([0, T] ; \mathbb{R}^{p_1}\right)$ for a representative player $i$, the controlled state processes are, respectively, described by
\begin{align}
&\left\{
	\begin{aligned}
		&\md x_0^{v_0,z} (t) = g_0\left(x_0^{v_0,z}(t), z(t), v_0(t)\right) \md t +\sigma_0 \left(x_0^{v_0,z}(t), z(t), v_0(t)\right) \md W_0(t) , \quad t\in(0,T],\\
		&x_0^{v_0,z}(t) =\xi_0(t), \quad t \in[-b, 0],
	\end{aligned}
\right. \label{eq:x_0}\\
&\left\{
	\begin{aligned}
		&\md x_1^{i,\delta_i,v_0,z,v_1} (t) =g_1\left(x_1^{i,\delta_i,v_0,z,v_1}(t), z(t), v_1^{i,\delta_i}(t), x_0^{v_0,z}\left(t-\delta_i\right)\right) \md t\\
        &\quad\qquad\qquad\qquad +\sigma_1  \left(x_1^{i,\delta_i,v_0,z,v_1}(t), z(t), v_1^{i,\delta_i}(t), x_0^{v_0,z}\left(t-\delta_i\right)\right) \md W_1^i(t) ,\\
		&x_1^{i,\delta_i,v_0,z,v_1}(0)  =\xi_1^i;\label{eq:x_1^i}
		\end{aligned}
\right. 
\end{align}
and the cost functional is given by
\begin{align}
    \mathcal{J}^{0}\left(v_0;z\right):=\ &\mathbb{E} \bigg[\int_0^T f_0\left(x_0^{v_0,z}(t), z(t), v_0(t)\right) \md t+h_0\left(x_0^{v_0,z}(T),z(T)\right)\bigg], \label{eq:J:MF:doni}\\
    \mathcal{J}^{i, \delta_i}(v_1^{i, \delta_i}; v_0,z):=\ & \mathbb{E}\bigg[ \int_0^T f_1\left(x_1^{i, \delta_i,v_0,z,v_1}(t), z(t), v_1^{i, \delta_i}(t), x_0^{v_0,z}\left(t-\delta_i\right)\right) \md t \notag \\
    &\qquad +h_1\left(x_1^{i, \delta_i,v_0,z,v_1}(T), z(T),  x_0^{v_0,z}\left(T-\delta_i\right)\right)\bigg]. \label{eq:J:MF}
\end{align}
Under assumptions stated in the following section, the controlled processes in \eqref{eq:x_0} and \eqref{eq:x_1^i} satisfy $x_0^{v_0,z}(\cdot) \in \lr_{\mathcal{F}^0}^2\left([-b, T] ; \mathbb{R}^{n_0}\right)$ and $x_1^{i,\delta,v_0,z,v_1} \in \lr_{\mathcal{G}^{i,\delta}}^2\left([0, T] ; \mathbb{R}^{n_1}\right)$. The limiting Stackelberg game can be formulated as follows.

\begin{problem}[Mean field Stackelberg game] \label{problem:Limit}
The problem consists the following two sequential optimization problems:
\begin{enumerate}[Step 1]
    \item (Mean field limiting Nash game for followers): for any given $v_0(\cdot)\in \lr_{\mathcal{F}^0}^q\left([0, T] ; \mathbb{R}^{p_0}\right)$, \\
    (i) given the conditional distribution flow $z\in \lr^2_{\F_{\cdot-a}^0} ([0,T];\pr_2(\RR^{n_1}))$, find an optimal control $u_1^{i,\delta_i}[v_0,z]$ of the following problem 
    \begin{align*}
        &\inf_{v_1^{i,\delta_i}} \mathcal{J}^{i, \delta_i}\left(v_1^{i, \delta_i}; v_0,z\right);
    \end{align*}
    (ii) search for a conditional distribution flow $z[v_0]$, such that the state process $x_1^{i,\delta_i,v_0,z,u_1}(\cdot)$ corresponding to the optimal control $u_1^{i,\delta_i}[v_0,z[v_0]]$ (which we simply denote by $u_1^{i,\delta_i}[v_0]$) satisfies the following consistent condition:
    \begin{condition}[Fixed Point Property] \label{def:z} 
        $\qquad z[v_0](t)=\int_{[a,b]}\mathbb{P}^{\text{\tiny $\F_{t-\delta}^0 \vee \F_{t}^z$}}_{\text{\tiny $x_1^{i,\delta,v_0,z,u_1}(t)$}} \md\pi_{\Delta}(\delta)$.
    \end{condition}
    \item (Stochastic control problem for leader): search for a control $u_0(\cdot) \in \lr_{\mathcal{F}^0}^q\left([0, T] ; \mathbb{R}^{p_0}\right)$, such that 
    \begin{align*}
        J^{0}\left(u_0;z[u_0]\right)=\min_{v_0} J^{0} \left(v_0,z[v_0]\right).
    \end{align*}
\end{enumerate}
If a triple $\left(u_0,z[u_0],u^{i,\delta_i}_1[v_0]\right)$ satisfies the above conditions, then we call $\left(u_0,z[u_0],u^{i,\delta_i}_1[v_0]\right)$ a solution of the limiting mean field Stackelberg game. 
\end{problem}

\begin{remark}
   Note that the definition of $z[v_0]$ in Condition~\ref{def:z} is a conditional distribution flow, rather than a conditional expectation flow. This arises from our assumption that the $i$-th player interacts with others through the empirical distribution rather than the mean of their states; for comparison, in \cite{Bensoussan-Chau-Yam}, $z$ is defined as a conditional expectation flow. Besides, we are the first to simultaneously incorporate the following features (and the first to incorporate the (iii)-th feature) into $N$-player Stackelberg game:   (i) the information available to the followers may be subject to different magnitudes of  delays; (ii) the followers interact with each other through the empirical distribution, rather than only through the mean of the states; (iii) the drift and diffusion coefficients of the state processes for both the leader and the followers are allowed to depend on the state, the control, and the distribution.  
\end{remark}

The aim of this article is to show that a solution of Problem~\ref{problem:Limit} can provide an approximate solution of Problem~\ref{problem:N} in view of Definition~\ref{def:equili}, that is, 

\begin{problem}\label{eq:aim}
    Suppose that $\left(u_0,z[u_0],u^{i,\delta_i}_1[u_0]\right)$ is a solution of Problem~\ref{problem:Limit}. Show that the strategies $\left\{u_0,u^{i,\delta_i}_1[u_0],\ 1\le i\le N\right\}$ is an $(\epsilon_1(N),\epsilon_2(N))$-Stackelberg Nash equilibrium for Problem~\ref{problem:N}, and establish the convergence rates of $\epsilon_1(N)$ and $\epsilon_2(N)$ as $N\to+\infty$.
\end{problem}

For notational convenience, in the rest of this article, we adopt the following notations
\begin{enumerate}
    \item  For any feasible $v_0$ for Player $0$, we always denote by $\left(z[v_0],u_1^{i,\delta_i}[v_0]\right)$ the solution of Step 1 of Problem~\ref{problem:Limit}  , and denote by $\left(x^{v_0}_0,x^{i,\delta_i,v_0}_1\right)$ the corresponding state processes, which satisfies Condition~\ref{def:z}, written as 
    \begin{align}\label{def:z'} 
        z[v_0](t)=\int_{[a,b]}\mathbb{P}^{\text{\tiny$\F_{t-\delta}^0 \vee \F_{t}^z$}}_{\text{\tiny $x_1^{i,\delta,v_0}(t)$}}\md\pi_{\Delta}(\delta);
    \end{align}
    while for any other feasible control $v_1^{i,\delta_i}$, we denote by $\left(x^{v_0}_0,x^{i,\delta_i,v_0,v_1}_1(\cdot)\right)$ the state processes corresponding to $\left(v_0,z[v_0],v_1^{i,\delta_i}\right)$. 
    \item  We denote by $u_0$ the solution of Step 2 of Problem~\ref{problem:Limit}, and denote by $\left(x_0,x^{i,\delta_i}_1\right)$ the state processes corresponding to $\left(z[u_0],u_1^{i,\delta_i}[u_0]\right)$ . 
    \item For any feasible control $v_0$ for Player $0$, we denote by the set of controls $\mathbf{u}[v_0]:= \left\{v_0, u^{j,\delta_j}_1[v_0],\ 1\le j\le N \right\}$ and denote by $\left\{y_0^{\mathbf{u}[v_0]},y_1^{j, \delta_j,\mathbf{u}[v_0]},\ 1\le j\le N\right\}$ the corresponding state processes; and for any other control $v^{1,\delta_i}_1$ for Player $i$, we denote by the  set of control ${\mathbf{v}^i}[v_0]:= \left\{v_0, v_1^{i,\delta_i}, u^{j,\delta_j}_1[u_0],\ j\neq i \right\}$, and denote by $\left\{y_0^{{\mathbf{v}^i}[v_0]},y_1^{j, \delta_j,{\mathbf{v}^i}[v_0]},\ 1\le j\le N\right\}$ the corresponding state processes. We denote by the  set of control $\mathbf{u}:= \left\{u_0, u^{j,\delta_j}_1[u_0],\ 1\le j\le N \right\}$, and denote by $\left\{y_0^{\mathbf{u}},y_1^{j, \delta_j,\mathbf{u}},\ 1\le j\le N \right\}$ the corresponding state processes. 
\end{enumerate}

\section{Assumptions and Estimates for Controlled SDEs in Section~\ref{sec:formulation}}\label{sec:assumption&SDE}

Our assumptions on the initial conditions are as follows.

\begin{assumption}[for the initials]\label{assumption:initial}
(i)  The initial conditions $\xi_0$ and  $\left\{\xi_1^j,\ 1\le j\le N\right\}$ satisfy
\begin{align}\label{assum:xi_0_bound}
    \e\left[\sup_{t\in[-b,0]}|\xi_0(t)|^q\right]<\infty,\qquad \e\left[\left|\xi_1^j\right|^q \right]<\infty,
\end{align}    
for some $q>4$. \\
(ii) The initial path, $\left\{\xi_0(t): t \in[-b, 0]\right\}$, satisfies the average Hölder continuity, such that there exists $L>0$ and $\tilde{q}\geq \frac{q-2}{q}$,
\begin{align}\label{assum:xi_0_conti}
    \mathbb{E}\left[\left|\xi_0(t)-\xi_0(s)\right|^2 \right]\leq L|t-s|^{\tilde{q}}, \quad t, s \in[-b, 0] .
\end{align} 
\end{assumption}

We refer the readers to the H\"older continuity \eqref{C.2.} for the state process in the parameter $\delta$, that is why we only need the power $\tilde{q}\geq \frac{q-2}{q}$ on the right hand side of \eqref{assum:xi_0_conti}. Our Condition \eqref{assum:xi_0_conti} is weaker than the commonly used one, i.e., $\tilde{q}=1$, see \cite{Bensoussan-Chau-Yam} for instance. The weaker assumption \eqref{assum:xi_0_conti} is enough for this article since we assume the boundedness condition \eqref{assum:xi_0_bound} with $q>4$; see \textit{a priori} estimate \eqref{C.2._proof_2} for details. 

\begin{remark}
    As a particular case, if the diffusion coefficient $\sigma_0$ of the dominating player is independent of the control variable, we can simply take $\tilde{q}=1$ in \eqref{assum:xi_0_conti}, then, we can obtain an even better estimate \eqref{better_estimate} than \eqref{C.2.}, by then, a faster convergence rate can be established; see Remark~\ref{remark:sigma_0} for details.
\end{remark}

We now impose some standard assumptions on the drift, diffusion and cost coefficient functions in the SDEs. For notational simplicity, we use the same constant $L$ below. 

\begin{assumption}[For the coefficients]\label{assump:lip}
For any $x_0,x_0'\in\mR^{n_0}; x_1,x_1'\in\mR^{n_1}; v_0,v_0'\in\mR^{p_0}; v_1,v_1'\in\mR^{p_1}$ and $z,z'\in\mathcal{P}_2(\mR^{n_1})$, we assume the following:\\
 	(i)	Lipschitz continuity. The drift coefficients $g_0$ and $g_1$  and the diffusion coefficients $\sigma_0$ and $\sigma_1$ are globally Lipschitz continuous in all arguments, i.e., there exists $L>0$, such that
 		$$
 		\begin{aligned}
 			&\mid g_0\left(x_0, z, v_0\right)-  g_0\left(x_0^{\prime}, z^{\prime}, v_0^{\prime}\right) \mid  + \mid \sigma_0\left(x_0, z, v_0\right)-  \sigma_0\left(x_0^{\prime}, z^{\prime}, v_0^{\prime}\right)  \mid \\
 			\le\ & L\left(\left|x_0-x_0^{\prime}\right|+W_2(z,z')+\left|v_0-v_0^{\prime}\right|\right), \\
 			&\mid g_1\left(x_1, z, v_1, x_0\right)-  g_1\left(x_1^{\prime}, z^{\prime}, v_1^{\prime}, x_0^{\prime}\right)  + \mid \sigma_1\left(x_1, z, v_1, x_0\right)-  \sigma_1\left(x_1^{\prime}, z^{\prime}, v_1^{\prime}, x_0^{\prime}\right)  \\
 			\le\ & L\left(\left|x_1-x_1^{\prime}\right|+W_2(z,z')+\left|v_1-v_1^{\prime}\right|+\left|x_0-x_0^{\prime}\right|\right)
 		\end{aligned}
 		$$ 	
    (ii) Linear growth. The drift coefficients $g_0$ and $g_1$ and the diffusion coefficients $\sigma_0$ and $\sigma_1$ are of linear growth in all arguments, i.e., there exists $L>0$, such that
$$
\begin{aligned}
	\left|g_0\left(x_0, z, v_0\right)\right|  + \left|\sigma_0\left(x_0, z, v_0\right)\right| & \leq L\left(1+\left|x_0\right|+\mathcal{M}_2(z)+\left|v_0\right|\right) \\
	\left|g_1\left(x_1, z, v_1, x_0\right)\right|  + \left|\sigma_1\left(x_1, z, v_1, x_0\right)\right| & \leq L\left(1+\left|x_1\right|+\mathcal{M}_2(z)+\left|v_1\right|+\left|x_0\right|\right)
\end{aligned}
$$
	(iii)	Quadratic condition. There exists $L>0$ such that
        	\begin{align*}
			\left|f_0\left(x_0, z, v_0\right)-f_0\left(x_0^{\prime}, z^{\prime}, v_0^{\prime}\right)\right| \leq & L\left[1+\left|x_0\right|+\left|x_0^{\prime}\right|+\mathcal{M}_2(z)+\mathcal{M}_2(z')+\left|v_0\right|+\left|v_0^{\prime}\right|\right] \\
			& \cdot\left[\left|x_0-x_0^{\prime}\right|+W_2(z,z^{\prime})+\left|v_0-v_0^{\prime}\right|\right], \\
			\left|f_1\left(x_1, z, v_1, x_0\right)-f_1\left(x_1^{\prime}, z^{\prime}, v_1^{\prime}, x_0^{\prime}\right)\right| \leq & L\left[1+\left|x_1\right|+\left|x_1^{\prime}\right|+\mathcal{M}_2(z)+\mathcal{M}_2(z^{\prime})+\left|v_1\right|+\left|v_1^{\prime}\right|+\left|x_0\right|+\left|x_0^{\prime}\right|\right] \\
			& \cdot\left[\left|x_1-x_1^{\prime}\right|+W_2(z,z^{\prime})+\left|v_1-v_1^{\prime}\right|+\left|x_0-x_0^{\prime}\right|\right],\\
            	\left|h_0\left(x_0, z\right)-h_0\left(x_0^{\prime}, z^{\prime}\right)\right| \leq & L\left[1+\left|x_0\right|+\left|x_0^{\prime}\right|+\mathcal{M}_2(z)+\mathcal{M}_2(z')\right] \\
			& \cdot\left[\left|x_0-x_0^{\prime}\right|+W_2(z,z^{\prime})\right],\\
            	\left|h_1\left(x_1, z, x_0\right)-h_1\left(x_1^{\prime}, z^{\prime}, x_0^{\prime}\right)\right| \leq & L\left[1+\left|x_1\right|+\left|x_1^{\prime}\right|+\mathcal{M}_2(z)+\mathcal{M}_2(z')+\left|x_0\right|+\left|x_0^{\prime}\right|\right] \\
			& \cdot\left[\left|x_1-x_1^{\prime}\right|+W_2(z,z^{\prime})+\left|x_0-x_0^{\prime}\right|\right]. 
		\end{align*}
\end{assumption}

\begin{remark}
For a coefficient $g:\RR^{n_1}\times \RR^{n_1}\times  \RR^{d_1}\times \RR^{n_0} \ni (x_1,y,v_1,x_0)\mapsto g(x_1,y,v_1,x_0)\in\RR^{n_1} $ which satisfies the $L$-Lipschitz continuous condition in all its arguments (see \cite{Bensoussan-Chau-Yam} for instance), we define the map $G:\RR^{n_1}\times \pr_2\left(\RR^{n_1}\right)\times\RR^{d_1}\times \RR^{n_0}\to\RR^{n_1}$ as follows 
\begin{align*}
    G(x_1,z,v_1,x_0):= g \left( x_1,\int_{\RR^{n_1}} y\ z(\md y),v_1,x_0\right).
\end{align*}
Then, from the continuity of $g$, we have
\begin{align*}
    &|G(x'_1,z',v'_1,x'_0)-G(x_1,z,v_1,x_0)|\\
    =\ & \left| g \left( x'_1,\int_{\RR^{n_1}} y\ z'(\md y),v'_1,x'_0\right) - g \left( x_1,\int_{\RR^{n_1}} y\ z(\md y),v_1,x_0\right) \right| \\
    \le\ & L \bigg(|x'_1-x_1|+\left|\int_{\RR^{n_1}} y\ z'(\md y)-\int_{\RR^{n_1}} y\ z(\md y) \right|+|v'_1-v_1|+|x'_0-x_0| \bigg).
\end{align*}
Note from \cite[Section 2]{MR3357611} that
\begin{align*}
    \left|\int_{\RR^{n_1}} y\ z'(\md y)-\int_{\RR^{n_1}} y\ z(\md y) \right|
    \le\ & W_2(z,z').
\end{align*} 
Therefore, we know that the map $G$ satisfies the Lipschitz-continuity condition (of $g_1$) in Assumption~\ref{assump:lip}. In a simialr way, it can be shown that $G$ satisfies the growth condition in Assumption~\ref{assump:lip}. 
\end{remark}

For some $q>4$, we work in the space $\mathcal{L}_{\mathcal{F}^0}^q\left([0, T] ; \mathbb{R}^{p_0}\right)\subset \mathcal{L}_{\mathcal{F}^0}^2\left([0, T] ; \mathbb{R}^{p_0}\right)$ for controls $v_0(\cdot)$ for Player $0$. By appropriate assumptions, we can show that Step 1 of Problem~\ref{problem:Limit} has a unique optimal control with the following properties. Here, we simply make this condition, which will be proven in detail in an alternative article. 

\begin{hypothesis}[for optimal control]\label{assumption:integral}
(i) For any feasible control $v_0\in \mathcal{L}_{\mathcal{F}^0}^q\left([0, T] ; \mathbb{R}^{p_0}\right)$, the solution $u_1^{j,\delta}[v_0](\cdot)$ for Step 1 of Problem~\ref{problem:Limit} satisfies
\begin{align*}
    \sup_{\delta\in[a,b]}\e\left[\int_0^T \left|u_1^{i,\delta}[v_0](s) \right|^q \md s\right]<+\infty.
\end{align*}
(ii) There exists $l_{u_1[v_0]} > 0$ such that the mapping $[a,b] \ni \delta \mapsto u_1^{i,\delta}[v_0](\cdot)\in \mathcal{L}_{\mathcal{G}^{i,\delta} }^q\left([0, T] ; \mathbb{R}^{p_0}\right)$ is H\"older continuous, that is,
\begin{align*}
    \e \left[\int_0^T \left|u_1^{i,\delta}[v_0](s) - u_1^{i,\gamma}[v_0](s)\right|^2 \md s\right] \leq l_{u_1[v_0]}|\delta - \gamma|^{\frac{q-2}{q}}, \quad \forall \delta, \gamma \in [a,b].
\end{align*}
\end{hypothesis}

We next give the boundedness estimates for the state processes for Problem~\ref{problem:N} and Problem~\ref{problem:Limit}, which will be used in the following sections. The proof of the following result is given in Appendix ~\ref{prop:proof:moment}.

\begin{lemma}\label{prop:moment}
(i) Under Assumptions~\ref{assumption:initial}-(i) and \ref{assump:lip} and Hypothesis~\ref{assumption:integral}-(i) with $q\geq 2$, we have the following estimates:\\
(i)(a) For any $v_0\in \mathcal{L}_{\mathcal{F}^0}^2\left([0, T] ; \mathbb{R}^{p_0}\right)$ and $v_1^{i,\delta_i}\in \lr_{\mathcal{G}^{i,\delta_i}}^2\left([0, T] ; \mathbb{R}^{n_1}\right)$, the state processes $\left(x^{v_0}_0,x^{i,\delta_i,v_0,v_1}_1(\cdot)\right)$ satisfy
\begin{align}
   \left\|x^{v_0}_0\right\|_{\sr^2(-b,T)}^2\le\ & C(L,T) \bigg[ 1+ \left\|\xi_0\right\|^2_{\sr^2(-b,0)} + \left\|\xi^i_1\right\|_2^2 +  \|v_0\|_{\lr^2}^2 + \sup_{\delta\in[a,b]} \left\|u_1^{i, \delta}[v_0]\right\|_{\lr^2}^2  \bigg], \label{C.1.-1'} \\
    \sup_{\delta\in[a,b]} \left\|x_1^{i, \delta, v_0,v_1}\right\|^2_{\sr^2} \le\ & C(L,T) \bigg[ 1+\left\|\xi_0\right\|_{\sr^2(-b,0)}^2 + \left\|\xi^i_1\right\|_2^2 + \left\|v_0\right\|_{\lr^2}^2   + \sup_{\delta\in[a,b]}\left\|u_1^{i, \delta}[v_0]\right\|_{\lr^2}^2 + \sup_{\delta\in[a,b]} \left\|v_1^{i, \delta}\right\|_{\lr^2}^2  \bigg],\label{C.1.-1''}
\end{align}
and particularly, for the state process $x^{i,\delta_i,v_0}_1(\cdot)$ corresponding to the optimal control $u_1^{i, \delta}[v_0]$ satisfies
\begin{align}
    \sup_{\delta\in[a,b]} \left\|x_1^{i, \delta, v_0}\right\|_{\sr^2}^2 \le\ & C(L,T) \bigg[ 1+\left\|\xi_0\right\|_{\sr^2(-b,0)}^2+ \left\|\xi^i_1\right\|_2^2 + \left\|v_0\right\|_{\lr^2}^2   + \sup_{\delta\in[a,b]}\left\|u_1^{i, \delta}[v_0]\right\|_{\lr^2}^2 \bigg]. \label{C.1.hat.x_1}
\end{align}
(i)(b) For the set of controls ${\mathbf{v}^i}[v_0]:= \left\{v_0, v_1^{i,\delta_i}, u^{j,\delta_j}_1[u_0],\ j\neq i \right\}$, the state processes $\left\{y_0^{{\mathbf{v}^i}[v_0]},y_1^{j, \delta_j,{\mathbf{v}^i}[v_0]},\ 1\le j\le N\right\}$ satisfy
\begin{align}
    &\left\|y_0^{{\mathbf{v}^i}[v_0]}\right\|^2_{\sr^2(-b,T)}+\sup_{\delta\in[a,b]} \left\|y_1^{j, \delta,{\mathbf{v}^i}[v_0]}\right\|^2_{\sr^2} \notag \\
    \le\ & C(L,T) \bigg[ 1+\left\|\xi_0(t)\right\|_{\sr^2(-b,0)}^2 + \left\|\xi^j_1\right\|_2^2 + \|v_0\|_{\lr^2}^2  +\sup_{\delta\in[a,b]}\left\|u_1^{j, \delta}[v_0]\right\|^2_{\lr^2} + \frac{1}{N-1}\sup_{\delta\in[a,b]} \left\|v_1^{i, \delta}(t)\right\|_{\lr^2}^2  \bigg], \label{C.1.-2'} \\
    &\sup_{\delta\in[a,b]} \left\|y_1^{i, \delta,{\mathbf{v}^i}[v_0]}\right\|^2_{\sr^2} \notag \\
    \le\ & C(L,T) \bigg[ 1+\left\|\xi_0(t)\right\|_{\sr^2(-b,0)}^2 + \left\|\xi^j_1\right\|_2^2 + \|v_0\|_{\lr^2}^2  +\sup_{\delta\in[a,b]}\left\|u_1^{j, \delta}[v_0]\right\|^2_{\lr^2} + \sup_{\delta\in[a,b]} \left\|v_1^{i, \delta}(t)\right\|_{\lr^2}^2  \bigg]. \label{C.1.-2''}
\end{align}
Here, in \eqref{C.1.-2'} and \eqref{C.1.-2''}, we recall the notations stated at the end of Section~\ref{sec:formulation}: Player $i$ adopts an arbitrary control $v_1^{i,\delta_i}$, and the other followers (we simply use Player $j$ to denote a representative) adopts the optimal control $u_1^{j,\delta_j}[v_0]$. \\
(ii) Furthermore, for $q>4$, suppose that Assumption~\ref{assumption:initial}-(i) and Hypothesis~\ref{assumption:integral}-(i) hold, we have
\begin{align}
    \left\|x^{v_0}_0\right\|^q_{\sr^q(-b,T)} + \sup_{\delta\in[a,b]} \left\|x_1^{i, \delta, v_0}\right\|^q_{\sr^q} \le\ & C(L,T,q) \bigg[ 1+ \|\xi_0\|_{\sr^q(-b,0)}^q + \left\|\xi^i_1\right\|_q^q + \|v_0\|_{\lr^q}^q + \sup_{\delta\in[a,b]} \left\|u_1^{i, \delta}[v_0]\right\|^q_{\lr^q}  \bigg]. \label{estimate:SDE:Lq-1} 
\end{align}
\end{lemma}

\begin{remark}
In the proof of Lemma \ref{prop:moment}-(i), we have used the following two estimates:
\begin{align}
    \e\left[\mathcal{M}_2^2(z[v_0](t))\right] \le\ & \sup_{\delta\in[a,b]} \left\|x_1^{i, \delta, v_0}(t)\right\|^2_2,\label{eq:z:estimate}\\
    \mE\bigg[\mathcal{M}_2^2\bigg(\text{\footnotesize$\frac{1}{N-1}\sum_{j \neq i} \delta_{y_1^{j, \Delta_j, {\mathbf{v}^i}[v_0]}(t)}$}\bigg)\bigg] \le\ & \sup_{\delta\in[a,b]}  \left\|y_1^{j, \delta, {\mathbf{v}^i}[v_0]}(t)\right\|^2_2 ,\label{eq:y/N:estimate}
\end{align}
which will be used repeatedly in what follows.
\end{remark}

We also provide the estimate on the continuity of $x_1^{i, \delta, v_0}(t)$ in $\delta$, which will be used in Subsubsection~\ref{subsubsec:General_Delta}. The proof of the following result is given in Appendix \ref{proof:prop:holder}.
 
\begin{lemma}\label{prop:holdercontinuity}

Under  Assumptions~\ref{assumption:initial} and \ref{assump:lip} and Hypothesis~\ref{assumption:integral}, we have the following H\"older continuity:
\begin{align}\label{C.2.}
    &\sup_{s\in[0,T]}\left\|x_1^{i,\delta,v_0}(s)-x_1^{i,\gamma,v_0}(s)\right\|^2_2 \notag \\
    \le\ & C\left(L,T,\xi_0,\xi^i_1,q\right)\bigg(1+l_{u_1[v_0]}+\|v_0\|_{\mathcal{L}^q}^2+\sup_{\delta\in[a,b]}\left\|u_1^{i, \delta}[v_0]\right\|_{\mathcal{L}^2}^2\bigg)(\gamma-\delta)^{\frac{q-2}{q}},\quad \forall\,\delta,\gamma\in[a,b].
\end{align}
\end{lemma}

\begin{remark} \label{remark:sigma_0}
 For if $\sigma_0$ is independent of the control $v_0$, we  can impose a stronger regularity condition on $\xi_0$ such that $\tilde{q}= 1$ in \eqref{assum:xi_0_conti}.
    Then, by following the argument in the proof of Lemma \ref{prop:holdercontinuity}, we obtain the following better result than \eqref{C.2.} of H\"older continuity estimate:
\begin{align}
    &\sup_{s\in[0,T]}\left\|x_1^{i,\delta,v_0}(s)-x_1^{i,\gamma,v_0}(s)\right\|^2_2 \notag \\
    \le\ & C\left(L,T,\xi_0,\xi^i_1\right)\bigg(1+l_{u_1[v_0]}+\|v_0\|_{\mathcal{L}^2}^2+\sup_{\delta\in[a,b]}\left\|u_1^{i, \delta}[v_0]\right\|_{\mathcal{L}^2}^2\bigg)(\gamma-\delta),\quad \forall\,\delta,\gamma\in[a,b]. \label{better_estimate}
\end{align}
\end{remark}

\section{Convergence of an $N$-player system}\label{sec:convergence}
To give a precise description of $\epsilon_1(N)$ and $\epsilon_2(N)$ in Problem~\ref{eq:aim}, one can  establish the  precise convergence rates of state process of $N$-player system towards the mean-field system. In this section, we  assume that Player $0$ takes a fixed control $v_0(\cdot) \in \mathcal{L}_{\mathcal{F}^0}^q\left([0, T] ; \mathbb{R}^{p_0}\right)$. The objective of this section is to give the convergence rate of the distance between the state process for  Problem~\ref{problem:N} and that for Problem~\ref{problem:Limit}, with the same control $\mathbf{u}[v_0]$, that is, the norm $\left\|y_0^{\mathbf{u}[v_0]}-x_0^{v_0}\right\|_{\mathcal{S}^2}$ and $ \left\|y_1^{i,\delta_i,\mathbf{u}[v_0]}-x_1^{i,\delta_i,v_0}\right\|_{\mathcal{S}^2}$, which will be done in Subsection \ref{sec:converge:optimal}. Beside, we also give the convergence rate for the norm of  
$y_0^{{\mathbf{v}^i}[v_0]}-x_0^{v_0}$, $y_1^{i,\delta_i,{\mathbf{v}^i}[v_0]}-x^{i,\delta_i,v_0,v_1}_1$ and $y_1^{j,\delta_{j},{\mathbf{v}^i}[v_0]}-x^{j,\delta_{j},v_0}_1$ for $j\neq i$, which will be done in Subsection \ref{sec:converge:arbitrary}.

As a whole, in both cases, we shall give the
$\mathcal{O}\left((f(N))^{\frac{q-2}{3q-4}}\right)$-convergence rate, where $f$ was defined in \eqref{def:function:f}. As a comparison, we refer the readers to an interesting work \cite[Theorem 7.1]{Carmona-Zhu} for the study conditional propagation of chaos (corresponding to MFGs with major and minor players) with an $ \mathcal{O}(N^{-1/(n_1+4)})$ convergence rate, where the diffusion coefficients are independent of the controls. Here, we consider a Stackelberg game with a dominating player, which differs from the mean field game with major and minor players in that the dominating player moves first and the followers respond afterward, 
which makes it more useful in economics, finance and engineering; and we include the time delay together with the control-dependence  diffusion coefficient, which significantly complicates the analysis and reduces the convergence rate. Besides, we also refer to \cite{MR3857463,Nourian-Caines} for the  $\mathcal{O}(1/\sqrt{N})$ convergence rate for the case when the coefficients are of the form  \eqref{intr_1}; and also \cite{Bensoussan-Chau-Yam,huang2025nonlinear} for an $o(1)$ convergence  rate  with constant diffusion. In \cite{Bensoussan-Chau-Yam},  it is anticipated that  unlike the common considerations in the existing literature, this convergence should not be (and should be slower than) $\mathcal{O}(1/\sqrt{N})$. In this article, we provide a precise answer with a rigorous proof.

\subsection{Convergence for state process under optimal control}\label{sec:converge:optimal}
In this subsection, we  examine how the state process of the $N$-player game under $u^{i,\delta_i}_1[v_0],i=1,2\cdots,N$ deviates from that of the limiting counterpart. The following proposition gives the moment estimate for $y_0^{\mathbf{u}[v_0]}-x_0^{v_0}$ and $y_1^{j,\delta,\mathbf{u}[v_0]}-x_1^{j,\delta,v_0}$.

\begin{lemma}\label{prop:rate:optimal}
   Under Assumptions~\ref{assumption:initial}-(i) and \ref{assump:lip} and Hypothesis~\ref{assumption:integral}-(i), the state processes of the $N$-player game and their mean field limits corresponding to the control $\mathbf{u}[v_0]$ satisfy the following estimate:
    \begin{align}
	&\left\|y_0^{\mathbf{u}[v_0]}-x_0^{v_0}\right\|^2_{\mathcal{S}^2} + \sup_{\delta\in[a,b]} \left\|y_1^{i,\delta,\mathbf{u}[v_0]}-x_1^{i,\delta,v_0}\right\|^2_{\mathcal{S}^2} \nonumber\\
    \leq\ & C(T,L)\int_0^T \mE \bigg[W_2^2 \bigg(\text{\footnotesize$\frac {1}{N-1} {\sum_{j\neq i} \delta_{x_1^{j, \Delta_j,v_0}(s)}}$},\ z[v_0](s)\bigg)\bigg]\md s+\frac{1}{N}C(T,L)\sup_{\delta\in[a,b]} \left\|x_1^{i, \delta,v_0}\right\|^2_{\mathcal{S}^2}.\label{eq:approx:optimal}
\end{align}
\end{lemma}

\begin{proof}
    By standard estimate of SDE (see \cite{MR3793166,AB_book,MR3674558} for instance) , we have 
\begin{align}
    &\left\|y_0^{\mathbf{u}[v_0]}-x_0^{v_0}\right\|^2_{\mathcal{S}^2(0,t)} + \sup_{\delta\in[a,b]} \left\|y_1^{i,\delta,\mathbf{u}[v_0]}-x_1^{i,\delta,v_0}\right\|^2_{\mathcal{S}^2(0,t)} \notag \\
    \leq{}& C(T,L)\bigg\{ \int_0^t \bigg( \left\|y_0^{\mathbf{u}[v_0]}-x_0^{v_0}\right\|^2_{\mathcal{S}^2(0,s)} + \sup_{\delta\in[a,b]} \left\|y_1^{i,\delta,\mathbf{u}[v_0]}-x_1^{i,\delta,v_0}\right\|^2_{\mathcal{S}^2(0,s)} \bigg) \md s  \notag \\
    &  + \int_0^t \bigg( \mE\bigg[W_2^2 \bigg( \text{\footnotesize$\frac{1}{N} \sum_{j=1}^N \delta_{y_1^{j, \Delta_j,\mathbf{u}[v_0]}(s)}$},\ z[v_0](s) \bigg)\bigg] + \mE\bigg[W_2^2 \bigg( \text{\footnotesize$\frac{1}{N-1} \sum_{\substack{ j \neq i}} \delta_{y_1^{j, \Delta_j,\mathbf{u}[v_0]}(s)}$},\  z[v_0](s) \bigg)\bigg] \bigg) \md s \bigg\}. \label{add-6}
\end{align}
Note that via the use of the coupling technique (see \cite{bensoussan2016mean,AB_book,carmona2018probabilistic}), we have
\begin{align}
    &\mE\bigg[	W_2^2 \bigg(\text{\footnotesize$\frac 1N {\sum_{j=1}^N \delta_{y_1^{j, \Delta_j,\mathbf{u}[v_0]}(s)}}$},\ \text{\footnotesize$\frac 1N {\sum_{j=1}^N \delta_{x_1^{j, \Delta_j,v_0}(s)}}$}\bigg)\bigg]\nonumber\\
    \leq\ & \mE\bigg[\int_{\mR^{n_1}\times\mR^{n_1}}|y-x|^2\md\bigg(\text{\footnotesize$\frac1N\sum_{j=1}^N\delta_{\left(y_1^{j, \Delta_j,\mathbf{u}[v_0]}(s),x_1^{j, \Delta_j,v_0}(s)\right)}(y,x)$}\bigg) \bigg]\nonumber\\
    =\ & \mE\bigg[\text{\footnotesize$\frac1N \sum_{j=1}^N\left|y_1^{j, \Delta_j,\mathbf{u}[v_0]}(s)-x_1^{j, \Delta_j,v_0}(s)\right|^2$}\bigg]\nonumber\\
    =\ & \left\|y_1^{i, \Delta_i,\mathbf{u}[v_0]}(s)-x_1^{i, \Delta_i,v_0}(s)\right\|^2_2\nonumber\\
    \le\ & \sup_{\delta\in[a,b]} \left\|y_1^{i,\delta,\mathbf{u}[v_0]}-x_1^{i,\delta,v_0}\right\|^2_{\mathcal{S}^2(0,s)},\label{eq:w2:conv1}
\end{align}
where the second equality uses symmetry. Combined with the triangle inequality
\begin{align*}
 W_2 \bigg(\text{\footnotesize$\frac 1N {\sum_{j=1}^N \delta_{y_1^{j, \Delta_j,\mathbf{u}[v_0]}(s)}}$},\ z[v_0](s)\bigg)\leq \ & W_2 \bigg(\text{\footnotesize$\frac 1N {\sum_{j=1}^N \delta_{y_1^{j, \Delta_j,\mathbf{u}[v_0]}(s)}},\ \frac 1N {\sum_{j=1}^N \delta_{x_1^{j, \Delta_j,v_0}(s)}}$}\bigg)\\ &+W_2 \bigg(\text{\footnotesize$\frac 1N {\sum_{j=1}^N \delta_{x_1^{j, \Delta_j,v_0}(s)}}$},\ z[v_0](s)\bigg),
\end{align*}
we obtain
\begin{align*}
	&\left\|y_0^{\mathbf{u}[v_0]}-x_0^{v_0}\right\|^2_{\mathcal{S}^2(0,t)} + \sup_{\delta\in[a,b]} \left\|y_1^{i,\delta,\mathbf{u}[v_0]}-x_1^{i,\delta,v_0}\right\|^2_{\mathcal{S}^2(0,t)}\\\leq\ & C(T,L)\int_0^t \bigg\{ \mE\bigg[W_2^2 \bigg( \text{\footnotesize$\frac{1}{N} \sum_{j=1}^N \delta_{x_1^{j, \Delta_j,v_0}(s)}$}, z[v_0](s) \bigg)\bigg] + \mE\bigg[W_2^2 \bigg( \text{\footnotesize$\frac{1}{N-1} \sum_{\substack{ j \neq i}} \delta_{x_1^{j, \Delta_j,v_0}(s)}$},\  z[v_0](s) \bigg)\bigg] \bigg\} \md s.
\end{align*}
Moreover, note that
\begin{align}
	&\mE\bigg[W^2_2 \bigg(\text{\footnotesize$\frac 1N {\sum_{j=1}^N \delta_{x_1^{j, \Delta_j,v_0}(s)}},\ \frac {1}{N-1} {\sum_{j\neq i} \delta_{x_1^{j, \Delta_j,v_0}(s)}}\bigg)$}\bigg]\nonumber\\
    =\ &\mE\bigg[W^2_2 \bigg(\text{\footnotesize$ \frac {N-1}{N}\frac{1}{N-1} {\sum_{j\neq i} \delta_{x_1^{j, \Delta_j,v_0}(s)}}+\frac1N\delta_{x_1^{i, \Delta_j,v_0}(s)},\ \frac {1}{N-1} {\sum_{j\neq i} \delta_{x_1^{j, \Delta_j,v_0}(s)}}\bigg)$}\bigg]\nonumber\\ 
    \leq\ & \frac 1N\mE\bigg[W^2_2 \bigg(\delta_{x_1^{i, \Delta_i,v_0}(s)},\ \text{\footnotesize$\frac {1}{N-1} {\sum_{j\neq i} \delta_{x_1^{j, \Delta_j,v_0}(s)}}\bigg)$}\bigg]\nonumber\\=\ &\frac{1}{N}\frac{1}{N-1}\text{\footnotesize$\sum_{j\neq i}\left\|x_1^{j, \Delta_j,v_0}(s)-{x}_1^{i, \Delta_i,v_0}(s)\right\|_2^2$}\label{eq:w2:con2}\\
    \leq \ &\frac{4}{N}\sup_{\delta\in[a,b]} \left\|x_1^{i, \delta,v_0}\right\|^2_{\mathcal{S}^2}.\nonumber
\end{align}
where the first inequality we have used \eqref{lma:convex:wa:3} in Lemma \ref{lma:convex:wa}.
Combined with another triangle inequality
\begin{align*}
	W_2 \bigg(\text{\footnotesize$\frac 1N {\sum_{j=1}^N \delta_{x_1^{j, \Delta_j,v_0}(s)}}$},\ z[v_0](s)\bigg)\le\ & W_2 \bigg(\text{\footnotesize$\frac {1}{N-1} {\sum_{j\neq i} \delta_{x_1^{j, \Delta_j,v_0}(s)}}$},\ z[v_0](s)\bigg)\\
    &+W_2 \bigg(\text{\footnotesize$\frac 1N {\sum_{j=1}^N \delta_{x_1^{j, \Delta_j,v_0}(s)}},\ \frac {1}{N-1} {\sum_{j\neq i}\delta_{x_1^{j, \Delta_j,v_0}(s)}}$}\bigg),
\end{align*}
we get \eqref{eq:approx:optimal}.
\end{proof}

In view of Lemma \ref{prop:rate:optimal}, to get the  convergence rate for $\left\|y_0^{\mathbf{u}[v_0]}-x_0^{v_0}\right\|^2_{\mathcal{S}^2} + \sup\limits_{\delta\in[a,b]} \left\|y_1^{i,\delta,\mathbf{u}[v_0]}-x_1^{i,\delta,v_0}\right\|^2_{\mathcal{S}^2}$ as $N\to \infty$,  we only need to compute the convergence rate for
\begin{align}\label{eq:conver:w_2}
	\mE \bigg[\int_0^T W_2^2 \bigg(\text{\footnotesize$\frac {1}{N-1} {\sum_{j\neq i}\delta_{x_1^{j, \Delta_j,v_0}(s)}}$},\ z[v_0](s)\bigg)\md s\bigg].
\end{align}
Next, we shall examine the convergence rate for \eqref{eq:conver:w_2} in three cases in order: 
\begin{enumerate}[(1)]
    \item $\Delta$ follows a degenerate (single-point) distribution;  
    \item $\Delta$ has a discrete distribution;
    \item $\Delta$ follows a general distribution.
\end{enumerate}
	
\subsubsection{$\Delta\equiv a\equiv b$}
Under Assumptions~\ref{assumption:initial}-(i)  and \ref{assump:lip} and Hypothesis~\ref{assumption:integral}-(i),  from \eqref{def:z'}, we know that in this case, $z[v_0]$ writes
\begin{align*}
    z[v_0](s)(\omega)={\mathbb{P}^{\text{\tiny $\mathcal{F}^0_{s-a}$}}_{\text{\tiny$x_1^{j, a,v_0}(s)$}}}(\omega,\cdot), \qquad \omega\in \Omega.
\end{align*} 
Given the filtration $\mathcal{F}_{s-a}^0$, note the fact that $\left\{x_1^{j, \Delta_j,v_0}(s)=x_1^{j, a,v_0}(s),1\leq j\leq N,j\neq i\right\}$ are iid with the common regular conditional distribution $\mathbb{P}^{\text{\tiny$\mathcal{F}^0_{s-a}$}}_{\text{\tiny $x_1^{j, a,v_0}(s)$}}$, then, by Lemma \ref{lemma:wasserstien} we have
\begin{align*}
	\mE\left[W_2^2 \bigg(\text{\footnotesize$\frac {1}{N-1} {\sum_{j\neq i} \delta_{x_1^{j, \Delta_j,v_0}(s)}}$},\ z[v_0](s)\bigg)\bigg|\mathcal{F}_{s-a}^0\right]\leq C(n_1,q)\left(\mE\left[\left|x_1^{j, a,v_0}(s)\right|^q\Big|\mathcal{F}_{s-a}^0\right]\right)^{\frac2q}f(N-1),
\end{align*}
and therefore,
\begin{align}		
	\mE\bigg[ W_2^2 \bigg(\text{\footnotesize$\frac {1}{N-1} {\sum_{j\neq i} \delta_{x_1^{j, a,v_0}(s)}}$},\ z[v_0](s)\bigg)\bigg]&\leq C(n_1,q)\mE\left[\left(\mE\left[\left|x_1^{j, a,v_0}(s)\right|^q\Big|\mathcal{F}_{s-a}^0\right]\right)^{\frac2q}\right]f(N-1)\nonumber\\
	&\leq C(n_1,q)\left\|x_1^{j, a,v_0}(s)\right\|_q^2f(N-1),\label{eq:convergence:single}
\end{align}
where the second inequality is a consequence of H\"older inequality.

\subsubsection{Discrete  random variable $\Delta$}

We now consider the case when the random variable $\Delta$ is discrete. Suppose that 
\begin{equation}\label{setting:Delta:discrete}
\begin{aligned}
    &\Delta\in\{a_k,k=0,1,\cdots,n\},\qquad a=a_0<a_1<\cdots<a_n=b,\\
    &\text{where\quad}  \mP(a_k)=p_k,\ 0\le k\le n,\quad \text{and\quad } \sum_{k=0}^np_k=1.
\end{aligned}
\end{equation}
For this case, from \eqref{def:z'}, we know that $z[v_0]$ writes
\begin{align*}
	z[v_0](s)=\sum_{k=0}^np_k \; {\mathbb{P}^{\text{\tiny$\mathcal{F}^0_{s-a_k}\vee \F^z_s$}}_{\text{\tiny $x_1^{1, a_k,v_0}(s)$}}}.
\end{align*} 
We give the estimate on \eqref{eq:conver:w_2} corresponding to this case.  

\begin{lemma}\label{prop:convergence:discrete}
 Under Assumptions~\ref{assumption:initial}-(i) and  \ref{assump:lip} and Hypothesis~\ref{assumption:integral}-(i), we also suppose that $\de$ satisfies \eqref{setting:Delta:discrete}. Then, we have the following  convergence rate for \eqref{eq:conver:w_2}:
    \begin{align}
	&\mE\bigg[W_2^2 \bigg(\text{\footnotesize$\frac {1}{N-1} {\sum_{j\neq i} \delta_{x_1^{j, \Delta_j,v_0}(s)}}$},\ z[v_0](s)\bigg)\bigg] \notag \\
    \leq\ & C(n_1,q,L,T,\xi_0,\xi_1)\left(1+\|v_0\|_{\mathcal{L}^q}^2+\sup_{\delta\in[a,b]}\|u_1^{i,\delta}[v_0]\|_{\mathcal{L}^q}^2\right)f(N-1)\sum_{k=0}^{n}\sqrt{p_k}. \label{eq:rate:discrete}
\end{align}
\end{lemma}

\begin{proof}
We denote $M := (M_0, M_1, \cdots, M_n)$ to be the multinomial random variable so that $M_k$ counts the number of players in the $k$-th hysteresis group. Given $M$, by permutation symmetry, we can re-index the players without altering the conditional expectation. Hence, without loss of generality, we can assume the ﬁrst $M_0$ players have $\Delta=a_0$ . Then, the next $M_1$ players have $\Delta=a_1$, and so on. Thus,
\begin{align}
	&\mE\bigg[W_2^2 \bigg(\text{\footnotesize$\frac {1}{N-1} {\sum_{j\neq i} \delta_{x_1^{j, \Delta_j,v_0}(s)}}$},\ z[v_0](s)\bigg)\bigg]=	\mE\bigg[\mE\bigg[W_2^2 \bigg(\text{\footnotesize$\frac {1}{N-1} {\sum_{j\neq i} \delta_{x_1^{j, \Delta_j,v_0}(s)}}$},\ z[v_0](s)\bigg)\bigg|M\bigg]\bigg] \notag \\
	=\ &\mE\bigg[\mE\bigg[W_2^2\text{\footnotesize$ \bigg({\sum_{k=0}^n\frac {M_k}{N-1} \frac{1}{M_k} \sum_{ M_k}\delta_{x_1^{j, a_k,v_0}(s)}},\ \sum_{k=0}^n \text{ \normalsize$p_k\,\mathbb{P}$}^{\mathcal{F}^0_{s-a_k}\vee F^z_s}_{x_1^{1, a_k,v_0}(s)}$}\bigg)\bigg|M\bigg]\bigg] \notag \\
    \leq \ &2\mE\bigg[\mE\bigg[W_2^2 \bigg(\text{\footnotesize${\sum_{k=0}^n\frac {M_k}{N-1} \frac{1}{M_k} \sum_{ M_k}\delta_{x_1^{j, a_k,v_0}(s)}},\ \sum_{k=0}^n\frac {M_k}{N-1} \text{\normalsize$\mathbb{P}$}^{\mathcal{F}^0_{s-a_k}\vee F^z_s}_{x_1^{1, a_k,v_0}(s)}$}\bigg)\bigg|M\bigg]\bigg] \notag \\
    &\qquad+2\mE\bigg[\mE\bigg[W_2^2 \bigg(\text{\footnotesize$\sum_{k=0}^n\frac {M_k}{N-1} \text{\normalsize$\mathbb{P}$}^{\mathcal{F}^0_{s-a_k}\vee F^z_s}_{x_1^{1, a_k,v_0}(s)},\ \sum_{k=0}^n \text{\normalsize$p_k\;\mathbb{P}$}^{{\mathcal{F}^0_{s-a_k}\vee F^z_s}}_{x_1^{1, a_k,v_0}(s)}$}\bigg)\bigg|M\bigg]\bigg]. \label{add-1}
\end{align}
For the first term of the right hand side of \eqref{add-1}, we have the following estimate
\begin{align}
	&\mE\bigg[\mE\bigg[W_2^2 \bigg(\text{\footnotesize${\sum_{k=0}^n\frac {M_k}{N-1} \frac{1}{M_k} \sum_{ M_k}\delta_{x_1^{j, a_k,v_0}(s)}},\ \sum_{k=0}^n\frac {M_k}{N-1} \text{\normalsize$\mathbb{P}$}^{\mathcal{F}^0_{s-a_k}\vee F^z_s}_{x_1^{1, a_k,v_0}(s)}$}\bigg)\bigg|M\bigg]\bigg] \notag \\
	\leq\ & 	\mE\bigg[\text{\footnotesize$\sum_{k=0}^{n}\frac {M_k}{N-1}\mE\bigg[W_2^2 \bigg( \frac{1}{M_k} \sum_{ M_k}\delta_{x_1^{j, a_k,v_0}(s)},\ \text{\normalsize$\mathbb{P}$}^{\mathcal{F}^0_{s-a_k}\vee F^z_s}_{x_1^{1, a_k,v_0}(s)}\bigg)$}\bigg|M\bigg]\bigg] \notag \\
	\leq\ & C(n_1,q)\sup_{\delta\in[a,b]}\left\|x_1^{j, \delta,v_0}(s)\right\|_q^2\mE\bigg[\text{\footnotesize$\sum_{k=0}^{n}\frac {M_k}{N-1}f(M_k)$}\bigg], \label{add-3}
\end{align}
 where the first inequality is a direct consequence of  \eqref{lma:convex:wa:2} in Lemma \ref{lma:convex:wa},  and the second inequality follows from a similar approach as \eqref{eq:convergence:single}. 
Note that $\mE [M_k]=(N-1)p_k$, by using Jensen's inequality  and the definition of the function $f$ in \eqref{def:function:f}, we have
\begin{align}
    \mE\bigg[\sum_{k=0}^{n} M_k f(M_k)\bigg]\le\ & 
    \left\{ 
    \begin{aligned}   \sum_{k=0}^{n}\sqrt{\mE[M_k]}=\ & \sqrt{N-1}\left(\sum_{k=0}^{n}\sqrt{p_k}\right), \qquad\qquad\qquad n_1<4;\\
    \log (N-1) \mE\left[\sum_{k=0}^{n}\sqrt{M_k}\right]\leq\ &  \left(\sum_{k=0}^{n}\sqrt{p_k}\right) \sqrt{N-1} \log (N-1),\quad n_1=4;\\
    \sum_{k=0}^{n}(\mE[M_k])^{1-\frac{2}{n_1}}=\ & (N-1)^{1-\frac{2}{n_1}}\bigg(\sum_{k=0}^{n}p_k^{\text{\footnotesize$1-\frac{2}{n_1}$}}\bigg), \ \ \ \qquad n_1>4,
    \end{aligned}
    \right. \notag \\
    =\ &  (N-1)f(N-1) \cdot \begin{cases}
	\sum_{k=0}^{n}\sqrt{p_k},\quad &n_1\leq 4;\\
	\sum_{k=0}^{n}p_k^{1-\frac{2}{n_1}},\quad &n_1>4.
    \end{cases} \label{add-4}
\end{align}
For the second term of the right hand side of \eqref{add-1}, by applying Lemma~\ref{lma:pq:mu} with $q_k:=\frac{M_k}{N-1}$ and $\mu_k:=\text{\normalsize$\mathbb{P}$}^{\mathcal{F}^0_{s-a_k}\vee F^z_s}_{x_1^{1, a_k,v_0}(s)}$, we have
\begin{align}
	&\mE\bigg[\mE\bigg[W_2^2 \bigg(\text{\footnotesize$\sum_{k=0}^n\frac {M_k}{N-1} \text{\normalsize$\mathbb{P}$}^{\mathcal{F}^0_{s-a_k}\vee F^z_s}_{x_1^{1, a_k,v_0}(s)},\ \sum_{k=0}^n \text{\normalsize$p_k\; \mathbb{P}$}^{\mathcal{F}^0_{s-a_k}\vee F^z_s}_{x_1^{1, a_k,v_0}(s)}$}\bigg)\bigg|M\bigg]\bigg] \notag \\
    =\ & \mE\bigg[\mE\bigg[W_2^2 \bigg(\sum_{k=0}^nq_k\mu_k,\ \sum_{k=0}^np_k\mu_k\bigg)\bigg|M\bigg]\bigg] \notag \\
    \leq\ & \mE\bigg[\sum_{h\neq l}\hat{\pi}_{hl}\mE\bigg[W_2^2 (\mu_{h},\mu_{l})\bigg|M\bigg]\bigg] \notag \\
	=\ & \mE\bigg[\sum_{h\neq l}\hat{\pi}_{hl}\mE\left[W_2^2 \left( \mu_h,\mu_l\right))\right]\bigg], \label{add-2}
\end{align}
 where $\hat{\pi}_{hl}$ is defined in \eqref{def:pi_ij}.  Note that for $h\neq l$, we have
	\begin{align*}
	\mE\left[W_2^2 \left( \mu_h,\mu_l\right))\right]=	&\mE\bigg[W_2^2 \bigg(\text{\footnotesize$ \text{\normalsize$\mathbb{P}$}^{\mathcal{F}^0_{s-a_h}\vee F^z_s}_{x_1^{1, a_h,v_0}(s)},\  \text{\normalsize$\mathbb{P}$}^{\mathcal{F}^0_{s-a_l}\vee F^z_s}_{x_1^{1, a_l,v_0}(s)}$}\bigg)\bigg]\\
    \le\ &  2\mE\left[W_2^2 \left(\text{\footnotesize$ \text{\normalsize$\mathbb{P}$}^{\mathcal{F}^0_{s-a_h}\vee F^z_s}_{x_1^{1, a_h,v_0}(s)}$},\  \delta_0\right)\right]+2\mE\left[W_2^2 \left( \delta_0,\ \text{\footnotesize$ \text{\normalsize$\mathbb{P}$}^{\mathcal{F}^0_{s-a_l}\vee F^z_s}_{x_1^{1, a_l,v_0}(s)}$}\right)\right]\\
	=\ &2\mE\left[\left|x_1^{1, a_h,v_0}(s)\right|^2+\left|x_1^{1, a_l,v_0}(s)\right|^2\right]
		\leq 4 \sup_{\delta\in[a,b]}\left\|x_1^{i, \delta,v_0}(s)\right\|_2^2.
	\end{align*}
 Substituting the last inequality and \eqref{eq:p_ii} back into \eqref{add-2}, we get
\begin{align*}
	&\mE\bigg[\mE\bigg(W_2^2 \bigg(\text{\footnotesize$\sum_{k=0}^n\frac {M_k}{N-1} \text{\normalsize$\mathbb{P}$}^{\mathcal{F}^0_{s-a_k}\vee F^z_s}_{x_1^{1, a_k,v_0}(s)},\ \sum_{k=0}^n\text{\normalsize$p_k\;\mathbb{P}$}^{\mathcal{F}^0_{s-a_k}\vee F^z_s}_{x_1^{1, a_k,v_0}(s)}$}\bigg)\bigg|M\bigg)\bigg]\\	
    \leq \ &2\sup_{\delta\in[a,b]}\left\|x_1^{i, \delta,v_0}(s)\right\|_2^2\mE\bigg[\text{\footnotesize$\sum_{k=0}^n\bigg|\frac{M_k}{N-1}-p_k\bigg|$}\bigg]
	\\
    \leq \ &2\sup_{\delta\in[a,b]}\left\|x_1^{i, \delta,v_0}(s)\right\|_2^2 \frac{1}{\sqrt{N-1}}\sum_{k=0}^{n}\sqrt{p_k(1-p_k)}.
\end{align*}
 Combining the last estimate with \eqref{add-1}, \eqref{add-3} and \eqref{add-4}, we have
\begin{align*}
	&\mE\bigg[W_2^2 \bigg(\text{\footnotesize$\frac {1}{N-1} {\sum_{j\neq i}^N \delta_{x_1^{j, \Delta_j,v_0}(s)}},\ z[v_0](s)$}\bigg)\bigg] \\\leq \ &2\sup_{\delta\in[a,b]}\left\|x_1^{i, \delta,v_0}(s)\right\|_2^2 \frac{1}{\sqrt{N-1}}\sum_{k=0}^{n}\sqrt{p_k(1-p_k)}\\
    &+ C(n_1,q)\sup_{\delta\in[a,b]}\left\|x_1^{i, \delta,v_0}(s)\right\|_q^2f(N-1)\cdot \begin{cases}
	\sum_{k=0}^{n}\sqrt{p_k},\quad &n_1\leq 4;\\
	\sum_{k=0}^{n}p_k^{1-\frac{2}{n_1}},\quad &n_1>4.
    \end{cases}\\
    \leq \ &C(n_1,q)\sup_{\delta\in[a,b]}\left\|x_1^{i, \delta,v_0}(s)\right\|_q^2f(N-1)\sum_{k=0}^{n}\sqrt{p_k}\\
    \leq \ &C(n_1,q,L,T,\xi_0,\xi_1)\left(1+\|v_0\|_{\mathcal{L}^q}^2+\sup_{\delta\in[a,b]}\left\|u_1^{i,\delta}[v_0]\right\|_{\mathcal{L}^q}^2\right)f(N-1)\sum_{k=0}^{n}\sqrt{p_k},
\end{align*}
where the last inequality uses \eqref{estimate:SDE:Lq-1}. Thus, the claim follows.
\end{proof}

\begin{remark}
    The proof of Lemma~\ref{prop:convergence:discrete} makes use of Lemma \ref{lemma:wasserstien}, therefore, the convergence rate $f(N-1)$ with respect to $N$ in \eqref{eq:rate:discrete} is optimal; while the  constant ( $\sum_{k=0}^n\sqrt{p_k}$) related to $n$ and $\{p_k,0\le k\leq n\}$ may not be optimal. 
\end{remark}

\subsubsection{General $\Delta$}\label{subsubsec:General_Delta}
In this subsubsection, we consider $\Delta$ being arbitrarily distributed on $[a,b]$ with a measure $\pi_{\Delta}$. 

\begin{lemma}\label{prop:general}
	Under Assumptions \ref{assumption:initial}-\ref{assump:lip} and Hypothesis \ref{assumption:integral}, for any $s\in[0,T]$, we have
\begin{align}
                &\mE\bigg[ W_2^2 \bigg(\text{\footnotesize$\frac {1}{N-1} {\sum_{j\neq i}\delta_{x_1^{j, \Delta_j,v_0}(s)}}$},\ z[v_0](s)\bigg)\bigg]\nonumber\\
                \leq\ & C(n_1,q,L,T,\xi_0,\xi_1)\bigg(1+l_{u_1[v_0]}+\|v_0\|_{\mathcal{L}^q}^2+\sup_{\delta\in[a,b]}\left\|u_1^{i,\delta}[v_0]\right\|_{\mathcal{L}^q}^2\bigg)(f(N-1))^{\frac{2q-4}{3q-4}}.\label{eq:conver:optimal} 
     \end{align}
\end{lemma}

\begin{proof}
	Let $\left\{a_k^{(n)},\ k=0,\dots,n\right\}$ be the level $n$ uniform partition on $[a,b]$, i.e., $a_k^{(n)}:= a + \frac{k}{n}(b-a)$; here, the parameter $n$ will be chosen as a function of $N$ later to optimize the convergence rate as $N\to\infty$. Let $M^{(n)} = \left(M_1^{(n)}, M_2^{(n)}, \ldots, M_n^{(n)}\right)$ be the multinomial random variable on $\left\{a_0^{(n)}, a_1^{(n)}, \ldots, a_{n-1}^{(n)}\right\}$ with event probabilities $p_k^{(n)} := \pi_{\Delta}\left(\left[a_{k-1}^{(n)}, a_k^{(n)}\right)\right)$ (hence $\sum_{k=1}^n p_k^{(n)} = 1$)\footnote{Here we abuse the notations, $[a_{n-1}^{(n)}, a_n^{(n)}):=[a_{n-1}^{(n)}, a_n^{(n)}]$ for the very last subinterval.}. 
	We  can write
    \begin{align*}
        \frac {1}{N-1} {\sum_{j\neq i} \delta_{x_1^{j, \Delta_j,v_0}(s)}}=\ & \frac {1}{N-1}  \sum_{j\neq i}\sum_{k=1}^n\ind_{\Delta_j \in \left[a_{k-1}^{(n)}, a_k^{(n)}\right)}{ \delta_{x_1^{j, \Delta_j,v_0}(s)}}\\
        =\ & \frac {1}{N-1} \sum_{k=1}^n \sum_{j\neq i}\ind_{\Delta_j \in \left[a_{k-1}^{(n)}, a_k^{(n)}\right)}{ \delta_{x_1^{j, \Delta_j,v_0}(s)}}.
    \end{align*}
    Therefore, we have
    \begin{align}
		&\mE\bigg[ W_2^2 \bigg(\text{\footnotesize$\frac {1}{N-1} {\sum_{j\neq i} \delta_{x_1^{j, \Delta_j,v_0}(s)}}$},z[v_0](s)\bigg)\bigg] \notag \\
		=\ &  \mathbb{E} \bigg[ \mathbb{E} \bigg[W_2^2 \bigg(\text{\footnotesize$\frac {1}{N-1} \sum_{k=1}^n \sum_{j\neq i}\ind_{\Delta_j \in \left[a_{k-1}^{(n)}, a_k^{(n)}\right)}{ \delta_{x_1^{j, \Delta_j,v_0}(s)}}$},\ z[v_0](s) \bigg) \bigg| M^{(n)} \bigg] \bigg] \notag \\
		\leq\ &3  \mathbb{E} \Bigg[\mathbb{E} \Bigg[W_2^2 \Bigg(\text{\footnotesize$ \frac {1}{N-1} \sum_{k=1}^n \sum_{j\neq i}\ind_{\Delta_j \in \left[a_{k-1}^{(n)}, a_k^{(n)}\right)}{ \delta_{x_1^{j, \Delta_j,v_0}(s)}}, \ \frac {1}{N-1} \sum_{k=1}^n \sum_{j\neq i}\ind_{\Delta_j \in \left[a_{k-1}^{(n)}, a_k^{(n)}\right)}{ \delta_{x_1^{j, a_{k-1}^{(n)},v_0}(s)}}$}\Bigg) \Bigg| M^{(n)} \Bigg] \Bigg] \notag \\
		&+3 \mathbb{E} \Bigg[ \mathbb{E} \Bigg[W_2^2 \Bigg(\text{\footnotesize$ \frac {1}{N-1} \sum_{k=1}^n \sum_{j\neq i}\ind_{\Delta_j \in \left[a_{k-1}^{(n)}, a_k^{(n)}\right)}{ \delta_{x_1^{j, a_{k-1}^{(n)},v_0}(s)}}$},\ z^{(n)}(s)\Bigg) \Bigg| M^{(n)} \Bigg] \Bigg] \notag \\
		&+3  \mathbb{E} \left[\mathbb{E} \left[W_2^2 \left( z^{(n)}[v_0](s),z[v_0](s)\right) \middle| M^{(n)} \right] \right], \label{add-5}
	\end{align}
where 
\begin{align*}
    z^{(n)}[v_0](s)  := \sum_{k=1}^n p_k^{(n)}\; \mathbb{P}^{\text{\tiny $\mathcal{F}_{s-a_{k-1}^{(n)}}^0 \vee \mathcal{F}_s^z$}}_{\text{\tiny $x_1^{1,a_{k-1}^{(n)},v_0} (s)$}}.
\end{align*}
For the first term on the right hand side of \eqref{add-5}, we note that 
\begin{align*}
	&\e\Bigg[\mathbb{E} \Bigg[W_2^2 \Bigg( \text{\footnotesize$\frac {1}{N-1} \sum_{k=1}^n \sum_{j\neq i}\ind_{\Delta_j \in \left[a_{k-1}^{(n)}, a_k^{(n)}\right)}{ \delta_{x_1^{j, \Delta_j,v_0}(s)}}, \ \frac {1}{N-1} \sum_{k=1}^n \sum_{j\neq i}\ind_{\Delta_j \in \left[a_{k-1}^{(n)}, a_k^{(n)}\right)}{ \delta_{x_1^{j, a_{k-1}^{(n)},v_0}(s)}}$}\Bigg) \Bigg| M^{(n)} \Bigg] \Bigg]\\
    \leq\ &\e\Bigg[\sum_{k=1}^n\frac{M^{(n)}_k}{N-1}\mathbb{E} \Bigg[W_2^2 \Bigg(\text{\footnotesize$ \frac{1}{M^{(n)}_k}\sum\limits_{j\neq i} \ind_{\Delta_j \in \left[a_{k-1}^{(n)}, a_k^{(n)}\right)}{ \delta_{x_1^{j, \Delta_j,v_0}(s)}} ,\ \frac{1}{M^{(n)}_k} \sum\limits_{j\neq i} \ind_{\Delta_j \in \left[a_{k-1}^{(n)}, a_k^{(n)}\right)}{ \delta_{x_1^{j, a_{k-1}^{(n)},v_0}(s)}} $} \Bigg) \Bigg| M^{(n)} \Bigg]\Bigg]\\
	\leq\  &\sum_{k=1}^n \frac{1}{N-1}\e\bigg[\text{\footnotesize$\sum_{j\neq i}$}\ind_{\Delta_j \in \left[a_{k-1}^{(n)}, a_k^{(n)}\right)} \mathbb{E} \bigg[W_2^2 \bigg(  \delta_{ \text{\tiny $x_1^{j,\Delta_j,v_0}(s)$}} , \  \delta_{\text{\tiny$x_1^{j,a_{k-1}^{(n)},v_0}(s)$}}  \bigg) \bigg| M^{(n)} \bigg] \bigg]\\
    =\  &\frac{1}{N-1}\sum_{j\neq i} \e\bigg[\mathbb{E} \bigg[\text{\footnotesize$\sum_{k=1}^n$}\ind_{\Delta_j \in \left[a_{k-1}^{(n)}, a_k^{(n)}\right)}W_2^2 \bigg(  \delta_{ \text{\tiny $x_1^{j,\Delta_j,v_0}(s)$}} , \  \delta_{\text{\tiny$x_1^{j,a_{k-1}^{(n)},v_0}(s)$}}  \bigg) \bigg| M^{(n)} \bigg]\bigg]\\
	\le\  & \frac{1}{N-1} \mE\Bigg[\sum_{j\neq i}\mE\bigg[\text{\footnotesize$\sum_{k=1}^n$}\ind_{\Delta_j \in \left[a_{k-1}^{(n)}, a_k^{(n)}\right)}\left|x_1^{j,\Delta_j,v_0}(s)-x_1^{j,a_{k-1}^{(n)},v_0}(s)\right|^2\bigg| M^{(n)} \bigg]\Bigg]\\
    \leq\  &\frac{1}{N-1} \sum_{j\neq i}  C(L,T,\xi_0,\xi^i_1,q)\bigg(1+l_{u_1[v_0]}+\|v_0\|_{\mathcal{L}^q}^2+\sup_{\delta\in[a,b]}\|u_1^{i, \delta}[v_0\|_{\mathcal{L}^2}^2\bigg){n^{-\frac{q-2}{q}}}\\
    =\ & C(L,T,\xi_0,\xi^i_1,q)\bigg(1+l_{u_1[v_0]}+\|v_0\|_{\mathcal{L}^q}^2+\sup_{\delta\in[a,b]}\|u_1^{i, \delta}[v_0\|_{\mathcal{L}^2}^2\bigg) n^{-\frac{q-2}{q}},
\end{align*}
where the first and second inequalities are the direct consequences of \eqref{lma:convex:wa:2} in Lemma \ref{lma:convex:wa},  the third inequality is due to the minimal nature of the definition of Wasserstein metric, and the last inequality is a consequence of \eqref{C.2.}.
For the second term on the right hand side of \eqref{add-5}, by using Lemma \ref{prop:convergence:discrete} and applying the inequality $\sum_{k=1}^n \sqrt{\text{\footnotesize $p^{(n)}_k$}}\leq \sqrt{n}$, we obtain
\begin{align*}
&\mathbb{E} \bigg[ \mathbb{E} \bigg[W_2^2 \bigg( \text{\footnotesize $\frac {1}{N-1} \sum_{j\neq i} \sum_{k=1}^n \ind_{\Delta_j \in \left[a_{k-1}^{(n)}, a_k^{(n)}\right)}{ \delta_{x_1^{j, a_{k-1}^{(n)},v_0}(s)}},\ z^{(n)}(s)$}\bigg) \bigg| M^{(n)} \bigg] \bigg]\\
	 \leq \ & C(n_1,q,L,T,\xi_0,\xi_1)\bigg(1+\|v_0\|_{\mathcal{L}^q}^2+\sup_{\delta\in[a,b]}\|u_1^{i,\delta}[v_0]\|_{\mathcal{L}^q}^2\bigg)f(N-1)\sqrt{n}.
\end{align*}
For the third term on the right hand side of \eqref{add-5}, we know from the definition of $z[v_0]$ and $z^{(n)}[v_0]$ that
\begin{align*}
    &\mathbb{E} \left[W_2^2 \left( z^{(n)}[v_0](s),z[v_0](s)\right)  \right]\\
    =\ & \mE \Bigg[W_2^2\Bigg(\int_{[a,b]}{\mathbb{P}^{\text{\tiny$\F_{s-\delta}^0 \vee \F_{s}^z$}}_{\text{\tiny$x_1^{1,\delta,v_0}(s)$}}  } \md\pi_{\Delta}(\delta),\ \sum_{k=1}^n p_k^{(n)}\; \mathbb{P}_{\text{\tiny$x_1^{1,a_{k-1}^{(n)},v_0} (s)$}}^{\text{\tiny$\mathcal{F}_{s-a_{k-1}^{(n)}}^0 \vee \mathcal{F}_s^z$}}\Bigg)\Bigg]\\
    =\ &\mE \Bigg[W_2^2\Bigg( \sum_{k=1}^n\int_{[a^{(n)}_{k-1},a^{(n)}_{k})} \text{\footnotesize$\text{\normalsize$\mathbb{P}$}^{\F_{s-\delta}^0 \vee \F_{s}^z}_{x_1^{1,\delta,v_0}(s)} $}\md\pi_{\Delta}(\delta),\ \sum_{k=1}^n p_k^{(n)}\; \text{\footnotesize$\text{\normalsize$\mathbb{P}$}_{x_1^{1,a_{k-1}^{(n)},v_0} (s)}^{\mathcal{F}_{s-a_{k-1}^{(n)}}^0 \vee \mathcal{F}_s^z}$} \Bigg)\Bigg]\\
    \leq \ &2\mE \Bigg[W_2^2\Bigg( \sum_{k=1}^n\int_{[a^{(n)}_{k-1},a^{(n)}_{k})} \text{\footnotesize$\text{\normalsize$\mathbb{P}$}^{\F_{s-\delta}^0 \vee \F_{s}^z}_{x_1^{1,\delta,v_0}(s)} $}\md\pi_{\Delta}(\delta),\ \sum_{k=1}^n\int_{[a^{(n)}_{k-1},a^{(n)}_{k})}\text{\footnotesize$ \text{\normalsize$\mathbb{P}$}^{\F_{s-a^{(n)}_{k-1}}^0 \vee \F_{s}^z}_{x_1^{1,\delta,v_0}(s)}$}  \md \pi_{\Delta}(\delta)\Bigg)\Bigg]\\
    &\qquad+2\mE \Bigg[W_2^2\Bigg( \sum_{k=1}^n\int_{[a^{(n)}_{k-1},a^{(n)}_{k})} \text{\footnotesize$\text{\normalsize$\mathbb{P}$}^{\F_{s-a^{(n)}_{k-1}}^0 \vee \F_{s}^z}_{x_1^{1,\delta,v_0}(s)}$} \md\pi_{\Delta}(\delta),\ \sum_{k=1}^n p_k^{(n)}\; \text{\footnotesize$\text{\normalsize$\mathbb{P}$}_{x_1^{1,a_{k-1}^{(n)},v_0} (s)}^{\mathcal{F}_{s-a_{k-1}^{(n)}}^0 \vee \mathcal{F}_s^z}$}\Bigg)\Bigg].
\end{align*}
For $\delta \in \left[a^{(n)}_{k-1},a^{(n)}_{k} \right)$, since $x_1^{1,\delta,v_0}(s)$ is $\mathcal{F}_s^{1,i} \vee \mathcal{F}_{s-\delta}^0 \vee \mathcal{F}_s^z$ and thus $\mathcal{F}_s^{1,i} \vee \mathcal{F}_{s-a^{(n)}_{k-1}}^0 \vee \mathcal{F}_s^z$ adapted, and $\mathcal{F}^{1,i}$ and $\mathcal{F}^0\vee \mathcal{F}^z$ are independent of each other,  by the usage of \eqref{eq:property:regular}, we have
\begin{align}
 \mathbb{P}^{\text{\tiny$\F_{s-\delta}^0 \vee \F_{s}^z$}}_{\text{\tiny$x_1^{1,\delta,v_0}(s)$}} =\mathbb{P}^{\text{\tiny$\F_{s-a^{(n)}_{k-1}}^0 \vee \F_{s}^z$}}_{\text{\tiny$x_1^{1,\delta,v_0}(s)$}},\quad \forall \delta \in \left[a^{(n)}_{k-1},a^{(n)}_{k}\right),\label{eq:pro:conditional}
\end{align}
Thus, by using \eqref{lma:convex:wa:2} and \eqref{lma:convex:wa:1} in Lemma \ref{lma:convex:wa} respectively, we have
\begin{align*}
    &\mE \left[W_2^2\left( \sum_{k=1}^n\int_{[a^{(n)}_{k-1},a^{(n)}_{k})} \mathbb{P}^{\text{\tiny$\F_{s-\delta}^0 \vee \F_{s}^z$}}_{\text{\tiny$x_1^{1,\delta,v_0}(s)$}} \md\pi_{\Delta}(\delta),\ \sum_{k=1}^n\int_{[a^{(n)}_{k-1},a^{(n)}_{k})} \mathbb{P}^{\text{\tiny$\F_{s-a^{(n)}_{k-1}}^0 \vee \F_{s}^z$}}_{\text{\tiny$x_1^{1,\delta,v_0}(s)$}}  \md \pi_{\Delta}(\delta)\right)\right]\\
    \leq \ & 
    \mE \left[\sum_{k=1}^np^{(n)}_kW_2^2\left( \frac{1}{p^{(n)}_k}\int_{[a^{(n)}_{k-1},a^{(n)}_{k})} \mathbb{P}^{\text{\tiny$\F_{s-\delta}^0 \vee \F_{s}^z$}}_{\text{\tiny$x_1^{1,\delta,v_0}(s)$}} \md\pi_{\Delta}(\delta),\ \frac{1}{p^{(n)}_k}\int_{[a^{(n)}_{k-1},a^{(n)}_{k})} \mathbb{P}^{\text{\tiny$\F_{s-a^{(n)}_{k-1}}^0 \vee \F_{s}^z$}}_{\text{\tiny$x_1^{1,\delta,v_0}(s)$}}  \md \pi_{\Delta}(\delta)\right)\right]\\
    \leq \ & \mE \left[\sum_{k=1}^n\int_{[a^{(n)}_{k-1},a^{(n)}_{k})} W^2_2\left(\mathbb{P}^{\text{\tiny$\F_{s-\delta}^0 \vee \F_{s}^z$}}_{\text{\tiny$x_1^{1,\delta,v_0}(s)$}},\ \mathbb{P}_{\text{\tiny$x_1^{1,\delta,v_0} (s)$}}^{\text{\tiny$\mathcal{F}_{s-a_{k-1}^{(n)}}^0 \vee \mathcal{F}_s^z$}}\right) \md\pi_{\Delta}(\delta)\right]=0.
\end{align*}
For the another term
 \begin{align*}
 &\mE \left[W_2^2\left( \sum_{k=1}^n\int_{[a^{(n)}_{k-1},a^{(n)}_{k})}\mathbb{P}^{\text{\tiny$\F_{s-a^{(n)}_{k-1}}^0 \vee \F_{s}^z$}}_{\text{\tiny$x_1^{1,\delta,v_0}(s)$}}
 \md\pi_{\Delta}(\delta),\ \sum_{k=1}^n p_k^{(n)} \mathbb{P}_{\text{\tiny$x_1^{1,a_{k-1}^{(n)},v_0} (s)$}}^{\text{\tiny$\mathcal{F}_{s-a_{k-1}^{(n)}}^0 \vee \mathcal{F}_s^z$}}\right)\right]\\
 	\leq\ &\mE \left[\sum_{k=1}^np^{(n)}_kW_2^2\left( \frac{1}{p^{(n)}_k}\int_{[a^{(n)}_{k-1},a^{(n)}_{k})} \mathbb{P}^{\text{\tiny$\F_{s-a^{(n)}_{k-1}}^0 \vee \F_{s}^z$}}_{\text{\tiny$x_1^{1,\delta,v_0}(s)$}}
 \md\pi_{\Delta}(\delta),\  \mathbb{P}_{\text{\tiny$x_1^{1,a_{k-1}^{(n)},v_0} (s)$}}^{\text{\tiny$\mathcal{F}_{s-a_{k-1}^{(n)}}^0 \vee \mathcal{F}_s^z$}}\right)\right]\\
 	\leq\ &\mE \left[\sum_{k=1}^n\int_{[a^{(n)}_{k-1},a^{(n)}_{k})} W^2_2\left(\mathbb{P}^{\text{\tiny$\F_{s-a^{(n)}_{k-1}}^0 \vee \F_{s}^z$}}_{\text{\tiny$x_1^{1,\delta,v_0}(s)$}},\ \mathbb{P}_{\text{\tiny$x_1^{1,a_{k-1}^{(n)},v_0} (s)$}}^{\text{\tiny$\mathcal{F}_{s-a_{k-1}^{(n)}}^0 \vee \mathcal{F}_s^z$}}\right) \md\pi_{\Delta}(\delta)\right]\\
    \leq \ &\mE\left[\sum_{k=1}^n\int_{[a^{(n)}_{k-1},a^{(n)}_{k})} 
    \mE\left[\left(x_1^{1,\delta,v_0}(s)-x_1^{1,a_{k-1}^{(n)},v_0} (s)\right)^2\Bigg|\mathcal{F}_{s-a_{k-1}^{(n)}}^0 \vee \mathcal{F}_s^z\right] \md\pi_{\Delta}(\delta)\right]
    \\
 	=\ & \sum_{k=1}^n
 	 \int_{[a^{(n)}_{k-1},a^{(n)}_{k})} \mE\left[\left|x_1^{1,\delta,v_0}(s)-x_1^{1,a_{k-1}^{(n)},v_0} (s)\right|^2\right] \md\pi_{\Delta}(\delta)\\
     \leq\ & C(L,T,\xi_0,\xi^i_1,q)\bigg(1+l_{u_1[v_0]}+\|v_0\|_{\mathcal{L}^q}^2+\sup_{\delta\in[a,b]}\|u_1^{i, \delta}[v_0\|_{\mathcal{L}^2}^2\bigg)n^{-\frac{q-2}{q}},
 	 \end{align*}
     where the first  and second inequality use \eqref{lma:convex:wa:2} and \eqref{lma:convex:wa:1} in Lemma \ref{lma:convex:wa} respectively, the third inequality uses the minimal nature for Wasserstein metric and the last inequality uses \eqref{C.2.}.
     Therefore, we obtain
     \begin{align*}
         &\mE\Bigg[ W_2^2 \Bigg(\text{\footnotesize$\frac {1}{N-1} {\sum_{j\neq i}\delta_{x_1^{j, \Delta_j,v_0}(s)}}$},z[v_0](s)\Bigg)\Bigg]\\ \leq \ &2C(L,T,\xi_0,\xi^i_1,q)\bigg(1+l_{u_1[v_0]}+\|v_0\|_{\mathcal{L}^q}^2+\sup_{\delta\in[a,b]}\|u_1^{i, \delta}[v_0\|_{\mathcal{L}^2}^2\bigg)
n^{-\frac{q-2}{q}}\\\ &+C(n_1,q,L,T,\xi_0,\xi_1)\left(1+\|v_0\|_{\mathcal{L}^q}^2+\sup_{\delta\in[a,b]}\|u_1^{i,\delta}[v_0]\|_{\mathcal{L}^q}^2\right)f(N-1)\sqrt{n}.
     \end{align*}
     Taking  $n\sim (f(N-1))^{-\frac{2q}{3q-4}}$, we get \eqref{eq:conver:optimal}.
\end{proof}
Combining  Lemma \ref{prop:rate:optimal} and Lemma \ref{prop:general}, we have the following theorem:

\begin{theorem}\label{thm:converge:Ntolimit}
Under  Assumptions \ref{assumption:initial}-\ref{assump:lip} and Hypothesis \ref{assumption:integral}.
 The state processes of the $N$-player game and their mean field limits corresponding to the optimal control $\mathbf{u}[v_0]$ satisfy the following estimate:
    \begin{align*}
	&\left\|y_0^{\mathbf{u}[v_0]}-x_0^{v_0}\right\|^2_{\mathcal{S}^2} + \sup_{\delta\in[a,b]} \left\|y_1^{i,\delta,\mathbf{u}[v_0]}-x_1^{i,\delta,v_0}\right\|^2_{\mathcal{S}^2}\\ \leq \ &C(n_1,q,L,T,\xi_0,\xi_1)\bigg(1+l_{u_1[v_0]}+\|v_0\|_{\mathcal{L}^q}^2+\sup_{\delta\in[a,b]}\left\|u_1^{i,\delta}[v_0]\right\|_{\mathcal{L}^q}^2\bigg)(f(N-1))^{\frac{2q-4}{3q-4}}.
\end{align*}
\end{theorem}
\begin{remark}\label{remark:delta}
  Suppose that $\Delta\in\{a_k,k=0,1,\cdots,n\}$ with $a=a_0<a_1<\cdots<a_n=b$,  $\mP(a_k)=p_k$ and $\sum_{k=0}^np_k=1$. Under the optimal control $\mathbf{u}[v_0]$, the state processes of the $N$-player game and their corresponding mean field limits satisfy the following estimate:
    \begin{align*}
        &\left\|y_0^{\mathbf{u}[v_0]}-x_0^{v_0}\right\|^2_{\mathcal{S}^2} + \sup_{\delta\in[a,b]} \left\|y_1^{i,\delta,\mathbf{u}[v_0]}-x_1^{i,\delta,v_0}\right\|^2_{\mathcal{S}^2}\\
        \leq \ & C(n_1,q,L,T,\xi_0,\xi_1)\left(1+\|v_0\|_{\mathcal{L}^q}^2+\sup_{\delta\in[a,b]}\|u_1^{i,\delta}[v_0]\|_{\mathcal{L}^q}^2\right)f(N-1)\sum_{k=0}^{n}\sqrt{p_k}.
    \end{align*}
\end{remark}

\subsection{Convergence for state process under arbitrary control}\label{sec:converge:arbitrary}
This section is devoted to analyzing the convergence rate of the state process of the $N$-player game to that of the mean field game under the control ${\mathbf{v}^i}[v_0]$. 
The next proposition provides moment estimates for the differences
$y_0^{{\mathbf{v}^i}[v_0]}-x_0^{v_0}$, $y_1^{i,\delta_i,{\mathbf{v}^i}[v_0]}-x^{i,\delta_i,v_0,v_1}_1$ and $y_1^{j,\delta_{j},{\mathbf{v}^i}[v_0]}-x^{j,\delta_{j},v_0}_1$ for $j\neq i$.
\begin{proposition}\label{prop:estimate:non-optimal}
Suppose that  Assumptions \ref{assumption:initial}-\ref{assump:lip} and Hypothesis \ref{assumption:integral} hold. Moreover,
suppose that the $j$-th player $(j\neq i)$ adopts the optimal controls $u^{j,\delta_j}_1[v_0]$ while the $i$-th player adopts an arbitrary control $v^{i,\delta_i}_1\in \mathcal{L}_{\mathcal{G}^{i,\delta_i}}^2\left([0, T] ; \mathbb{R}^{p_1}\right)$. Then we have the following estimate:
    \begin{align*}
        &\left\|y_0^{{\mathbf{v}^i}[v_0]}-x_0^{v_0}\right\|^2_{\mathcal{S}^2}+\sup_{\delta_j\in[a,b]}\left\|y_1^{j,\delta_j,{\mathbf{v}^i}[v_0]}-x^{j,\delta_j,v_0}_1\right\|_{\mathcal{S}^2}^2 +\sup_{\delta_i\in[a,b]}\left\|y_1^{i,\delta_i,{\mathbf{v}^i}[v_0]}-x^{i,\delta_i,v_0,v_1}_1\right\|_{\mathcal{S}^2}^2\\
   \leq\ & C(n_1,q,L,T,\xi_0,\xi_1)\bigg(1+l_{u_1[v_0]}+\|v_0\|_{\mathcal{L}^q}^2+\sup_{\delta\in[a,b]}\|u_1^{i,\delta}[v_0]\|_{\mathcal{L}^q}^2\bigg)(f(N-1))^{\frac{2q-4}{3q-4}}\\
   &+\frac1N C(L,T)\sup_{\delta\in[a,b]}\left\|v_1^{i,\delta}\right\|_{\mathcal{L}^2}.
    \end{align*}
\end{proposition}
\begin{proof}
Similar to \eqref{add-6} as in the proof of Lemma \ref{prop:rate:optimal},
applying standard estimates to $y_0^{{\mathbf{v}^i}[v_0]}-x_0^{v_0}$, $y_1^{i,\delta_i,{\mathbf{v}^i}[v_0]}-x^{i,\delta_i,v_0,v_1}_1$ and $y_1^{j,\delta_{j},{\mathbf{v}^i}[v_0]}-x^{j,\delta_{j},v_0}_1$, we get
    \begin{align}
    \left\|y_0^{{\mathbf{v}^i}[v_0]}-x_0^{v_0}\right\|^2_{\mathcal{S}^2(0,t)}& \ \leq C(T,L)\int_0^t\mE\bigg[W^2_2\bigg(\text{\footnotesize$\frac 1N {\sum_{k=1}^N \delta_{y_1^{k, \Delta_k,{\mathbf{v}^i}[v_0]}(r)}}$},z[v_0](r)\bigg)\bigg]\md r\nonumber\\
    &\qquad+C(T,L)\int_0^t\bigg(\|y_0^{{\mathbf{v}^i}[v_0]}-x_0^{v_0}\|^2_{\mathcal{S}^2(0,r)}\bigg)\md r;\label{eq:proof:4.5:1}\\ \sup_{\delta_i\in[a,b]}\left\|y_1^{i,\delta_i,{\mathbf{v}^i}[v_0]}-x^{i,\delta_i,v_0,v_1}_1\right\|_{\mathcal{S}^2(0,t)}^2 & \ \leq C(T,L)\int_0^t\mE\bigg[W^2_2\bigg(\text{\footnotesize$\frac{1}{N-1}{\sum_{k \neq i}\delta_{y_1^{k, \Delta_k,{\mathbf{v}^i}[v_0]}(r)}}$},z[v_0](r)\bigg)\bigg]\md r  \nonumber\\ &\qquad + C(T,L)\int_0^t\sup_{\delta_i\in[a,b]}\|y_1^{i,\delta_i,{\mathbf{v}^i}[v_0]}-x^{i,\delta_i,v_0,v_1}_1\|_{\mathcal{S}^2(0,r)}^2\md r\nonumber\\ &\qquad + C(T,L)\int_0^t \|y_0^{{\mathbf{v}^i}[v_0]}-x_0^{v_0}\|^2_{\mathcal{S}^2(0,r)}\md r;\label{eq:proof:4.5:2}\\
    \sup_{\delta_j\in[a,b]}\left\|y_1^{j,\delta_j,{\mathbf{v}^i}[v_0]}-x^{j,\delta_j,v_0}_1\right\|_{\mathcal{S}^2(0,t)}^2 & \nonumber\ \leq C(T,L)\int_0^t\mE\bigg[W^2_2\bigg(\text{\footnotesize$\frac{1}{N-1}{\sum_{k \neq j}\delta_{y_1^{k, \Delta_k,{\mathbf{v}^i}[v_0]}(r)}}$},z[v_0](r)\bigg)\bigg]\md r \nonumber \\ &\qquad + C(T,L)\int_0^t\sup_{\delta_j\in[a,b]}\|y_1^{j,\delta_j,{\mathbf{v}^i}[v_0]}-x^{j,\delta_j,v_0}_1\|_{\mathcal{S}^2(0,r)}^2\md r\nonumber\\ &\qquad + C(T,L)\int_0^t \|y_0^{{\mathbf{v}^i}[v_0]}-x_0^{v_0}\|^2_{\mathcal{S}^2(0,r)}\md r.\label{eq:proof:4.5:3}
\end{align}
We now estimate \eqref{eq:proof:4.5:1}, \eqref{eq:proof:4.5:2} and \eqref{eq:proof:4.5:3} in order as follows:
\begin{enumerate}
\item[(1)] We estimate the Wasserstein metric in \eqref{eq:proof:4.5:1} as follows:
\begin{align*}
    &W^2_2\bigg(\text{\footnotesize$\frac 1N {\sum_{k=1}^N \delta_{y_1^{k, \Delta_k,{\mathbf{v}^i}[v_0]}(r)}}$},z[v_0](r)\bigg)\\\leq\ &  3W^2_2\bigg(\text{\footnotesize$\frac 1N {\sum_{k=1}^N \delta_{y_1^{k, \Delta_k,{\mathbf{v}^i}[v_0]}(r)}},\frac 1N {\sum_{k\neq i} \delta_{x_1^{k, \Delta_k,v_0}(r)}}+\frac1N\delta_{x_1^{i, \Delta_i,v_0,v_1}(r)}$}\bigg)\\
    &+3W^2_2\bigg(\text{\footnotesize$\frac 1N {\sum_{k\neq i} \delta_{x_1^{k, \Delta_k,v_0}(r)}}+\frac1N\delta_{x_1^{i, \Delta_i,v_0,v_1}(r)},\frac {1}{N-1} {\sum_{k\neq i} \delta_{x_1^{k, \Delta_k,v_0}(r)}}$}\bigg)\\
    &+3W^2_2\bigg(\text{\footnotesize$\frac {1}{N-1} {\sum_{k\neq i} \delta_{x_1^{k, \Delta_k,v_0}(r)}}$},z[v_0](r)\bigg).
\end{align*}
Similar to \eqref{eq:w2:conv1}, using \eqref{lma:convex:wa:2} in Lemma \ref{lma:convex:wa}, we have
\begin{align*}  &\mE\bigg[W^2_2\bigg(\text{\footnotesize$\frac 1N {\sum_{k=1}^N \delta_{y_1^{k, \Delta_k,{\mathbf{v}^i}[v_0]}(r)}},\ \frac 1N {\sum_{k\neq i} \delta_{x_1^{k, \Delta_k,v_0}(r)}}+\frac1N\delta_{x_1^{i, \Delta_i,v_0,v_1}(r)}$}\bigg)\bigg]\\
    \leq \ & \frac{N-1}{N}\sup_{\delta_j\in[a,b]}\left\|y_1^{j,\delta_j,{\mathbf{v}^i}[v_0]}-x^{j,\delta_j,v_0}_1\right\|_{\mathcal{S}^2(0,r)}^2+\frac1N \sup_{\delta_i\in[a,b]}\left\|y_1^{i,\delta_i,{\mathbf{v}^i}[v_0]}-x^{i,\delta_i,v_0,v_1}_1\right\|_{\mathcal{S}^2(0,r)}^2.
\end{align*}
Similar to \eqref{eq:w2:con2}, using \eqref{lma:convex:wa:3} in Lemma \ref{lma:convex:wa}, we obtain
\begin{align*}
    &\mE\bigg[W^2_2\bigg(\text{\footnotesize $\frac 1N {\sum_{k\neq i} \delta_{x_1^{k, \Delta_k,v_0}(r)}}+\frac1N\delta_{x_1^{i, \Delta_i,v_0,v_1}(r)},\frac {1}{N-1} {\sum_{k\neq i} \delta_{x_1^{k, \Delta_k,v_0}(r)}}$ } \bigg)\bigg]\\ 
    =\ & \mE\bigg[W^2_2\bigg(\text{\footnotesize $\frac {N-1}N \frac{1}{N-1} {\sum_{k\neq i} \delta_{x_1^{k, \Delta_k,v_0}(r)}}+ \frac{1}{N} \delta_{x_1^{i, \Delta_i,v_0,v_1}(r)},\frac {1}{N-1} {\sum_{k\neq i} \delta_{x_1^{k, \Delta_k,v_0}(r)}}$ } \bigg)\bigg]\\
    \leq\ &\frac{N-1}{N}\cdot0+ \frac1N\mE\bigg[W^2_2\bigg( \text{\footnotesize $ \delta_{x_1^{i, \Delta_i,v_0,v_1}(r)},\ \frac {1}{N-1} {\sum_{k\neq i} \delta_{x_1^{k, \Delta_k,v_0}(r)}}  $}\bigg)\bigg]\\
    \leq\ & \frac1N\left\|x_1^{i, \Delta_i,v_0,v_1}(r)-x_1^{j, \Delta_j,v_0}(r)\right\|_2^2\leq \frac2N\left(\sup_{\delta\in[a,b]} \left\|x_1^{i, \delta,v_0,v_1}\right\|^2_{\sr^2}+\sup_{\delta\in[a,b]} \left\|x_1^{j, \delta,v_0}\right\|^2_{\sr^2}\right).
\end{align*}
Substituting the last three estimates into \eqref{eq:proof:4.5:1}, we get
\begin{align}
     \left\|y_0^{{\mathbf{v}^i}[v_0]}-x_0^{v_0}\right\|^2_{\mathcal{S}^2(0,t)}\leq  \ &C(T,L)\int_0^t\sup_{\delta_j\in[a,b]}\left\|y_1^{j,\delta_j,{\mathbf{v}^i}[v_0]}-x^{j,\delta_j,v_0}_1\right\|_{\mathcal{S}^2(0,r)}^2\md r\nonumber\\
    &+  \frac1N C(T,L)\int_0^t\sup_{\delta_i\in[a,b]}\left\|y_1^{i,\delta_i,{\mathbf{v}^i}[v_0]}-x^{i,\delta_i,v_0,v_1}_1\right\|_{\mathcal{S}^2(0,r)}^2\md r\nonumber\\
    &+C(T,L)\int_0^t\mE\bigg[W^2_2\bigg(\text{\footnotesize$\frac {1}{N-1} {\sum_{k\neq i} \delta_{x_1^{k, \Delta_k,v_0}(r)}}$},z[v_0](r)\bigg)\bigg]\md r\nonumber\\
    &+\frac1N C(T,L)\left(\sup_{\delta\in[a,b]} \left\|x_1^{i, \delta,v_0,v_1}\right\|^2_{\sr^2}+\sup_{\delta\in[a,b]} \left\|x_1^{j, \delta,v_0}\right\|^2_{\sr^2}\right).\label{eq:proof:4.5:4}
\end{align}
    \item [(2)]We next estimate the Wasserstein metric in \eqref{eq:proof:4.5:2} as follows: 
    \begin{align}
        &\mE\bigg[W^2_2\bigg(\text{\footnotesize$\frac{1}{N-1}{\sum_{k \neq i}\delta_{y_1^{k, \Delta_k,{\mathbf{v}^i}[v_0]}(r)}}$},z[v_0](r)\bigg)\bigg]\nonumber\\
        \leq\ &2\mE\bigg[W^2_2\bigg(\text{\footnotesize$\frac{1}{N-1}{\sum_{k \neq i}\delta_{y_1^{k, \Delta_k,{\mathbf{v}^i}[v_0]}(r)}},\frac{1}{N-1}{\sum_{k \neq i}\delta_{x_1^{k, \Delta_k,v_0}(r)}}$}\bigg)\bigg]+2\mE\bigg[W^2_2\bigg(\text{\footnotesize$\frac{1}{N-1}{\sum_{k \neq i}\delta_{x_1^{k, \Delta_k,v_0}(r)}}$},z[v_0](r)\bigg)\bigg]\nonumber\\
        \leq \ & 2\sup_{\delta_j\in[a,b]}\left\|y_1^{j,\delta_j,{\mathbf{v}^i}[v_0]}-x^{j,\delta_j,v_0}_1\right\|_{\mathcal{S}^2(0,r)}^2+2\mE\bigg[W^2_2\bigg(\text{\footnotesize$\frac{1}{N-1}{\sum_{k \neq i}\delta_{x_1^{k, \Delta_k,v_0}(r)}}$},z[v_0](r)\bigg)\bigg],\label{eq:proof:J:diff}
    \end{align}
    where the second inequality uses \eqref{lma:convex:wa:2} in Lemma \ref{lma:convex:wa}. Substituting the last estimate into \eqref{eq:proof:4.5:2}, we obtain
    \begin{align} \sup_{\delta_i\in[a,b]}\left\|y_1^{i,\delta_i,{\mathbf{v}^i}[v_0]}-x^{i,\delta_i,v_0,v_1}_1\right\|_{\mathcal{S}^2(0,t)}^2  \leq\ & C(T,L)\int_0^t\sup_{\delta_j\in[a,b]}\left\|y_1^{j,\delta_j,{\mathbf{v}^i}[v_0]}-x^{j,\delta_j,v_0}_1\right\|_{\mathcal{S}^2(0,r)}^2\md r  \nonumber\\  
     &+C(T,L)
     \int_0^t\mE\bigg[W^2_2\bigg(\text{\footnotesize$\frac {1}{N-1} {\sum_{k\neq i} \delta_{x_1^{k, \Delta_k,v_0}(r)}}$},z[v_0](r)\bigg)\bigg]\md r\nonumber\\
    & + C(T,L)\int_0^t \|y_0^{{\mathbf{v}^i}[v_0]}-x_0^{v_0}\|^2_{\mathcal{S}^2(0,r)}\md r.
  \label{eq:proof:4.5:5}
    \end{align}
 \item [(3)] Finally, we estimate the Wasserstein metric in \eqref{eq:proof:4.5:3} as follows:
 \begin{align*}
     &\mE\bigg[W^2_2\bigg(\text{\footnotesize$\frac {1}{N-1} {\sum_{k\neq i} \delta_{x_1^{k, \Delta_k,v_0}(r)}}$},z[v_0](r)\bigg)\bigg]\\
     \leq\ & 2\mE\bigg[W^2_2\bigg(\text{\footnotesize$\frac {1}{N-1} {\sum_{k\neq i} \delta_{x_1^{k, \Delta_k,v_0}(r)}},\frac {1}{N-1} {\sum_{k\neq i,k\neq j} \delta_{x_1^{k, \Delta_k,v_0}(r)}}+\frac {1}{N-1}\delta_{x_1^{i, \Delta_i,v_0,v_1}(r)}$}\bigg)\bigg]\\
     &+2\mE\bigg[W^2_2\bigg(\text{\footnotesize$\frac {1}{N-1} {\sum_{k\neq i,k\neq j} \delta_{x_1^{k, \Delta_k,v_0}(r)}}+\frac {1}{N-1}\delta_{x_1^{i, \Delta_i,v_0,v_1}(r)}$},z[v_0](r)\bigg)\bigg]\\
     \leq \ & 2\frac{N-2}{N-1}\sup_{\delta_j\in[a,b]}\left\|y_1^{j,\delta_j,{\mathbf{v}^i}[v_0]}-x^{j,\delta_j,v_0}_1\right\|_{\mathcal{S}^2(0,r)}^2+\frac{2}{N-1}\sup_{\delta_i\in[a,b]}\left\|y_1^{i,\delta_i,{\mathbf{v}^i}[v_0]}-x^{i,\delta_i,v_0,v_1}_1\right\|_{\mathcal{S}^2(0,r)}^2\\
     &+2\frac{N-2}{N-1}\mE\bigg[W^2_2\bigg(\text{\footnotesize$\frac {1}{N-2} {\sum_{k\neq i,k\neq j} \delta_{x_1^{k, \Delta_k,v_0}(r)}}$},z[v_0](r)\bigg)\bigg]+\frac{2}{N-1}\mE\bigg[\delta_{x_1^{i, \Delta_i,v_0,v_1}(r)},z[v_0](r)\bigg],
 \end{align*}
  where the last inequality uses \eqref{lma:convex:wa:3} in Lemma \ref{lma:convex:wa}. Moreover,
 \begin{align*}
   \mE\bigg[\delta_{x_1^{i, \Delta_i,v_0,v_1}(r)},z[v_0](r)\bigg]\leq\ & 2\left\|x_1^{i, \Delta_i,v_0,v_1}(r)\right\|_2^2 +2\mE\left[\mathcal{M}^2_2(z[v_0](r))\right]\\
   \leq \ &2\left\|x_1^{i, \Delta_i,v_0,v_1}(r)\right\|_2^2 +2\sup_{\delta\in[a,b]}\left\|x_1^{i, \delta, v_0}(r)\right\|^2_2.
 \end{align*}
 Substituting the last two estimates into \eqref{eq:proof:4.5:3}, we have
 \begin{align}    \sup_{\delta_j\in[a,b]}\left\|y_1^{j,\delta_j,{\mathbf{v}^i}[v_0]}-x^{j,\delta_j,v_0}_1\right\|_{\mathcal{S}^2(0,t)}^2  \leq\ &  \frac1N C(T,L)\int_0^t\sup_{\delta_i\in[a,b]}\left\|y_1^{i,\delta_i,{\mathbf{v}^i}[v_0]}-x^{i,\delta_i,v_0,v_1}_1\right\|_{\mathcal{S}^2(0,r)}^2\md r
    \nonumber\\
    &+ \frac1N C(T,L)\left(\sup_{\delta\in[a,b]} \left\|x_1^{i, \delta,v_0,v_1}\right\|^2_{\sr^2}+\sup_{\delta\in[a,b]} \left\|x_1^{j, \delta,v_0}\right\|^2_{\sr^2}\right)
   \nonumber\\
    & + C(T,L)\int_0^t\mE\bigg[W^2_2\bigg(\text{\footnotesize$\frac {1}{N-2} {\sum_{k\neq i,k\neq j} \delta_{x_1^{k, \Delta_k,v_0}(r)}}$},z[v_0](r)\bigg)\bigg]\md r \nonumber\\  
    & +C(T,L)\int_0^t \|y_0^{{\mathbf{v}^i}[v_0]}-x_0^{v_0}\|^2_{\mathcal{S}^2(0,r)}\md r.
    \label{eq:proof:4.5:6}
 \end{align}
\end{enumerate}
Summing up \eqref{eq:proof:4.5:4}, \eqref{eq:proof:4.5:5} and \eqref{eq:proof:4.5:6}, and applying  Gr\"onwall's inequality, we obtain 
\begin{align*}
  &\left\|y_0^{{\mathbf{v}^i}[v_0]}-x_0^{v_0}\right\|^2_{\mathcal{S}^2}+\sup_{\delta_j\in[a,b]}\left\|y_1^{j,\delta_j,{\mathbf{v}^i}[v_0]}-x^{j,\delta_j,v_0}_1\right\|_{\mathcal{S}^2}^2 +\sup_{\delta_i\in[a,b]}\left\|y_1^{i,\delta_i,{\mathbf{v}^i}[v_0]}-x^{i,\delta_i,v_0,v_1}_1\right\|_{\mathcal{S}^2}^2  \\
  \leq \ & C(T,L)
     \int_0^T\mE\bigg[W^2_2\bigg(\text{\footnotesize$\frac {1}{N-1} {\sum_{k\neq i} \delta_{x_1^{k, \Delta_k,v_0}(r)}}$},z[v_0](r)\bigg)\bigg]\md r\ &\\
  &+ C(T,L)\int_0^T\mE\bigg[W^2_2\bigg(\text{\footnotesize$\frac {1}{N-2} {\sum_{k\neq i,k\neq j} \delta_{x_1^{k, \Delta_k,v_0}(r)}}$},z[v_0](r)\bigg)\bigg]\md r\\
   &+\frac1N C(T,L)\left(\sup_{\delta\in[a,b]} \left\|x_1^{i, \delta,v_0,v_1}\right\|^2_{\sr^2}+\sup_{\delta\in[a,b]} \left\|x_1^{j, \delta,v_0}\right\|^2_{\sr^2}\right)\\
   \leq\ & C(n_1,q,L,T,\xi_0,\xi_1)\bigg(1+l_{u_1[v_0]}+\|v_0\|_{\mathcal{L}^q}^2+\sup_{\delta\in[a,b]}\|u_1^{i,\delta}[v_0]\|_{\mathcal{L}^q}^2\bigg)(f(N-1))^{\frac{2q-4}{3q-4}}\\
   &+\frac1N C(L,T)\sup_{\delta\in[a,b]}\left\|v_1^{i,\delta}\right\|_{\mathcal{L}^2}.
\end{align*}
\end{proof}

\subsection{A class of sub-cases with standard convergence rate}\label{sec:example}

By now, we have established the convergence rates of state process of $N$-player system towards the mean-field system, which depends on the dimension $n_1$ in view of the usage of Lemma~\ref{lemma:wasserstien}. In this section, we also provide a class of sub-cases with a faster and standard convergence rate, i.e., an $\mathcal{O}\left(\frac{1}{\sqrt{N}}\right)$-convergence of the norm $\left\|y_0^{\mathbf{u}[v_0]}-x_0^{v_0}\right\|_{\mathcal{S}^2}$ and $ \left\|y_1^{i,\delta,\mathbf{u}[v_0]}-x_1^{i,\delta,v_0}\right\|_{\mathcal{S}^2}$, which is independent of the dimension $n_1$. This kind of sub-cases require the drift coefficients and the diffusion coefficients to be linear in the distribution argument, see Assumption~\ref{assumption:exp} below. Such convergence rate is also obtained in \cite{Nourian-Caines} for MFGs with a major and $N$ minor players with the coefficients of the form \eqref{intr_1} in contrast against the dominating player as considered in our article. 

\begin{assumption}\label{assumption:exp}
    The drift coefficients $g_0$ and $g_1$ and the diffusion coefficients $\sigma_0$ and $\sigma_1$ are of the following forms: there exist maps 
    \begin{align*}
        &\overline{g}^0_0 (y):\RR^{n_1}\to \RR^{n_0},&&\overline{g}^1_0:\RR^{n_0}\times\RR^{p_0}\to\RR^{n_0}, \\
        &\overline{\sigma}^0_0 (y):\RR^{n_1}\to \RR^{n_0\times d_0},&&\overline{\sigma}^1_0:\RR^{n_0}\times\RR^{p_0}\to\RR^{n_0\times d_0},\\
        &\overline{g}^0_1 (y):\RR^{n_1}\to \RR^{n_1}, &&\overline{g}^1_1: \RR^{n_1}\times\RR^{p_1}\times\RR^{n_0}\to\RR^{n_1},\\
        &\overline{\sigma}^0_1 (y):\RR^{n_1}\to \RR^{n_1\times d_1}, &&\overline{\sigma}^1_1: \RR^{n_1}\times\RR^{p_1}\times\RR^{n_0}\to\RR^{n_1\times d_1}
    \end{align*}
    such that
    \begin{align*}
        g_0(x_0,z,v_0):=\ & \int_{\RR^{n_1}} \overline{g}^0_0 (y)z(\md y) + \overline{g}^1_0 (x_0,v_0),\\
        \sigma_0(x_0,z,v_0):=\ & \int_{\RR^{n_1}} \overline{\sigma}^0_0 (y)z(\md y) + \overline{\sigma}^1_0 (x_0,v_0),\\
        g_1(x,z,v,y_0):=\ & \int_{\RR^{n_1}} \overline{g}^0_1 (y)z(\md y) + \overline{g}^1_1 (x,v,x_0),\\
        \sigma_1(x,z,v,y_0):=\ & \int_{\RR^{n_1}} \overline{\sigma}^0_1 (y)z(\md y) + \overline{\sigma}^1_1 (x,v,x_0),
    \end{align*}
    and the functions $\overline{g}^0_0$, $\overline{g}^1_0$, $\overline{g}^0_1$, $\overline{g}^1_1$,  $\overline{\sigma}^0_0$, $\overline{\sigma}^1_0$, $\overline{\sigma}^0_1$ and $\overline{\sigma}^1_1$ are $L$-Lipschitz continuous in all their arguments.     
\end{assumption}

The functions $\overline{g}^0_0$, $\overline{g}^0_1$,   $\overline{\sigma}^0_0$ and $\overline{\sigma}^0_1$ in Assumption~\ref{assumption:exp} can be seen as  the generating kernels of the coefficients $g_0$, $g_1$, $\sigma_0$ and $\sigma_1$, respectively. Under the particular setting with Assumption~\ref{assumption:exp}, the SDEs for $y_0^{\mathbf{u}[v_0]}(\cdot)$ and $y_1^{i,\delta_i,\mathbf{u}[v_0]}(\cdot)$ also write
{\small
\begin{align}
        &\left\{
            \begin{aligned}
                \md y_0^{\mathbf{u}[v_0]}(t) & =\bigg[\frac{1}{N}\sum_{j=1}^N \overline{g}^0_0 \left(y_1^{j,\Delta_j,\mathbf{u}[v_0]}(t)\right)+\overline{g}^1_0\left(y_0^{\mathbf{u}[v_0]}(t), v_0(t)\right) \bigg] \md t\\
                &\,\, +\bigg[\frac{1}{N}\sum_{j=1}^N \overline{\sigma}_0 \left(y_1^{j,\Delta_j,\mathbf{u}[v_0]}(t)\right)+\overline{\sigma}^1_0\left(y_0^{\mathbf{u}[v_0]}(t), v_0(t)\right) \bigg] \md W_0(t), \quad t\in(0,T],\\
		        y_0^{\mathbf{u}[v_0]}(t) & =\xi_0(t), \quad t \in[-b, 0];\label{eq:y_0:4.3}
            \end{aligned}        
        \right.\\
        &\left\{
            \begin{aligned}
                \md y_1^{i,\delta_i,\mathbf{u}[v_0]}(t) & =\bigg[\frac{1}{N-1}\sum_{j\neq i} \overline{g}^0_1 \left(y_1^{j,\Delta_j,\mathbf{u}[v_0]}(t)\right)+\overline{g}^1_1\bigg(y_1^{i,\delta_i,\mathbf{u}[v_0]}(t), u_1^{i, \delta_i}[v_0](t), y_0^{\mathbf{u}[v_0]}\left(t-\delta_i\right)\bigg)\bigg] \md t\\
                &\,\, +\bigg[\frac{1}{N-1}\sum_{j\neq i} \overline{\sigma}^0_1 \left(y_1^{j,\Delta_j,\mathbf{u}[v_0]}(t)\right)+\overline{\sigma}^1_1\bigg(y_1^{i,\delta_i,\mathbf{u}[v_0]}(t), u_1^{i, \delta_i}[v_0](t), y_0^{\mathbf{u}[v_0]}\left(t-\delta_i\right)\bigg)\bigg] \md W_1^i(t), \quad t\in(0,T],\\
		y_1^{i,\delta_i,\mathbf{u}[v_0]}(0) & =\xi_1^i.
            \end{aligned}
        \right.\label{eq:y_1:4.3}
\end{align}
}

We have the following estimate on the norms of $y_0^{\mathbf{u}[v_0]}(t)- x_0^{v_0}(t)$ and $y_1^{i,\delta_i,\mathbf{u}[v_0]}- x_1^{i,\delta_i,v_0}(t)$ under the particular setting with Assumption~\ref{assumption:exp}, whose proof is given in Appendix~\ref{pf:prop:example}.

\begin{proposition}\label{prop:example}
    Under the particular setting with Assumption~\ref{assumption:exp}, we have 
    \begin{align}\label{exp:thm_0}
   & \left\|y_0^{\mathbf{u}[v_0]}-x_0^{v_0}\right\|^2_{\mathcal{S}^2} + \sup_{\delta\in[a,b]} \left\|y_1^{i,\delta,\mathbf{u}[v_0]}-x_1^{i,\delta,v_0}\right\|^2_{\mathcal{S}^2}\nonumber\\
    \le \ &C(n_1,L,T,\xi_0,\xi_1)\bigg(1+\|v_0\|_{\mathcal{L}^2}^2+\sup_{\delta\in[a,b]}\left\|u_1^{i,\delta}[v_0]\right\|_{\mathcal{L}^2}^2\bigg)\frac{1}{N}.
    \end{align}
\end{proposition}

\section{Approximate Stackelberg Equilibrium}\label{sec:main}
The aim of this section is to employ the results obtained in the previous section to demonstrate that a solution of the limiting Stackelberg game yields an approximate solution tothe $N$-Player Stackelberg game.  Section \ref{sec:appro:1} establishes that, for a fixed $v_0$, the optimal strategy in the limiting game induces an $\epsilon$ Nash equilibrium  for the $N$-player game, see Definition \ref{def:equili}. Section \ref{sec:appro:2} further shows  that $u_0$ serves as an approximate Stackelberg Equilibrium.

\subsection{Approximate Nash equilibrium with a fixed $v_0$}\label{sec:appro:1}
In this section, we assume that Player $0$ takes a fixed control $v_0(\cdot) \in \mathcal{L}_{\mathcal{F}^0}^q\left([0, T] ; \mathbb{R}^{p_0}\right)$. We establish a central result showing that, when the optimal control derived from the limiting game is applied uniformly to all players in the original $N$-player game, the resulting strategy constitutes an approximate Nash equilibrium. 

We begin by introducing the following hypothesis (which will be proven in a separate paper \cite{?}) concerning the $\mathcal{L}^2$-norm of the optimal control $\mathbf{u}[v_0]$:
\begin{hypothesis}\label{hyp:kappa}
There exists a constant $\widehat{C}>0$, such that for any given $v_0(\cdot) \in \mathcal{L}_{\mathcal{F}^0}^q\left([0, T] ; \mathbb{R}^{p_0}\right)$,
\begin{align*}
    \sup_{\delta\in[a,b]}\left\|u_1^{i,\delta}[v_0]\right\|_{\mathcal{L}^2}^2\leq \widehat{C} \left(1 + \| v_0\|_{\lr^2}^2\right),\quad 1\le i\le N.
\end{align*}
\end{hypothesis}

Following \cite{Carmona-Zhu}, we define the sets of admissible controls  as follows.

\begin{definition}[$(\kappa,v_0)$-admissible control set for followers]\label{def:U_i}
    For a given $v_0(\cdot) \in \mathcal{L}_{\mathcal{F}^0}^q\left([0, T] ; \mathbb{R}^{p_0}\right)$ and a constant $\kappa>\widehat{C}$, we define the admissible control set $\mathcal{U}_i(\kappa,v_0)$ for Player $i$ as 
    \begin{align*}
        \mathcal{U}_i(\kappa,v_0):= \bigg\{v^{i,\delta_i}_1\in \mathcal{L}_{\mathcal{G}^{i,\delta_i}}^2\left([0, T] ; \mathbb{R}^{p_1}\right):\sup_{\delta\in[a,b]}\left\|v_1^{i,\delta}\right\|_{\mathcal{L}^2}^2\leq \kappa \left(1 + \| v_0\|_{\lr^2}^2\right) \bigg\}.
    \end{align*}
\end{definition}
We begin with giving an upper bound for $\left|\mathcal{J}^{i, \delta_i, N}({\mathbf{v}^i}[v_0])- \mathcal{J}^{i, \delta_i}(v_1^{i, \delta_i}; v_0,z[v_0])\right|$.

\begin{proposition} \label{prop:J:diff}
Under Assumptions \ref{assumption:initial}-\ref{assump:lip} and Hypothesis \ref{assumption:integral}, 
  suppose that  the $j$-th player $(j\neq i)$ adopts the optimal controls $u^{j,\delta_j}_1[v_0]$ while the $i$-th player adopts an arbitrary control $v^{i,\delta_i}_1\in \mathcal{L}_{\mathcal{G}^{i,\delta_i}}^2\left([0, T] ; \mathbb{R}^{p_1}\right)$. Then we have the following estimate:
    \begin{align*}
    &\left|\mathcal{J}^{i, \delta_i, N}({\mathbf{v}^i}[v_0])- \mathcal{J}^{i, \delta_i}(v_1^{i, \delta_i}; v_0,z[v_0])\right|\\
    \leq\ &  C(n_1,q,L,T,\xi_0,\xi_1)\bigg(1+l_{u_1[v_0]}+\|v_0\|_{\mathcal{L}^q}^2+\sup_{\delta\in[a,b]}\left\|u_1^{i,\delta}[v_0]\right\|_{\mathcal{L}^q}^2+\sup_{\delta\in[a,b]}\left\|v_1^{i,\delta}\right\|_{\mathcal{L}^2}^2\bigg)(f(N-1))^{\frac{q-2}{3q-4}}.
\end{align*}
\end{proposition}
\begin{proof}
By Assumption \ref{assump:lip}, we have
\begin{align*}
    &\mE\bigg[ \bigg|f_1\bigg(y_1^{i, \delta_i,\mathbf{v}^i[v_0]}(t), \text{\footnotesize$\frac{1}{N-1}\sum_{j \neq i} \delta_{y_1^{j, \Delta_j,\mathbf{v}^i[v_0]}(t)}$}, v_1^{i, \delta_i}(t), y_0^{\mathbf{v}^i[v_0]}\left(t-\delta_i\right)\bigg)\\
    &\qquad-f_1\left(x_1^{i, \delta_i,v_0,v_1}(t), z[v_0](t), v_1^{i, \delta_i}(t), x_0^{v_0}\left(t-\delta_i\right)\right)\bigg|\bigg]
    \leq L\|(A)\|_2\cdot\|(B)\|_2,
\end{align*}
where
\begin{align*}
     \|(A)\|_2&:=\bigg\|1+\Big|y_1^{i, \delta_i,\mathbf{v}^i[v_0]}(t)\Big|+\Big|x_1^{i, \delta_i,v_0,v_1}(t)\Big|+\mathcal{M}_2\bigg(\text{\footnotesize$\frac{1}{N-1}\sum_{j \neq i} \delta_{y_1^{j, \Delta_j,\mathbf{v}^i[v_0]}(t)}$}\bigg)\\
    &\qquad+\mathcal{M}_2\left(z[v_0](t)\right)+\Big|y_0^{\mathbf{v}^i[v_0]}\left(t-\delta_i\right)\Big|+\Big|x_0^{v_0}\left(t-\delta_i\right)\Big|\bigg\|_2;\\
     \|(B)\|_2&:=\bigg\|\Big|y_1^{i, \delta_i,\mathbf{v}^i[v_0]}(t)-x_1^{i, \delta_i,v_0,v_1}(t)\Big|+W_2\bigg(\text{\footnotesize$\frac{1}{N-1}\sum_{j \neq i} \delta_{y_1^{j, \Delta_j,\mathbf{v}^i[v_0]}(t)}$},z[v_0](t)\bigg)\\
    &\qquad+\Big|y_0^{\mathbf{v}^i[v_0]}\left(t-\delta_i\right)-x_0^{v_0}\left(t-\delta_i\right)\Big|\bigg\|_2.
\end{align*}
Note that
\begin{align*}
    \|(A)\|_2\ & \leq 1+\Big\|y_1^{i, \delta_i,\mathbf{v}^i[v_0]}(t)\Big\|_2+\Big\|x_1^{i, \delta_i,v_0,v_1}(t)\Big\|_2+\bigg\|\mathcal{M}_2\bigg(\text{\footnotesize$\frac{1}{N-1}\sum_{j \neq i} \delta_{y_1^{j, \Delta_j,\mathbf{v}^i[v_0]}(t)}$}\bigg)\bigg\|_2\\
     &\qquad+\Big\|\mathcal{M}_2\left(z[v_0](t)\right)\Big\|_2+\Big\|y_0^{\mathbf{v}^i[v_0]}\left(t-\delta_i\right)\Big\|_2+\Big\|x_0^{v_0}\left(t-\delta_i\right)\Big\|_2\\
     \ &\leq 1+\Big\|y_1^{i, \delta_i,\mathbf{v}^i[v_0]}\Big\|_{\mathcal{S}^2}+\Big\|x_1^{i, \delta_i,v_0,v_1}\Big\|_{\mathcal{S}^2}+\sup_{\delta\in[a,b]}  \Big\|y_1^{j, \delta, {\mathbf{v}^i}[v_0]}\Big\|_{\mathcal{S}^2}\\
     &\qquad+\sup_{\delta\in[a,b]}\Big\|x_1^{i, \delta, v_0}\Big\|_{\mathcal{S}^2}+\Big\|y_0^{\mathbf{v}^i[v_0]}\Big\|_{\mathcal{S}^2(-b,T)}+\Big\|x_0^{v_0}\Big\|_{\mathcal{S}^2(-b,T)}\\
     \ &\leq C(L,T,\xi_0,\xi_1)\bigg(1+\|v_0\|_{\mathcal{L}^2}+\sup_{\delta\in[a,b]}\left\|u_1^{i,\delta}[v_0]\right\|_{\mathcal{L}^2}+\sup_{\delta\in[a,b]}\left\|v_1^{i,\delta}\right\|_{\mathcal{L}^2}\bigg),
\end{align*}
where the second inequality uses \eqref{eq:z:estimate} and \eqref{eq:y/N:estimate}, and the last inequality uses Lemma \ref{prop:moment}-(i). Moreover,
\begin{align*}
    \|(B)\|_2\ &\leq\Big\|y_1^{i, \delta_i,\mathbf{v}^i[v_0]}(t)-x_1^{i, \delta_i,v_0,v_1}(t)\Big\|_2+\bigg\|W_2\bigg(\text{\footnotesize$\frac{1}{N-1}\sum_{j \neq i} \delta_{y_1^{j, \Delta_j,\mathbf{v}^i[v_0]}(t)}$},z[v_0](t)\bigg)\bigg\|_2\\
    &\qquad+\Big\|y_0^{\mathbf{v}^i[v_0]}\left(t-\delta_i\right)-x_0^{v_0}\left(t-\delta_i\right)\Big\|_2\\
    \ &\leq  \sup_{\delta_i\in[a,b]}\left\|y_1^{i,\delta_i,{\mathbf{v}^i}[v_0]}-x^{i,\delta_i,v_0,v_1}_1\right\|_{\mathcal{S}^2}+\sup_{\delta_j\in[a,b]}\left\|y_1^{j,\delta_j,{\mathbf{v}^i}[v_0]}-x^{j,\delta_j,v_0}_1\right\|_{\mathcal{S}^2}\\
    &\qquad+\bigg\|W_2\bigg(\text{\footnotesize$\frac{1}{N-1}{\sum_{k \neq i}\delta_{x_1^{k, \Delta_k,v_0}(t)}}$},z[v_0](t)\bigg)\bigg\|_2+\left\|y_0^{{\mathbf{v}^i}[v_0]}-x_0^{v_0}\right\|^2_{\mathcal{S}^2}\\
    \ &\leq C(n_1,q,L,T,\xi_0,\xi_1)\bigg(1+\sqrt{l_{u_1[v_0]}}+\|v_0\|_{\mathcal{L}^q}+\sup_{\delta\in[a,b]}\|u_1^{i,\delta}[v_0]\|_{\mathcal{L}^q}\bigg)(f(N-1))^{\frac{q-2}{3q-4}}\\
   &\qquad+\frac{1}{\sqrt{N}} C(L,T)\sup_{\delta\in[a,b]}\left\|v_1^{i,\delta}\right\|_{\mathcal{L}^2}\\
  \ & \leq C(n_1,q,L,T,\xi_0,\xi_1)\\
  &\qquad\cdot\bigg(1+\sqrt{l_{u_1[v_0]}}+\|v_0\|_{\mathcal{L}^q}+\sup_{\delta\in[a,b]}\left\|u_1^{i,\delta}[v_0]\right\|_{\mathcal{L}^q}+\sup_{\delta\in[a,b]}\left\|v_1^{i,\delta}\right\|_{\mathcal{L}^2}\bigg)(f(N-1))^{\frac{q-2}{3q-4}}.
\end{align*}
where the second inequality is similar to \eqref{eq:proof:J:diff}, and the third inequality is due to Lemma \ref{prop:general} and  Proposition \ref{prop:estimate:non-optimal}. Thus, we get
\begin{align*}
    \|(A)\|_2\cdot\|(B)\|_2\leq& C(n_1,q,L,T,\xi_0,\xi_1)\bigg(1+\|v_0\|_{\mathcal{L}^2}+\sup_{\delta\in[a,b]}\left\|u_1^{i,\delta}[v_0]\right\|_{\mathcal{L}^2}+\sup_{\delta\in[a,b]}\left\|v_1^{i,\delta}\right\|_{\mathcal{L}^2}\bigg)\\
    &\cdot\bigg(1+\sqrt{l_{u_1[v_0]}}+\|v_0\|_{\mathcal{L}^q}+\sup_{\delta\in[a,b]}\left\|u_1^{i,\delta}[v_0]\right\|_{\mathcal{L}^q}+\sup_{\delta\in[a,b]}\left\|v_1^{i,\delta}\right\|_{\mathcal{L}^2}\bigg)(f(N-1))^{\frac{q-2}{3q-4}}\\
     \leq \ & C(n_1,q,L,T,\xi_0,\xi_1)
    \\&\cdot\bigg(1+l_{u_1[v_0]}+\|v_0\|_{\mathcal{L}^q}^2+\sup_{\delta\in[a,b]}\left\|u_1^{i,\delta}[v_0]\right\|_{\mathcal{L}^q}^2+\sup_{\delta\in[a,b]}\left\|v_1^{i,\delta}\right\|_{\mathcal{L}^2}^2\bigg)(f(N-1))^{\frac{q-2}{3q-4}}.
\end{align*}
The terminal cost can be treated in a similar manner, yielding the same upper bound, and thus the claim follows.
\end{proof}

\begin{remark}
    In the proof of Proposition \ref{prop:J:diff}, the main inequalities we used are Cauchy-Schwarz inequality and the triangle inequality.   Moreover, the convergence rate appears only in $\|(B)\|$, including two term of $\mathcal{O}\left(f(N)^{\frac{q-2}{3q-4}}\right)$ and $\mathcal{O}\left(\frac{1}{\sqrt{N}}\right)$ (the former is slower that the latter). Therefore, the convergence rate is unlikely improven.
\end{remark}

Next, we aim to show that $u_1^{i,\delta}[v_0]\in \mathcal{U}_i(\kappa,v_0), i=1,2,\cdots, N$ is an $\epsilon=\epsilon(N)$-Nash equilibrium with $\epsilon(N)\to 0$ as $N\to \infty$.
\begin{theorem}\label{thm:epsilon:nash}
    $u_1^{i,\delta}[v_0]\in \mathcal{U}_i(\kappa,v_0), i=1,2,\cdots, N$ is a $\epsilon=\epsilon(N)$-Nash equilibrium with
    \begin{align*}
        \epsilon(N)=C(n_1,q,L,T,\xi_0,\xi_1)\bigg(1+\kappa+l_{u_1[v_0]}+\kappa \|v_0\|_{\mathcal{L}^q}^2\bigg)(f(N-1))^{\frac{q-2}{3q-4}}.
    \end{align*}
\end{theorem}
\begin{proof}
    For any $i=1,2,\cdots, N$ and $v^{i,\delta_i}_1\in \mathcal{U}_i(\kappa,v_0)$, from Proposition \ref{prop:J:diff}, we know that 
    \begin{align}
          &\mathcal{J}^{i, \delta_i, N}(\mathbf{u}[v_0]) \notag \\
          \leq\ &  \mathcal{J}^{i, \delta_i}(u_1^{i,\delta_i}; v_0,z[v_0]) \notag \\
          &
   +  C(n_1,q,L,T,\xi_0,\xi_1)\bigg(1+l_{u_1[v_0]}+\|v_0\|_{\mathcal{L}^q}^2+\sup_{\delta\in[a,b]}\left\|u_1^{i,\delta}[v_0]\right\|_{\mathcal{L}^q}^2\bigg)(f(N-1))^{\frac{q-2}{3q-4}} \notag \\
   \leq \ & \mathcal{J}^{i, \delta_i}(v_1^{i,\delta_i}; v_0,z[v_0])  \notag \\
          &
   +  C(n_1,q,L,T,\xi_0,\xi_1)\bigg(1+l_{u_1[v_0]}+\|v_0\|_{\mathcal{L}^q}^2+\sup_{\delta\in[a,b]}\left\|u_1^{i,\delta}[v_0]\right\|_{\mathcal{L}^q}^2\bigg)(f(N-1))^{\frac{q-2}{3q-4}} \notag \\
   \leq \ &\mathcal{J}^{i, \delta_i, N}({\mathbf{v}^i}[v_0]) \notag \\
     &+  C(n_1,q,L,T,\xi_0,\xi_1)\bigg(1+l_{u_1[v_0]}+\|v_0\|_{\mathcal{L}^q}^2+\sup_{\delta\in[a,b]}\left\|u_1^{i,\delta}[v_0]\right\|_{\mathcal{L}^q}^2+\sup_{\delta\in[a,b]}\left\|v_1^{i,\delta}\right\|_{\mathcal{L}^2}^2\bigg)(f(N-1))^{\frac{q-2}{3q-4}} \notag 
     \\
     \leq \ &\mathcal{J}^{i, \delta_i, N}({\mathbf{v}^i}[v_0])+  C(n_1,q,L,T,\xi_0,\xi_1)\bigg(1+\kappa+l_{u_1[v_0]}+\kappa \|v_0\|_{\mathcal{L}^q}^2\bigg)(f(N-1))^{\frac{q-2}{3q-4}}. \label{add-7}
    \end{align}
    where the second inequality uses the optimality of $u_1^{i,\delta_i}[v_0]$, and the last inequality uses the definition of $\mathcal{U}_i(\kappa,v_0)$. Therefore, the claim follows.
\end{proof}

Furthermore, when $\sigma_0$ is independent of $v_0$, a slightly modified analysis of Lemma \ref{prop:general}, Proposition \ref {prop:estimate:non-optimal}, Proposition \ref{prop:J:diff} and Theorem \ref{thm:epsilon:nash}, based on Remark \ref{remark:sigma_0}, yields an improven rate of convergence:
       \begin{align*}      \epsilon(N)=C(n_1,q,L,T,\xi_0,\xi_1)\bigg(1+\kappa+l_{u_1[v_0]}+\kappa\|v_0\|_{\mathcal{L}^q}^2 \bigg)(f(N-1))^{\frac13}.
    \end{align*}
Similarly, by Remark \ref{remark:delta}, suppose that $\Delta \in \{a_k : k=0,1,\ldots,n\}$ with $a=a_0 < a_1 < \cdots < a_n = b$, $\mP(\Delta=a_k)=p_k$, and $\sum_{k=0}^n p_k = 1$. Then the convergence rate can be further improven:
       \begin{align*}      \epsilon(N)=C(n_1,q,L,T,\xi_0,\xi_1)\bigg(1+\kappa+l_{u_1[v_0]}+\kappa\|v_0\|_{\mathcal{L}^q}^2 \bigg)(f(N-1))^{\frac12}\bigg(\sum_{k=0}^{n}\sqrt{p_k}\bigg)^{\frac12}.
\end{align*} 
Finally, for the particular setting with Assumption~\ref{assumption:exp} in Subsection~\ref{sec:example}, we have the  standard  $\mathcal{O}\left(\frac{1}{\sqrt{N}}\right)$ convergence rate as:
\begin{align*}
  \epsilon(N)=  C(n_1,L,T,\xi_0,\xi_1)\bigg(1+\kappa+\kappa\|v_0\|_{\mathcal{L}^2}^2\bigg)\frac{1}{\sqrt{N}}.
\end{align*}
In the next final subsection, a similar discussion applies to the approximate Stackelberg Nash equilibrium, we omit the details.

\subsection{Approximate Stackelberg Nash equilibrium}
\label{sec:appro:2}

We start by giving the definition of admissible control set (similar to Definition \ref{def:U_i} for the followers) for the dominating player as follows:
\begin{definition}[$\gamma$-admissible control set for leader]\label{def:U_0}
For any $\gamma\geq\|u_0\|_{\mathcal{L}^q}^2$, we define the $\gamma$-admissible set $\mathcal{U}_0(\gamma)$ for  Player $0$ by:
    \begin{align*}
        \mathcal{U}_0(\gamma):= \{v_0\in \mathcal{L}_{\mathcal{F}^0}^q\left([0, T] ; \mathbb{R}^{p_0}\right):\|v_0\|_{\mathcal{L}^q}^2\leq \gamma\}.
    \end{align*}
\end{definition}
\noindent
The goal of this final subsection is to establish that the optimal admissible control pair, 
\begin{align}\label{set:kappa}
    \left\{u_0\in \mathcal{U}_0(\gamma),\ u_1^{i,\delta_i}[u_0]\in  \mathcal{U}_0(u_0,\kappa),\ 1\le j\le N \right\},
\end{align}
constitutes an $(\epsilon_1,\epsilon_2)$-Stackelberg Nash equilibrium for Problem~\ref{eq:aim}. To ensure the existence of a $(\epsilon_1,\epsilon_2)$-Stackelberg Nash equilibrium, we impose the following technical condition on the followers’ strategies associated with each admissible leader’s control (which will be proven in a separate paper \cite{?}).
\begin{hypothesis}\label{assump:u_0}
    For any feasible $v_0\in \mathcal{L}_{\mathcal{F}^0}^q\left([0, T] ; \mathbb{R}^{p_0}\right)$, it holds that
        \begin{align*}
           l_{u_1[v_0]}+\sup_{\delta\in[a,b]}\left\|u_1^{i,\delta}[v_0]\right\|_{\mathcal{L}^q}^2 \leq \widehat{C} \left(1+\|v_0\|_{\mathcal{L}^q}^2 \right),
        \end{align*}
        where $\widehat{C}$ is an universal constant only depends on $n_1,q,L,T,\xi_0,\xi_1$.
\end{hypothesis}

Under this hypothesis, we can now establish the main result of this subsection, which quantifies the approximation error of the Stackelberg-Nash equilibrium for Problem~\ref{eq:aim}.

\begin{theorem}
Suppose that Assumption \ref{assump:u_0} holds, then the optimal admissible control pair \eqref{set:kappa} with $\kappa:=\widehat{C}$ constitutes an $(\epsilon_1,\epsilon_2)$-Stackelberg Nash equilibrium with
  \begin{align}\label{result}
      \epsilon_1(N),\ \epsilon_2(N)\ =C(n_1,q,L,T,\xi_0,\xi_1)\left(1+\gamma\right)(f(N-1))^{\frac{q-2}{3q-4}},
  \end{align}
  and $\epsilon_1$ and $\epsilon_2$ may correspond to different constants $C$ but with the same $\mathcal{O}(f(N-1))^{\frac{q-2}{3q-4}}$ rate. 
\end{theorem}

\begin{proof}
By using Assumption \ref{assump:lip} and Lemma \ref{prop:moment} and Theorem \ref{thm:converge:Ntolimit}, following a similar approach as the proof of Proposition \ref{prop:J:diff}, we can prove that for all $v_0\in \mathcal{L}_{\mathcal{F}^0}^q\left([0, T] ; \mathbb{R}^{p_0}\right)$, 
    \begin{align*}
         &\left| \mathcal{J}^{0, N}(\mathbf{u}[v_0])-\mathcal{J}^{0}\left(v_0,z[v_0]\right)\right|\\\leq \ &C(n_1,q,L,T,\xi_0,\xi_1)\bigg(1+l_{u_1[v_0]}+\|v_0\|_{\mathcal{L}^q}^2+\sup_{\delta\in[a,b]}\left\|u_1^{i,\delta}[v_0]\right\|_{\mathcal{L}^q}^2\bigg)(f(N-1))^{\frac{q-2}{3q-4}}.
    \end{align*}
    Therefore, for any $v_0\in\mathcal{U}_0$, parallel to the approach used in estimating \eqref{add-7} in the proof of Theorem~\ref{thm:epsilon:nash}, we can deduce that
    \begin{align*}
          & \mathcal{J}^{0, N}(\mathbf{u}[u_0])\\
          \leq \ &\mathcal{J}^{0}\left(u_0,z[u_0]\right)+C(n_1,q,L,T,\xi_0,\xi_1)\bigg(1+l_{u_1[u_0]}+\left\|u_0\right\|_{\mathcal{L}^q}^2+\sup_{\delta\in[a,b]}\left\|u_1^{i,\delta}[u_0]\right\|_{\mathcal{L}^q}^2\bigg)(f(N-1))^{\frac{q-2}{3q-4}}\\
          \leq \ & \mathcal{J}^{0}\left(v_0,z[v_0]\right)+C(n_1,q,L,T,\xi_0,\xi_1)\bigg(1+l_{u_1[u_0]}+\left\|u_0\right\|_{\mathcal{L}^q}^2+\sup_{\delta\in[a,b]}\left\|u_1^{i,\delta}[u_0]\right\|_{\mathcal{L}^q}^2\bigg)(f(N-1))^{\frac{q-2}{3q-4}}\\
          \leq \ & \mathcal{J}^{0, N}(\mathbf{u}[v_0])+C(n_1,q,L,T,\xi_0,\xi_1)\bigg(1+l_{u_1[u_0]}+\|u_0\|_{\mathcal{L}^q}^2+\sup_{\delta\in[a,b]}\left\|u_1^{i,\delta}[u_0]\right\|_{\mathcal{L}^q}^2\bigg)(f(N-1))^{\frac{q-2}{3q-4}}\\
          &\qquad\qquad\quad + C(n_1,q,L,T,\xi_0,\xi_1)\bigg(1+l_{u_1[v_0]}+\|v_0\|_{\mathcal{L}^q}^2+\sup_{\delta\in[a,b]}\left\|u_1^{i,\delta}[v_0]\right\|_{\mathcal{L}^q}^2\bigg)(f(N-1))^{\frac{q-2}{3q-4}}\\
          \leq \ & \mathcal{J}^{0, N}(\mathbf{u}[v_0])+C(n_1,q,L,T,\xi_0,\xi_1)\left(1+\|u_0\|_{\mathcal{L}^q}^2+\|v_0\|_{\mathcal{L}^q}^2\right)(f(N-1))^{\frac{q-2}{3q-4}}\\
          \leq \ & \mathcal{J}^{0, N}(\mathbf{u}[v_0])+C(n_1,q,L,T,\xi_0,\xi_1)\left(1+\gamma\right)(f(N-1))^{\frac{q-2}{3q-4}},
    \end{align*}
    from which we obtain $\epsilon_2(N)= C(n_1,q,L,T,\xi_0,\xi_1)\left(1+\gamma\right)(f(N-1))^{\frac{q-2}{3q-4}}$. Then, from Hypothesis \ref{assump:u_0}, by using Theorem~\ref{thm:epsilon:nash} with $\kappa=\widehat{C}$, taking  $\epsilon_1(N)= C(n_1,q,L,T,\xi_0,\xi_1)\left(1+\gamma\right)(f(N-1))^{\frac{q-2}{3q-4}}$ to fulfill \eqref{result}. 
\end{proof}

\section{Conclusion and Future Works}\label{sec:conclution}

In this article, we establish precise convergence rates for a general class of $N$-player Stackelberg games toward their mean field limits. Our framework accommodates time-delayed information, interactions through empirical distributions, and control-dependent diffusion coefficients.
Throughout the paper, we assume that the leader’s Brownian motion does not act as a common noise for the followers, and we focus on the convergence rate by imposing a solvability hypothesis for the limiting mean field Stackelberg game. 

Nevertheless, under the probabilistic approach, incorporating the leader’s Brownian motion as a common noise for the followers does not introduce additional difficulties, since the distributional flow of the followers is already conditional with respect to the filtration associated with the leader.
The solvability hypothesis adopted here will be rigorously justified in a separate work \cite{?}, where we provide a detailed analysis of the well-posedness of the stochastic system consisting of a forward SDE and a backward dynamic equation (see \cite{MR2857245}) arising from the maximum principle for mean field Stackelberg games with time delay. In that study as well, the presence of the leader’s Brownian motion as a common noise for the followers can be handled without increasing the complexity of the problem under the probabilistic approach.
Finally, to apply Lemma \ref{lemma:wasserstien}, we impose the condition $q>4$. Extending our results to the case $q=2$ remains an open problem.

\section*{Acknowledgement}
Alain Bensoussan is supported by the National Science Foundation under grant NSF-DMS-2204795. Ziyu Huang acknowledges the financial supports as a postdoctoral fellow from Department of Statistics and Data Science of The Chinese University of Hong Kong. Sheng Wang acknowledges Professor Ka Chun Cheung and the financial supports as a postdoctoral fellow from Department of Statistics and Actuarial Science, School of Computing and Data Science, The University of
Hong Kong.
Phillip Yam acknowledges the financial supports from HKGRF-14301321 with the project title ``General Theory for Infinite Dimensional Stochastic Control: Mean field and Some Classical Problems'', and HKGRF-14300123 with the project title ``Well-posedness of Some Poisson-driven Mean Field Learning Models and their Applications''. The work described in this article was also supported by a grant from the Germany/Hong Kong Joint Research Scheme sponsored by the Research Grants Council of Hong Kong and the German Academic Exchange Service of Germany (Reference No. G-CUHK411/23). He also thanks The University of Texas at Dallas for the kind invitation to be a Visiting Professor in Naveen Jindal School of Management.

\appendix
\appendixpage	
\section{Proof of Lemmas in Section~\ref{sec:formulation}}\label{app:sec:2}

\subsection{Proof of Lemma~\ref{lma:convex:wa}}\label{pf:lma:convex:wa}

We only prove \eqref{lma:convex:wa:1}, as \eqref{lma:convex:wa:2} and \eqref{lma:convex:wa:3} are special cases of \eqref{lma:convex:wa:1}. Since  the infimum of \eqref{eq:def:w}	is  always attainable (see \cite{AB,Villani} or \cite[Section 5.1]{carmona2018probabilistic} for instance), for any $s\in A$, there is a $\Gamma_s\in \mathcal{P}_2\left(\mR^{n_1}\times \mR^{n_1}\right)$ with marginal distribution  $\mu_s$ and $\nu_s$  such that
$$W^2_2(\mu_s,\nu_s)=\int_{\mR^{n_1}\times\mR^{n_1}}|x-y|^2\md \Gamma_{s}(x,y),\quad s\in A.$$
	Note that the joint distribution $\int_{s\in A}\Gamma_s\md \pi(s)$ has the marginal distribution $\int_A\mu_s\md \pi(s)\in \mathcal{P}_2(\mR^{n_1})$ and $\int_A\nu_s\md \pi(s)\in \mathcal{P}_2(\mR^{n_1})$. 
    By the definition of Wasserstein metric, we have    
\begin{align*}
    W_2^2 \left(\int_A\mu_s\md \pi(s), \int_A\nu_s\md \pi(s) \right)\le\ & \int_{\mR^{n_1}\times\mR^{n_1}} |x-y|^2 \md \left(\int_{s\in A}\Gamma_s\md \pi(s)\right)(x,y)  \\
    =\ & \int_A \int_{\mR^{n_1}\times\mR^{n_1}} |x-y|^2 \md \Gamma_s (x,y) \md \pi(s) \\
    =\ & \int_AW_2^2(\mu_s,\nu_s)\md \pi(s).
\end{align*}

\subsection{Proof of Lemma~\ref{lma:pq:mu}}\label{pf:lma:pq:mu}

Similar to Lemma \ref{lma:convex:wa}, we have
	\begin{align*}
		W^2_2\left(\sum_{k=1}^np_k\mu_k,\sum_{k=1}^nq_k\mu_k\right)\leq \inf_{\pi\in \Pi}\sum_{h=1}^n\sum_{l=1}^n\pi_{hl}W^2_2(\mu_h,\mu_l),
	\end{align*}
	where $\Pi:= \Big\{\pi_{hl}\geq0 \Big| \sum_{l=1}^n \pi_{hl}=p_h,\sum_{h=1}^n \pi_{hl}=q_l\Big\}$. Note that
	\begin{align*}
		\sum_{h=1}^n (p_h-\min\{p_h,q_h\})=\sum_{h\in A^{\complement}}(p_h-q_h), \qquad 	\sum_{h=1}^n (q_h-\min\{p_h,q_h\})=\sum_{h\in A}(q_h-p_h).
	\end{align*}
	Thus, we obtain
	\begin{align*}
		&\sum_{h=1}^n (p_h-\min\{p_h,q_h\})+\sum_{h=1}^n (q_h-\min\{p_h,q_h\})=\sum_{h=1}^n|p_h-q_h|,\\
		&\sum_{h=1}^n(p_h-\min\{p_h,q_h\})-\sum_{h=1}^n(q_h-\min\{p_h,q_h\})=0.
	\end{align*}
	Therefore
	\begin{align}\label{eq:A=B}
		\sum_{h\in A^{\complement}}(p_h-q_h)=\sum_{h\in A}(q_h-p_h).
	\end{align}
Next we verify that $\sum_{l=1}^n\hat{\pi}_{hl}=p_h,\sum_{h=1}^n\hat{\pi}_{hl}=q_l$, from which the conclusion follows.
\begin{enumerate}
	\item 	For $h\in A$,  by definition of $\hat{\pi}$, we know that $\hat{\pi}_{hl}=0$ for $l\neq h$, thus we have $\sum_{l=1}^n\hat{\pi}_{hl}=\hat{\pi}_{hh}=p_h$.
	\item 	For $h\in A^{\complement}$, we can deduce that 
	\begin{align*}
		\sum_{l=1}^n\hat{\pi}_{hl}&=q_h+\sum_{l\neq h}\hat{\pi}_{hl}=q_h+\sum_{l\neq h,l\in A}\hat{\pi}_{hl}+\sum_{l\neq h,l\in A^{\complement}}\hat{\pi}_{hl}\\
		&=q_h+\sum_{l\neq h,l\in A}\hat{\pi}_{hl}=q_h+(p_h-q_h)\frac{\sum_{l\in A}(q_l-p_l)}{\sum_{k\in A^{\complement}}(p_k-q_k)}=p_h.
	\end{align*}
	\item  	For $l\in A^{\complement}$, by definition of $\hat{\pi}$, we know that $\hat{\pi}_{hl}=0$ for $h\neq l$, so we have
	$\sum_{h=1}^n\hat{\pi}_{hl}=\hat{\pi}_{ll}=q_l$.	
    \item For $l\in A$, 
    we know that
	\begin{align*}
		\sum_{h=1}^n\hat{\pi}_{hl}&=p_l+\sum_{h\neq l}\hat{\pi}_{hl}=p_l+\sum_{h\neq l,h\in A}\hat{\pi}_{hl}+\sum_{h\neq l,h\in A^{\complement}}\hat{\pi}_{hl}\\
		&=p_l+\sum_{h\neq l,h\in A^{\complement}}\hat{\pi}_{hl}=p_l+(q_l-p_l)\frac{\sum_{h\in A^{\complement}}(p_h-q_h)}{\sum_{k\in A^{\complement}}(p_k-q_k)}=q_l.
	\end{align*}
\end{enumerate}
Finally, \eqref{eq:p_ii} is a direct consequence of \eqref{eq:A=B}.

{\subsection{Proof of Lemma \ref{lma:property:probability}}\label{pf:lma:property:probability}

The proof of \eqref{eq:property:X} is similar to that in \cite{aksamit2017enlargement,bremaud1978changes,carmona2018probabilistic}, so here we only sketch out the argument. 
First, by independence,  \eqref{eq:property:X} holds for any $X=\ind_{C_1\cap G}$ with $C_1\in\mathscr{C_1}$ and $G\in \mathscr{G}$. Since $\Pi:=\{C_1\cap G:C_1\in\mathscr{C_1}, G\in \mathscr{G}\}$ is a $\pi$-system, the monotone class theorem implies that \eqref{eq:property:X} holds for any integrable $X$. Next, we prove \eqref{eq:property:regular}. For any $t=(t_1,t_2,\cdot,t_{n_1})^\top\in\mR^{n_1}$ and $i=1,2$, define
\begin{align*}
    F_{\mathscr{C}_i}(t)=\mP^{\mathscr{C}_i}_X((-\infty,t_1]\times(-\infty,t_2]\times\cdots\times(-\infty,t_{n_1}])=\mE\left[\large{\text{$\ind$}}_{(-\infty,t_1]\times(-\infty,t_2]\times\cdots\times(-\infty,t_{n_1}]}(X)|\mathscr{C}_i\right].
\end{align*}
By \eqref{eq:property:X}, we know that for any $t$, we have $F_{\mathscr{C}_1}(t)=F_{\mathscr{C}_2}(t),\  \mP-a.s.$. Hence,
\begin{align*}
    \mP\Big( F_{\mathscr{C}_1}(t)=F_{\mathscr{C}_2}(t),\ \forall t\in \mathbb{Q}^{n_1}\Big)=1,
\end{align*}
where $\mathbb{Q}^{n_1}=\{t:t=(t_1,t_2,\dots,t_{n_1})^\top\in\mathbb{Q}^{n_1}\}$ and $\mathbb{Q}$ denotes the rational numbers of $\mR$. Since $ F_{\mathscr{C}_1}$ and $ F_{\mathscr{C}_2}$ are right-continuous, we have
\begin{align*}
    \mP\Big( F_{\mathscr{C}_1}(t)=F_{\mathscr{C}_2}(t),\forall t\in \mathbb{R}^{n_1}\Big)=1,
\end{align*}
which yields \eqref{eq:property:regular}.
}

\section{Proof of lemmas in Section~\ref{sec:assumption&SDE}}\label{app:sec:3}

\subsection{Proof of Lemma \ref{prop:moment} }\label{prop:proof:moment}
We first prove \eqref{C.1.-1'}-\eqref{C.1.hat.x_1}.
    From Assumption~\ref{assump:lip} and standard estimate for SDEs (see \cite{MR3793166,AB_book,MR3674558} for instance) on SDEs of $x_0^{v_0}$ and $x_1^{i,\delta_i,v_0,v_1}$, we know that 
\begin{align}
\left\|x_0^{v_0}\right\|_{\mathcal{S}^2(-b,t)}^2\leq &C(L,T)\left[1+\left\|\xi_0\right\|^2_{\mathcal{S}^2(-b,0)}+\left\|\mathcal{M}_2(z[v_0]\right\|^2_{\mathcal{L}^2(0,t)}+\left\|v_0\right\|^2_{\mathcal{L}^2(0,t)} \right],
    \label{C.1._proof_01}\\  \left\|x_1^{i,\delta_i,v_0,v_1}\right\|_{\mathcal{S}^2(0,t)}^2\leq &C(L,T)\left[1+\left\|\xi_1\right\|_2^2+\left\|\mathcal{M}_2(z[v_0]\right\|^2_{\mathcal{L}^2(0,t)}+\left\|v_1^{i,\delta_i}\right\|^2_{\mathcal{L}^2(0,t)} +\left\|x_0^{v_0}\right\|^2_{\mathcal{L}^2(-b,t)}\right] . 
    \label{C.1._proof_02}
\end{align}
From the definition of $z[v_0]$, we see that 
\begin{align}
    \left\|\mathcal{M}_2(z[v_0](s))\right\|_2^2 = \e\left[ \int_{\RR^{n_1}}|y|^2 z[v_0](s)(\md y)\right] =\ & \e\left[\int_{\RR^{n_1}}|y|^2 \int_{[a,b]} \text{\footnotesize$\mathbb{P}^{\F_{s-\delta}^0 \vee \F_{s}^z}_{x_1^{i,\delta,v_0}(s)}$} \md\pi_{\Delta}(\delta)(\md y)\right]\nonumber \\
    =\ &\e\left[ \int_{[a,b]} \int_{\RR^{n_1}}|y|^2  \text{\footnotesize$\mathbb{P}^{\F_{s-\delta}^0 \vee \F_{s}^z}_{x_1^{i,\delta,v_0}(s)} $}(\md y) \md\pi_{\Delta}(\delta) \right] \nonumber\\
    =\ & \int_{[a,b]} \e\left[\int_{\RR^{n_1}}  |y|^2  \text{\footnotesize$\mathbb{P}^{\F_{s-\delta}^0 \vee \F_{s}^z}_{x_1^{i,\delta,v_0}(s)}$} (\md y) \right]\md\pi_{\Delta}(\delta) \nonumber \\
    \le\ & \sup_{\delta\in[a,b]} \left\|x_1^{i,\delta,v_0}(s)\right\|^2_2\leq  \sup_{\delta\in[a,b]} \left\|x_1^{i,\delta,v_0}\right\|_{\mathcal{S}^2(0,s)}^2.\label{eq:mz}
\end{align}
Substituting the last estimate into \eqref{C.1._proof_01} and \eqref{C.1._proof_02}, we have 
\begin{align}
\left\|x_0^{v_0}\right\|_{\mathcal{S}^2(-b,t)}^2\leq &C(L,T)\left[1+\left\|\xi_0\right\|^2_{\mathcal{S}^2(-b,0)}+\|v_0\|_{\mathcal{L}^2(0,t)}^2+\int_0^t\sup_{\delta\in[a,b]} \left\|x_1^{i,\delta,v_0}\right\|_{\mathcal{S}^2(0,s)}^2\md s \right],
    \label{C.1._proof_03}\\
    \sup_{\delta\in[a,b]}\left\|x_1^{i,\delta,v_0,v_1}\right\|_{\mathcal{S}^2(0,t)}^2\leq &C(L,T)\bigg[1+\left\|\xi_1\right\|_2^2+\sup_{\delta\in[a,b]}\left\|v_1^{i,\delta}\right\|_{\mathcal{L}^2(0,t)}^2+\|x_0^{v_0}\|_{\mathcal{L}^2(-b,t)}^2\nonumber\\
    &\qquad\qquad \qquad\qquad\qquad\qquad\qquad+ \int_0^t\sup_{\delta\in[a,b]} \left\|x_1^{i,\delta,v_0}\right\|_{\mathcal{S}^2(0,s)}^2\md s\bigg] .
    \label{C.1._proof_04}
\end{align}
Substituting \eqref{C.1._proof_03} into \eqref{C.1._proof_04}, we know that
\begin{align}
    \sup_{\delta\in[a,b]}\left\|x_1^{i,\delta,v_0,v_1}\right\|_{\mathcal{S}^2(0,t)}^2\leq &C(L,T)\bigg[1+\left\|\xi_0\right\|^2_{\mathcal{S}^2(-b,0)}+\left\|\xi_1\right\|_2^2+\|v_0\|_{\mathcal{L}^2(0,t)}^2+\sup_{\delta\in[a,b]}\|v_1^{i,\delta}\|_{\mathcal{L}^2(0,t)}^2\notag \\
    &\qquad\qquad \qquad\qquad\qquad\qquad\qquad+ \int_0^t\sup_{\delta\in[a,b]} \left\|x_1^{i,\delta,v_0}\right\|_{\mathcal{S}^2(0,s)}^2\md s \bigg] ,\label{C.1._proof_05}
\end{align}
and particularly, when $v_1^{i, \delta_i}(\cdot)=u_1^{i, \delta_i}[v_0](\cdot)$,
\begin{align*}
\sup_{\delta\in[a,b]}\left\|x_1^{i,\delta,v_0}\right\|_{\mathcal{S}^2(0,t)}^2\leq &C(L,T)\bigg[1+\left\|\xi_0\right\|^2_{\mathcal{S}^2(-b,0)}+\left\|\xi_1\right\|_2^2+\|v_0\|_{\mathcal{L}^2}^2+\sup_{\delta\in[a,b]}\|u_1^{i,\delta}[v_0]\|_{\mathcal{L}^2}^2\notag \\
    &\qquad\qquad \qquad\qquad\qquad\qquad\qquad+ \int_0^t\sup_{\delta\in[a,b]} \left\|x_1^{i,\delta,v_0}\right\|_{\mathcal{S}^2(0,s)}^2\md s \bigg].
\end{align*}
Applying Gr\"onwall's inequality to the last inequality, we obtain \eqref{C.1.hat.x_1}.
Substituting \eqref{C.1.hat.x_1} into \eqref{C.1._proof_03} and \eqref{C.1._proof_05}, we obtain \eqref{C.1.-1'} and \eqref{C.1.-1''}.

We now prove \eqref{C.1.-2'}. Similar as in \eqref{C.1._proof_01} and \eqref{C.1._proof_02}, we have the following estimates on $y_0^{{\mathbf{v}^i}[v_0]}$, $y_1^{i,\delta_i,{\mathbf{v}^i}[v_0]}$ and $y_1^{j,\delta_j,{\mathbf{v}^i}[v_0]}$:
\begin{align}
\left\|y_0^{{\mathbf{v}^i}[v_0]}\right\|_{\mathcal{S}^2(0,t)}^2\leq &C(L,T)\left[1+\left\|\xi_0\right\|^2_{\mathcal{S}^2(-b,0)}+\|v_0\|_{\mathcal{L}^2(0,t)}^2+\bigg\|\mathcal{M}_2\bigg(\text{\footnotesize$\frac 1N {\sum_{k=1}^N \delta_{y_1^{k, \Delta_k,{\mathbf{v}^i}[v_0]}}}$}\bigg)\bigg\|_{\mathcal{L}^2(0,t)}^2\right];\label{C.1._proof_1}\\
 \left\|y_1^{i,\delta_i,{\mathbf{v}^i}[v_0]}\right\|_{\mathcal{S}^2(0,t)}^2\leq &C(L,T)\bigg[1+\left\|\xi_1^i\right\|^2_2
    +\left\|v_1^{i,\delta_i}\right\|_{\mathcal{L}^2(0,t)}+\left\|y_0^{{\mathbf{v}^i}[v_0]}\right\|_{\mathcal{L}^2(-b,t)}\nonumber\\&\quad\qquad\qquad\qquad\qquad\qquad\qquad+\bigg\|\mathcal{M}_2\bigg(\text{\footnotesize$\frac {1}{N-1} {\sum_{k\neq i} \delta_{y_1^{k, \Delta_k,{\mathbf{v}^i}[v_0]}}}$}\bigg)\bigg\|_{\mathcal{L}^2(0,t)}^2\bigg];\label{C.1._proof_3}\\
    \left\|y_1^{j,\delta_j,{\mathbf{v}^i}[v_0]}\right\|_{\mathcal{S}^2(0,t)}^2\leq &C(L,T)\bigg[1+\left\|\xi_1^j\right\|^2_2
    +\left\|u_1^{j,\delta_j}[v_0]\right\|_{\mathcal{L}^2(0,t)}+\left\|y_0^{{\mathbf{v}^i}[v_0]}\right\|_{\mathcal{L}^2(-b,t)}\nonumber\\&\quad\qquad\qquad\qquad\qquad\qquad\qquad+\bigg\|\mathcal{M}_2\bigg(\text{\footnotesize$\frac {1}{N-1} {\sum_{k\neq j} \delta_{y_1^{k, \Delta_k,{\mathbf{v}^i}[v_0]}}}$}\bigg)\bigg\|_{\mathcal{L}^2(0,t)}^2\bigg].\label{C.1._proof_2}
\end{align}
Note the fact that $\left\{y_1^{k, \Delta_k,{\mathbf{v}^i}[v_0]}(s),\ 1\le k\le N,k\neq i \right\}$ are identically distributed, we can deduce that
\begin{align*}
    \e\bigg[\mathcal{M}^2_2\bigg(\text{\footnotesize$\frac{1}{N} {\sum_{k=1}^N \delta_{y_1^{k, \Delta_k,{\mathbf{v}^i}[v_0]}(s)}}$}\bigg) \bigg] =\ & \e\bigg[ \int_{\RR^{n_1}} |y|^2  \text{\footnotesize$\frac{1}{N} {\sum_{k\neq i} \delta_{y_1^{k, \Delta_k,{\mathbf{v}^i}[v_0]}(s)}(\md y) }$}\bigg]+\e\bigg[ \int_{\RR^{n_1}} |y|^2 \text{\footnotesize$ \frac{1}{N}  \delta_{y_1^{i, \Delta_i,{\mathbf{v}^i}[v_0]}(s)}(\md y) $}\bigg]\\\
    =&\frac{N-1}{N} \e\left[ \int_{\RR^{n_1}} |y|^2    \delta_{y_1^{j, \Delta_j,{\mathbf{v}^i}[v_0]}(s)}(\md y)  \right]+\frac{1}{N} \e\left[ \int_{\RR^{n_1}} |y|^2    \delta_{y_1^{i, \Delta_i,{\mathbf{v}^i}[v_0]}(s)}(\md y)  \right]\\
    =&\frac{N-1}{N} \e \left[ \left|y_1^{j, \Delta_j,{\mathbf{v}^i}[v_0]}(s)\right|^2 \right]+\frac{1}{N} \e \left[ \left|y_1^{i, \Delta_i,{\mathbf{v}^i}[v_0]}(s)\right|^2 \right]\\
     \le& \frac{N-1}{N}\sup_{\delta\in[a,b]}  \left\|y_1^{j,\delta,{\mathbf{v}^i}[v_0]}\right\|_{\mathcal{S}^2(0,s)}^2+\frac{1}{N}\sup_{\delta\in[a,b]}  \left\|y_1^{i,\delta,{\mathbf{v}^i}[v_0]}\right\|_{\mathcal{S}^2(0,s)}^2,
\end{align*}
and similarly, 
\begin{align*}
   &  \e\bigg[\mathcal{M}^2_2\bigg(\text{\footnotesize$\frac{1}{N-1} {\sum_{k\neq i} \delta_{y_1^{k, \Delta_k,{\mathbf{v}^i}[v_0]}(s)}}$}\bigg) \bigg]\le \sup_{\delta\in[a,b]}  \left\|y_1^{j,\delta,{\mathbf{v}^i}[v_0]}\right\|_{\mathcal{S}^2(0,s)}^2,\\
    & \e\bigg[\mathcal{M}^2_2\bigg(\text{\footnotesize$\frac{1}{N-1} {\sum_{k\neq j} \delta_{y_1^{k, \Delta_k,{\mathbf{v}^i}[v_0]}(s)}}$}\bigg) \bigg]\le \frac{N-2}{N-1}\sup_{\delta\in[a,b]}  \left\|y_1^{j,\delta,{\mathbf{v}^i}[v_0]}\right\|_{\mathcal{S}^2(0,s)}^2+\frac{1}{N-1} \sup_{\delta\in[a,b]}  \left\|y_1^{i,\delta,{\mathbf{v}^i}[v_0]}\right\|_{\mathcal{S}^2(0,s)}^2.
\end{align*}
Substituting the last three estimates into \eqref{C.1._proof_1}, \eqref{C.1._proof_2} and \eqref{C.1._proof_3}, we have
\begin{align}
\left\|y_0^{{\mathbf{v}^i}[v_0]}\right\|_{\mathcal{S}^2(0,t)}^2\leq &C(L,T)\bigg[1+\left\|\xi_0\right\|^2_{\mathcal{S}^2(-b,0)}+\|v_0\|_{\mathcal{L}^2(0,t)}^2\nonumber\\
&+\frac{N-1}{N}\int_0^t\sup_{\delta\in[a,b]}  \left\|y_1^{j,\delta,{\mathbf{v}^i}[v_0]}\right\|_{\mathcal{S}^2(0,s)}^2\md s+\frac{1}{N}\int_0^t\sup_{\delta\in[a,b]}  \left\|y_1^{i,\delta,{\mathbf{v}^i}[v_0]}\right\|_{\mathcal{S}^2(0,s)}^2\md s
\bigg];\label{eq:proof:y0}\\
 \left\|y_1^{i,\delta_i,{\mathbf{v}^i}[v_0]}\right\|_{\mathcal{S}^2(0,t)}^2\leq &C(L,T)\bigg[1+\left\|\xi_1^i\right\|^2_2
    +\left\|v_1^{i,\delta_i}\right\|_{\mathcal{L}^2(0,t)}+\left\|y_0^{{\mathbf{v}^i}[v_0]}\right\|_{\mathcal{L}^2(-b,t)}+\int_0^t\sup_{\delta\in[a,b]}  \left\|y_1^{j,\delta,{\mathbf{v}^i}[v_0]}\right\|_{\mathcal{S}^2(0,s)}^2\md s\bigg];\label{eq:proof:yi}\\
    \left\|y_1^{j,\delta_j,{\mathbf{v}^i}[v_0]}\right\|_{\mathcal{S}^2(0,t)}^2\leq& C(L,T)\bigg[1+\left\|\xi_1^j\right\|^2_2
    +\left\|u_1^{j,\delta_j}[v_0]\right\|_{\mathcal{L}^2(0,t)}+\left\|y_0^{{\mathbf{v}^i}[v_0]}\right\|_{\mathcal{L}^2(-b,t)}\nonumber\\&+\frac{N-2}{N-1}\int_0^t\sup_{\delta\in[a,b]}  \left\|y_1^{j,\delta,{\mathbf{v}^i}[v_0]}\right\|_{\mathcal{S}^2(0,s)}^2\md s+\frac{1}{N-1}\int_0^t\sup_{\delta\in[a,b]}  \left\|y_1^{i,\delta,{\mathbf{v}^i}[v_0]}\right\|_{\mathcal{S}^2(0,s)}^2\md s\bigg].\label{eq:proof:yj}
\end{align}
 Summing up \eqref{eq:proof:y0} and \eqref{eq:proof:yj}, we get
 \begin{align*}
   \left\|y_0^{{\mathbf{v}^i}[v_0]}\right\|_{\mathcal{S}^2(0,t)}^2&+  \sup_{\delta\in[a,b]} \left\|y_1^{j,\delta,{\mathbf{v}^i}[v_0]}\right\|_{\mathcal{S}^2(0,t)}^2\leq C(L,T)\bigg[1
    +\left\|\xi_0\right\|^2_{\mathcal{S}^2(-b,0)}+\left\|\xi_1^j\right\|^2_2+\|v_0\|_{\mathcal{L}^2}^2+\sup_{\delta\in[a,b]}\left\|u_1^{j,\delta}[v_0]\right\|_{\mathcal{L}^2}\nonumber\\&+\left\|y_0^{{\mathbf{v}^i}[v_0]}\right\|_{\mathcal{L}^2(-b,t)}+\frac{N-1}{N}\int_0^t\sup_{\delta\in[a,b]}  \left\|y_1^{j,\delta,{\mathbf{v}^i}[v_0]}\right\|_{\mathcal{S}^2(0,s)}^2\md s+\frac{1}{N-1}\int_0^t\sup_{\delta\in[a,b]}  \left\|y_1^{i,\delta,{\mathbf{v}^i}[v_0]}\right\|_{\mathcal{S}^2(0,s)}^2\md s\bigg].
 \end{align*}
 Substituting \eqref{eq:proof:yi} into the last estimate and applying Gr\"onwall's inequality, we obtain \eqref{C.1.-2'}. Substituting \eqref{C.1.-2'} into \eqref{eq:proof:yi}, we get \eqref{C.1.-2''} 
 The proof for \eqref{estimate:SDE:Lq-1} is similar as that for \eqref{C.1.-1'} and \eqref{C.1.hat.x_1}, which is omitted here. 

\subsection{Proof of Lemma~\ref{prop:holdercontinuity}}\label{proof:prop:holder}
    From the SDE of the process $x_0^{v_0}(\cdot)$, we know that for $0\le t\le s\le T$,
\begin{align*}
    x_0^{v_0}(s)-x_0^{v_0}(t)= \int _t^s g_0 \left(x_0^{v_0}(r), z[v_0](r), v_0(r)\right) \md t + \int_t^s  \sigma_0 \left(x_0^{v_0}(r), z[v_0](r), v_0(r)\right) \md W_0(r).
\end{align*}
Then from the Cauchy–Schwarz inequality and Assumption~\ref{assump:lip}, we have
\begin{align}
    &\left\| x_0^{v_0}(s)-x_0^{v_0}(t) \right\|_2^2  \notag \\
    \le\ & 2 (s-t)\e\left[\int_t^s \left| g_0 \left(x_0^{v_0}(r), z[v_0](r), v_0(r)\right)\right|^2  \md r\right] + 2 \e\left[\int_t^s \left| \sigma_0 \left(x_0^{v_0}(r), z[v_0](r), v_0(r)\right)\right|^2  \md r\right] \notag \\
    \le\ &  C(L,T)(s-t)\bigg(1+\sup_{r\in[0,T]} \left\|x_0^{v_0}(r)\right\|_2^2 + \sup_{r\in[0,T]} \left\|\mathcal{M}_2(z[v_0](r))\right\|_2^2 \bigg) + C(L,T) \|v_0\|_{\mathcal{L}^2(t,s)}^2. \label{C.2._proof_1}
\end{align}
From \eqref{eq:mz}, we see that 
\begin{align}
    \sup_{r\in[0,T]} \left\|\mathcal{M}_2(z[v_0](r))\right\|_2^2  
    \le \sup_{\delta\in[a,b]} \sup_{r\in[0,T]}  \left\|x_1^{i,\delta,v_0}(r)\right\|_2^2. \notag
\end{align}
Substituting the last estimate into \eqref{C.2._proof_1}, and using \eqref{C.1.-1'} and \eqref{C.1.hat.x_1}, we have
\begin{align}
    &\left\| x_0^{v_0}(s)-x_0^{v_0}(t) \right\|_2^2 \notag 
    \\
    \le\ & C(L,T)\bigg[ 1+ \left\|\xi_0\right\|^2_{\sr^2(-b,0)} + \left\|\xi^i_1\right\|_2^2 +  \|v_0\|_{\lr^2}^2 + \sup_{\delta\in[a,b]} \left\|u_1^{i, \delta}[v_0]\right\|_{\lr^2}^2\bigg](s-t)+C(L,T) \|v_0\|_{\mathcal{L}^2(t,s)}^2 \notag 
    \\
    \le\ &C(L,T,\xi_0,\xi^i_1)\left(1+\|v_0\|_{\lr^2}^2 + \sup_{\delta\in[a,b]} \left\|u_1^{i, \delta}[v_0]\right\|_{\lr^2}^2\bigg]\right)(s-t)+ C(L,T) \|v_0\|_{\lr^q}^2(s-t)^{\frac{q-2}{q}} \notag \\
    \le&C(L,T,\xi_0,\xi^i_1,q)\bigg(1+\|v_0\|_{\lr^q}^2+\sup_{\delta\in[a,b]} \left\|u_1^{i, \delta}[v_0]\right\|_{\lr^2}^2\bigg)(s-t)^{\frac{q-2}{q}},
    \label{C.2._proof_2}
\end{align}
where the second inequality uses the H\"older inequality to  get $\|v_0\|_{\mathcal{L}^2(t,s)}^2\leq \|v_0\|_{\lr^q}^2(s-t)^{\frac{q-2}{q}}$. Combined with \eqref{assum:xi_0_conti}, we know that \eqref{C.2._proof_2} holds for any $s,t\in[-b,T]$, possibly with a larger constant $C(L,T,\xi_0,\xi^i_1,q)$.

Next, we suppose that $a\le \delta\le \gamma\le b$. From the SDE  of $x_1^{i,\delta_i,v_0}(\cdot)$ and Assumption~\ref{assump:lip}, we know that 
\begin{align}
    &\left\|x_1^{i,\delta,v_0}(s)-x_1^{i,\gamma,v_0}(s)\right\|^2_2 \notag \\
    \le\ & C(L,T)
   \int_0^s \left(\left\|x_1^{i,\delta,v_0}(t)-x_1^{i,\gamma,v_0}(t)\right\|^2_2+\left\|u_1^{i,\delta}[v_0](t)-u_1^{i,\gamma}[v_0](t)\right\|^2_2+ \left\|x_0^{v_0}\left(t-\delta\right) - x_0^{v_0}\left(t-\gamma\right)\right\|^2_2 \right) \md t \notag \\
    \le \ & C(L,T) \int_0^t\left\|x_1^{i,\delta,v_0}(t)-x_1^{i,\gamma,v_0}(t)\right\|^2_2\md s+C(L,T)l_{u_1[v_0]}|\delta-\gamma|^{\frac{q-2}{q}}\nonumber\\
    &\qquad\qquad+C(L,T,\xi_0,\xi^i_1,q)\bigg(1+\|v_0\|_{\lr^q}^2+\sup_{\delta\in[a,b]} \left\|u_1^{i, \delta}[v_0]\right\|_{\lr^2}^2\bigg)(\delta-\gamma)^{\frac{q-2}{q}},\nonumber
\end{align}
where the last inequality uses \eqref{C.2._proof_2}  and  Hypothesis \ref{assumption:integral}-(ii). Then, by applying Gr\"onwall's inequality, we obtain 
   \begin{align*}
\left\|x_1^{i,\delta,v_0}(s)-x_1^{i,\gamma,v_0}(s)\right\|^2_2 \leq & C(L,T)l_{u_1[v_0]}(\gamma-\delta)^{\frac{q-2}{q}} \\ &+C(L,T,\xi_0,\xi^i_1,q)\bigg(1+\|v_0\|_{\lr^q}^2+\sup_{\delta\in[a,b]} \left\|u_1^{i, \delta}[v_0]\right\|_{\lr^2}^2\bigg)(\delta-\gamma)^{\frac{q-2}{q}},
\end{align*}
thus \eqref{C.2.} holds.

\section{Proof of Proposition~\ref{prop:example}}\label{pf:prop:example}

We first give estimate of the norm of $y_1^{i,\delta_i,\mathbf{u}[v_0]}(t)- x_1^{i,\delta_i,v_0}(t)$. Recall the SDE for $x_1^{i,\delta_i,v_0}(\cdot)$, we can write that
\begin{align*}
    x_1^{i,\delta_i,v_0}(t) =\ &  \xi_1^i+\int_0^t \bigg[\int_{\RR^{n_1}}\overline{g}^0_1 \left(y\right)z[v_0](s)(\md y) + \overline{g}^1_1 \left(x_1^{i,\delta_i,v_0}(s),u^{i,\delta_i}_1[v_0](s),x_0^{v_0}(s-\delta_i)\right) \bigg] \md s \\
    &+ \int_0^t \bigg[\int_{\RR^{n_1}}\overline{\sigma}^0_1 \left(y\right)z[v_0](s)(\md y) + \overline{\sigma}^1_1 \left(x_1^{i,\delta_i,v_0}(s),u^{i,\delta_i}_1[v_0](s),x_0^{v_0}(s-\delta_i)\right) \bigg] \md W_1^i (s)\\
    =\ &  \xi_1^i+\int_0^t \bigg[\int_{\RR^{n_1}}\overline{g}^0_1 \left(y\right)z[v_0](s)(\md y)  -\text{\footnotesize$\frac{1}{N-1} \sum_{j\neq i}$} \overline{g}^0_1 \left(x_1^{j,\Delta_j,v_0}(s)\right) \bigg] \md s \\
    &+\int_0^t \bigg[\text{\footnotesize$\frac{1}{N-1} \sum_{j\neq i}$} \overline{g}^0_1 \left(x_1^{j,\Delta_j,v_0}(s)\right) + \overline{g}^1_1 \left(x_1^{i,\delta_i,v_0}(s),u^{i,\delta_i}_1[v_0](s),x_0^{v_0}(s-\delta_i)\right)\bigg] \md s \\
    &+ \int_0^t \bigg[\int_{\RR^{n_1}}\overline{\sigma}^0_1 \left(y\right)z[v_0](s)(\md y)  -\text{\footnotesize$\frac{1}{N-1} \sum_{j\neq i}$} \overline{\sigma}^0_1 \left(x_1^{j,\Delta_j,v_0}(s)\right) \bigg] \md W_1^i (s)\\
    &+ \int_0^t \bigg[\text{\footnotesize$\frac{1}{N-1} \sum_{j\neq i}$} \overline{\sigma}^0_1 \left(x_1^{j,\Delta_j,v_0}(s)\right) + \overline{\sigma}^1_1 \left(x_1^{i,\delta_i,v_0}(s),u^{i,\delta_i}_1[v_0](s),x_0^{v_0}(s-\delta_i)\right) \bigg] \md W_1^i (s).
\end{align*}
Then, by using SDE \eqref{eq:y_1:4.3} for $y_1^{i,\delta_i,\mathbf{u}[v_0]}(t)$, we know that $y_1^{i,\delta_i,\mathbf{u}[v_0]}(t)- x_1^{i,\delta_i,v_0}(t)$ satisfies the following equation: 
\begin{align}
    &y_1^{i,\delta_i,\mathbf{u}[v_0]}(t)- x_1^{i,\delta_i,v_0}(t)\nonumber\\
    =\ & \frac{1}{N-1}\sum_{j\neq i} \int_0^t \left[ \overline{g}^0_1 \left(y_1^{j,\Delta_j,\mathbf{u}[v_0]}(s)\right) - \overline{g}^0_1 \left(x_1^{j,\Delta_j,v_0}(s)\right) \right] \md s \notag \\
    &+\int_0^t \left[\overline{g}^1_1\bigg(y_1^{i,\delta_i,\mathbf{u}[v_0]}(s), u_1^{i, \delta_i}[v_0](s), y_0^{\mathbf{u}[v_0]}\left(s-\delta_i\right)\bigg) - \overline{g}^1_1 \left(x_1^{i,\delta_i,v_0}(s),u^{i,\delta_i}_1[v_0](s),x_0^{v_0}(s-\delta_i)\right) \right]\md s \notag\\
    & - \int_0^t \bigg[\int_{\RR^{n_1}}\overline{g}^0_1 \left(y\right)z[v_0](s)(\md y)  -\text{\footnotesize$\frac{1}{N-1} \sum_{j\neq i}$} \overline{g}^0_1 \left(x_1^{j,\Delta_j,v_0}(s)\right) \bigg] \md s \notag\\
    &+\frac{1}{N-1}\sum_{j\neq i} \int_0^t \left[ \overline{\sigma}^0_1 \left(y_1^{j,\Delta_j,\mathbf{u}[v_0]}(s)\right) - \overline{\sigma}^0_1 \left(x_1^{j,\Delta_j,v_0}(s)\right) \right] \md W_1^i (s) \notag \\
    &+\int_0^t \left[\overline{\sigma}^1_1\bigg(y_1^{i,\delta_i,\mathbf{u}[v_0]}(s), u_1^{i, \delta_i}[v_0](s), y_0^{\mathbf{u}[v_0]}\left(s-\delta_i\right)\bigg) - \overline{\sigma}^1_1 \left(x_1^{i,\delta_i,v_0}(s),u^{i,\delta_i}_1[v_0](s),x_0^{v_0}(s-\delta_i)\right) \right]\md W_1^i (s) \notag\\
    & - \int_0^t \bigg[\int_{\RR^{n_1}}\overline{\sigma}^0_1 \left(y\right)z[v_0](s)(\md y)  -\text{\footnotesize$\frac{1}{N-1} \sum_{j\neq i}$} \overline{\sigma}^0_1 \left(x_1^{j,\Delta_j,v_0}(s)\right) \bigg] \md W_1^i (s). \label{SDE:y-x}
\end{align}
From the formulation of $z[v_0](\cdot)$ in \eqref{def:z'}, we see that 
\begin{align}
    &\int_{\RR^{n_1}}\overline{g}^0_1 \left(y\right)z[v_0](s)(\md y) - \text{\footnotesize$\frac{1}{N-1} \sum_{j\neq i} $}\overline{g}^0_1 \left(x_1^{j,\Delta_j,v_0}(s)\right) \notag \\
    =\ & \frac{1}{N-1} \sum_{j\neq i} \bigg[\int_\brd \overline{g}^0_1 \left(y\right)z[v_0](s)(\md y) - \overline{g}^0_1 \left(x_1^{j,\Delta_j,v_0}(s)\right)\bigg] \notag \\
    =\ & \frac{1}{N-1} \sum_{j\neq i} \bigg[\int_\brd \overline{g}^0_1 \left(y\right) \int_{[a,b]} \mathbb{P}^{\F^0_{s-\delta_j}\vee \F_s^z }_{ x_1^{j,\delta_j,v_0}(s)} \md\pi_{\de}(\delta_j) (\md y) - \overline{g}^0_1 \left(x_1^{j,\Delta_j,v_0}(s)\right)\bigg] \notag \\
    =\ & \frac{1}{N-1} \sum_{j\neq i} \bigg[\int_{[a,b]}\int_\brd \overline{g}^0_1 \left(y\right)  \mathbb{P}^{\F^0_{s-\delta_j}\vee \F_s^z }_{ x_1^{j,\delta_j,v_0}(s)} (\md y) \md\pi_{\de}(\delta_j) - \overline{g}^0_1 \left(x_1^{j,\Delta_j,v_0}(s)\right)\bigg] \notag \\
    =\ & \frac{1}{N-1} \sum_{j\neq i} \bigg[\int_{[a,b]} \e\left[ \overline{g}^0_1 \left(x_1^{j,\delta_j,v_0}(s)\right) \Big| \F^0_{s-\delta_j}\vee \F_s^z \right] \md\pi_{\de}(\delta_j)  - \overline{g}^0_1 \left(x_1^{j,\Delta_j,v_0}(s)\right)\bigg]. \notag
\end{align}
We temporarily denote by 
\begin{align}\label{def:eta}
    \eta^{g,j}_s:=\ & \int_{[a,b]} \e\left[ \overline{g}^0_1 \left(x_1^{j,\delta_j,v_0}(s)\right) \Big| \F^0_{s-\delta_j}\vee \F_s^z \right] d\pi_{\de}(\delta_j)  - \overline{g}^0_1 \left(s,x_1^{j,\Delta_j,v_0}(s)\right).
\end{align}
Then, it is obvious that
\begin{align}\label{eta^j}
    \e\left[\eta^{g,j}_s\right]=0,\quad s\in[0,T],\quad j=1,\dots,N.
\end{align}
Since $x_1^{j,\delta_j,v_0}(s)$ is $\F^0_{s-\delta_j}\vee \F_s^z\vee \F^{1,j}_s$-adapted, we can write
\begin{align}\label{add-8}
    \overline{g}^0_1 \left(x_1^{j,\delta_j,v_0}(s)\right)=
    \e\left[\overline{g}^0_1 \left(x_1^{j,\delta_j,v_0}(s)\right) \Big| \F^0_{s-\delta_j}\vee \F_s^z\vee \F^{1,j}_s \right] = \e\left[\overline{g}^0_1 \left(x_1^{j,\delta_j,v_0}(s)\right) \Big| \F^0_{s}\vee \F_s^z\vee \F^{1,j}_s \right],
\end{align}
then, similar to \eqref{eq:pro:conditional}, by using the fact that $\mathcal{F}^{1,j}$ and $\mathcal{F}^0\vee \mathcal{F}^z$ are independent of each other, and  by using \eqref{eq:property:X},
we have
\begin{align*}
    &\e\left[ \overline{g}^0_1 \left(x_1^{j,\delta_j,v_0}(s)\right) \Big| \F^0_{s-\delta_j}\vee \F_s^z \right]=\e\left[ \overline{g}^0_1 \left(x_1^{j,\delta_j,v_0}(s)\right) \Big| \F^0_s\vee \F_s^z \right].
\end{align*}
As a consequence, together with the independence of $\{\F^{1,j},\ j=1,\dots,N\}$ of different players and \eqref{eta^j},
\begin{align*}
    \e\left[\eta^{g,j}_s\cdot \eta^{g,j'}_s\right]=0,\quad j\neq j',\quad j=1,\dots,N;
\end{align*}
also see \cite[Pages 175-176]{Sznitman} for details about the similar approach. Then, we have
\begin{align}
    \e\left[\bigg|\frac{1}{N-1} \sum_{j\neq i} \eta^{g,j}_s\bigg|^2\right]=\ & \frac{1}{(N-1)^2} \e\left[\sum_{j\neq i} \left|\eta^{g,j}_s\right|^2+ 2 \sum_{j'\neq j\neq i}  \eta^{g,j}_s\cdot \eta^{g,j'}_s \right] \notag\\
    =\ & \frac{1}{(N-1)^2} \e\left[\sum_{j\neq i} \left|\eta^{g,j}_s\right|^2 \right] = \frac{1}{N-1} \e\left[\left|\eta^{g,i}_s\right|^2 \right]. \label{eta^j&eta^j'}
\end{align}
Similarly, we denote by 
\begin{align}\label{def:eta'}
    \eta^{\sigma,j}_s:=\ & \int_{[a,b]} \e\left[ \overline{\sigma}^0_1 \left(x_1^{j,\delta_j,v_0}(s)\right) \Big| \F^0_{s-\delta_j}\vee \F_s^z \right] d\pi_{\de}(\delta_j)  - \overline{\sigma}^0_1 \left(s,x_1^{j,\Delta_j,v_0}(s)\right),
\end{align}
then, the last term on the right hand side of \eqref{SDE:y-x} also writes
\begin{align*}
    \int_0^t \bigg[\int_{\RR^{n_1}}\overline{\sigma}^0_1 \left(y\right)z[v_0](s)(\md y)  -\frac{1}{N-1} \sum_{j\neq i} \overline{\sigma}^0_1 \left(x_1^{j,\Delta_j,v_0}(s)\right) \bigg] \md W_1^i (s) = \int_0^t  \frac{1}{N-1} \sum_{j\neq i} \eta^{\sigma,j}_s \md W_1^i (s);
\end{align*}
and $\left\{\eta^{\sigma,j}_s,\ j=1,\dots,N\right\}$ satisfies
\begin{align}
    \e\left[\bigg|\frac{1}{N-1} \sum_{j\neq i} \eta^{\sigma,j}_s\bigg|^2\right]=\ & \frac{1}{N-1} \e\left[\left|\eta^{\sigma,i}_s\right|^2 \right]. \label{eta'^j&eta^j'}
\end{align} 
Now, by substituting \eqref{def:eta} and \eqref{def:eta'} into \eqref{SDE:y-x}, we get
\begin{align}
    &y_1^{i,\delta_i,\mathbf{u}[v_0]}(t)- x_1^{i,\delta_i,v_0}(t) \notag \\
    =\ & \int_0^t \frac{1}{N-1}\sum_{j\neq i} \left[ \overline{g}^0_1 \left(y_1^{j,\Delta_j,\mathbf{u}[v_0]}(s)\right) - \overline{g}^0_1 \left(x_1^{j,\Delta_j,v_0}(s)\right) \right] \md s - \int_0^t \frac{1}{N-1} \sum_{j\neq i} \eta^{g,j}_s \md s \notag \\
    &+\int_0^t \left[\overline{g}^1_1\bigg(y_1^{i,\delta_i,\mathbf{u}[v_0]}(s), u_1^{i, \delta_i}[v_0](s), y_0^{\mathbf{u}[v_0]}\left(s-\delta_i\right)\bigg) - \overline{g}^1_1 \left(x_1^{i,\delta_i,v_0}(s),u^{i,\delta_i}_1[v_0](s),x_0^{v_0}(s-\delta_i)\right) \right]\md s \notag\\
    &+ \int_0^t \frac{1}{N-1}\sum_{j\neq i} \left[ \overline{\sigma}^0_1 \left(y_1^{j,\Delta_j,\mathbf{u}[v_0]}(s)\right) - \overline{\sigma}^0_1 \left(x_1^{j,\Delta_j,v_0}(s)\right) \right] \md W_1^i (s) - \int_0^t \frac{1}{N-1} \sum_{j\neq i} \eta^{\sigma,j}_s \md W_1^i (s) \notag \\
    &+\int_0^t \left[\overline{\sigma}^1_1\bigg(y_1^{i,\delta_i,\mathbf{u}[v_0]}(s), u_1^{i, \delta_i}[v_0](s), y_0^{\mathbf{u}[v_0]}\left(s-\delta_i\right)\bigg) - \overline{\sigma}^1_1 \left(x_1^{i,\delta_i,v_0}(s),u^{i,\delta_i}_1[v_0](s),x_0^{v_0}(s-\delta_i)\right) \right]\md W_1^i (s) . \label{SDE:y-x_1}
\end{align}
From Assumption~\ref{assumption:exp}, applying standard arguments for SDEs, we have the following estimate for SDE \eqref{SDE:y-x_1}:
\begin{align}
    &\e\left[\sup_{t\in[0,T]} \left|y_1^{i,\delta_i,\mathbf{u}[v_0]}(t)- x_1^{i,\delta_i,v_0}(t)\right|^2 \right] \notag \\
    \le\ &  C(L,T)\e\left[\int_0^T \bigg(\frac{1}{N-1}\sum_{j\neq i} \left|y_1^{j,\Delta_j,\mathbf{u}[v_0]}(t)- x_1^{j,\Delta_j,v_0}(t)\right|\bigg)^2 \md t + \int_0^T \left|y_0^{\mathbf{u}[v_0]}(t-\delta_i)-x_0^{v_0}(t-\delta_i)\right|^2 \md t \right] \notag \\
    &+C(L,T)\int_0^T \e\left[\bigg|\frac{1}{N-1} \sum_{j\neq i} \eta^{g,j}_t\bigg|^2  + \bigg|\frac{1}{N-1} \sum_{j\neq i} \eta^{\sigma,j}_t\bigg|^2 \right] \md t. \label{exp:thm_1}
\end{align}
Using the fact that $\left\{y_1^{j,\Delta_j,\mathbf{u}[v_0]}(t)- x_1^{j,\Delta_j,v_0}(t),\ j=1,\dots,N \right\}$ are identically distributed, by symmetry or following a similar approach as that leading to \eqref{eq:w2:conv1}, we have
\begin{align*}
    &\e\left[\int_0^T \bigg(\frac{1}{N-1}\sum_{j\neq i} \left|y_1^{j,\Delta_j,\mathbf{u}[v_0]}(t)- x_1^{j,\Delta_j,v_0}(t)\right|\bigg)^2 \md t\right]
    \le \int_0^T\sup_{\delta\in[a,b]}\e\left[\left|y^{i,\delta,\mathbf{u}[v_0]}_1(t)- x^{i,\delta,v_0}_1(t)\right|^2 \right] \md t.
\end{align*}
Substituting this last inequality back into \eqref{exp:thm_1} and taking the supremum in $\delta\in[a,b]$, we can write 
\begin{align*}
    \sup_{\delta\in[a,b]}\e\left[\sup_{t\in[0,T]} \left|y^{i,\delta,\mathbf{u}[v_0]}_1(t)- x^{i,\delta,v_0}_1(t)\right|^2 \right]\le\ &  C(L,T)\int_0^T\sup_{\delta\in[a,b]}\e\left[\left|y^{i,\delta,\mathbf{u}[v_0]}_1(t)- x^{i,\delta,v_0}_1(t)\right|^2 \right] \md t \notag \\
    &+C(L,T)\sup_{\delta\in[a,b]}\e\left[ \int_0^T \left|y_0^{\mathbf{u}[v_0]}(t-\delta)-x_0^{v_0}(t-\delta)\right|^2 \md t \right] \notag \\
    &+C(L,T)\int_0^T \e\left[\bigg|\frac{1}{N-1} \sum_{j\neq i} \eta^{g,j}_t\bigg|^2 + \bigg|\frac{1}{N-1} \sum_{j\neq i} \eta^{\sigma,j}_t\bigg|^2 \right] \md t.
\end{align*}
Then, by using Gr\"onwall's inequality, we have
\begin{align*}
    \sup_{\delta\in[a,b]}\e\left[\sup_{t\in[0,T]} \left|y^{i,\delta,\mathbf{u}[v_0]}_1(t)- x^{i,\delta,v_0}_1(t)\right|^2 \right]\le\ &C(L,T) \e\left[ \int_{-b}^{T-a} \left|y_0^{\mathbf{u}[v_0]}(t)-x_0^{v_0}(t)\right|^2 \md t \right] \notag \\
    &+C(L,T)\int_0^T \e\left[\bigg|\frac{1}{N-1} \sum_{j\neq i} \eta^{g,j}_t\bigg|^2 + \bigg|\frac{1}{N-1} \sum_{j\neq i} \eta^{\sigma,j}_t\bigg|^2\right] \md t.
\end{align*}
Substituting \eqref{eta^j&eta^j'} and \eqref{eta'^j&eta^j'}  into the last inequality, we have
\begin{align}
    \sup_{\delta\in[a,b]}\e\left[\sup_{t\in[0,T]} \left|y^{i,\delta,\mathbf{u}[v_0]}_1(t)- x_1^{i,\delta_i,v_0}(t)\right|^2 \right]\le\ &  C(L,T)\e\left[ \int_{-b}^{T-a} \left|y_0^{\mathbf{u}[v_0]}(t)-x_0^{v_0}(t)\right|^2 \md t \right] \notag \\
    &+\frac{C(L,T)}{N-1}\int_0^T  \e\left[\left|\eta^{g,i}_t\right|^2  + \left|\eta^{\sigma,i}_t\right|^2 \right] \md t; \label{exp:estimate_1}
\end{align}
Now we give the boundedness of the $\mathcal{L}^2$-norm of $\eta^{g,i}_s$. From \eqref{def:eta} and Assumption~\ref{assumption:exp}, 
\begin{align*}
    \e\left[\left|\eta^{g,i}_s\right|^2 \right] =\ & 4 \e\left[\left|\overline{g}^0_1 \left(x_1^{i,\de_i,v_0}(s)\right)\right|^2 \right] \notag \\
    \le\ & 8L^2 \left(1+ \e\left[\left| x_1^{i,\de_i,v_0}(s)\right|^2 \right]\right) \\
    \le\ & 8 L^2 \left(1+ \sup_{\delta\in[a,b]} \e\left[\sup_{s\in[0,T]} \left|x_1^{i,\delta,v_0}(s)\right|^2 \right]\right);
\end{align*} 
and similarly, 
\begin{align*}
    \e\left[\left|\eta^{\sigma,i}_s\right|^2 \right] 
    \le\ & 8 L^2 \left(1+ \sup_{\delta\in[a,b]} \e\left[\sup_{s\in[0,T]} \left|x_1^{i,\delta,v_0}(s)\right|^2 \right]\right).
\end{align*} 
Substituting the last two estimates back into \eqref{exp:estimate_1}, we have
\begin{align}
    \sup_{\delta\in[a,b]}\e\left[\sup_{t\in[0,T]} \left|y^{i,\delta,\mathbf{u}[v_0]}_1(t)- x_1^{i,\delta_i,v_0}(t)\right|^2 \right]\le\ &  C(L,T)\e\left[ \int_{-b}^{T-a} \left|y_0^{\mathbf{u}[v_0]}(t)-x_0^{v_0}(t)\right|^2 \md t \right] \notag \\
    &+\frac{C(L,T)}{N-1}\left(1+ \sup_{\delta\in[a,b]} \e\left[\sup_{s\in[0,T]} \left|x_1^{i,\delta,v_0}(s)\right|^2 \right]\right). \label{exp:estimate_2}
\end{align}
Applying a similar approach to the process $y_0^{\mathbf{u}[v_0]}(\cdot)-x_0^{v_0}(\cdot)$ and the Gr\"onwall inequality, we can also obtain that
\begin{align}
    \e\left[\sup_{t\in[-b,T]} \left|y_0^{\mathbf{u}[v_0]}(t)- x_0^{v_0}(t)\right|^2 \right]\le\ &  \frac{C(L,T)}{N}\left(1+  \sup_{\delta\in[a,b]} \e\left[\sup_{s\in[0,T]} \left|x_1^{i,\delta,v_0}(s)\right|^2 \right]\right). \label{exp:estimate_3}
\end{align}
Combining \eqref{exp:estimate_2} and \eqref{exp:estimate_3} and using \eqref{C.1.hat.x_1} in Lemma \ref{prop:moment}, we have
\begin{align*}
    &\e\left[\sup_{t\in[-b,T]} \left|y_0^{\mathbf{u}[v_0]}(t)- x_0^{v_0}(t)\right|^2 \right]+\sup_{\delta\in[a,b]}\e\left[\sup_{t\in[0,T]} \left|y^{i,\delta,\mathbf{u}[v_0]}_1(t)- x_1^{i,\delta_i,v_0}(t)\right|^2 \right]\\
    \le\ & \frac{C(L,T)}{N-1}\left(1+ \sup_{\delta\in[a,b]} \e\left[\sup_{s\in[0,T]} \left|x_1^{i,\delta,v_0}(s)\right|^2 \right]\right)\\
    \le\ & \frac{C(n_1,L,T,\xi_0,\xi_1)}{N-1} \bigg(1+\|v_0\|_{\mathcal{L}^2}^2+\sup_{\delta\in[a,b]}\left\|u_1^{i,\delta}[v_0]\right\|_{\mathcal{L}^2}^2\bigg),
\end{align*}
from which we obtain \eqref{exp:thm_0}.

\footnotesize

\bibliographystyle{acm} 

\bibliography{reference}

\normalsize

\renewcommand{\theequation}{\thesection.\arabic{equation}}

\end{document}